\documentclass[12pt,oneside]{amsart}
\usepackage{mathpazo}

\usepackage{latexsym}
\usepackage{amsmath,amssymb,amsthm,mathrsfs,amsfonts,dsfont, mathtools, setspace} 
\usepackage[usenames,dvipsnames]{xcolor}
\usepackage{hyperref}
\usepackage{stmaryrd}
\usepackage{caption}
\usepackage{tikz}
\usepackage{tikz-cd}
\usepackage{bm}
\usepackage{verbatim}

\usetikzlibrary{matrix,arrows,backgrounds,shapes.misc,shapes.geometric,fit}
\usetikzlibrary{arrows,decorations.pathmorphing,backgrounds,positioning,fit,petri,calc,shapes.misc,decorations.markings}
\usepackage{adjustbox}
\hypersetup{
    unicode=false,          
    pdftoolbar=true,        
    pdfmenubar=true,        
    pdffitwindow=false,     
    pdfstartview={FitH},    
    pdftitle={My title},    
    pdfauthor={Author},     
    pdfsubject={Subject},   
    pdfcreator={Creator},   
    pdfproducer={Producer}, 
    pdfkeywords={keyword1} {key2} {key3}, 
    pdfnewwindow=true,      
    colorlinks=true,       
    linkcolor=Sepia,        
    citecolor=Red,        
    filecolor=magenta,      
    urlcolor=cyan           
}

\let\oldtocsection=\tocsection

\let\oldtocsubsection=\tocsubsection

\let\oldtocsubsubsection=\tocsubsubsection

\renewcommand{\tocsection}[2]{\hspace{0em}\oldtocsection{#1}{#2}}
\renewcommand{\tocsubsection}[2]{\hspace{2em}\oldtocsubsection{#1}{#2}}
\renewcommand{\tocsubsubsection}[2]{\hspace{4em}\oldtocsubsubsection{#1}{#2}}

\newtheorem{thm}{Theorem}[subsection]
\newtheorem{prop}[thm]{Proposition}
\newtheorem{lemma}[thm]{Lemma}
\newtheorem{cor}[thm]{Corollary}
\newtheorem{remark}[thm]{Remark}
\newtheorem{defn}[thm]{Definition}

\newtheorem{ass}[thm]{Assumption}

\newtheorem*{thm*}{Theorem}

\numberwithin{equation}{subsection}

\makeatletter
\def\imod#1{\allowbreak\mkern5mu{\operator@font mod}\,\,#1}
\makeatother

\def\makeop#1{\expandafter\def\csname#1\endcsname
  {\mathop{\rm #1}\nolimits}\ignorespaces}
\makeop{Hom}   \makeop{End}   \makeop{Aut}   \makeop{Isom}  \makeop{Pic}
\makeop{Gal}   \makeop{ord}   \makeop{Char}  \makeop{Div}   \makeop{Lie}
\makeop{PGL}   \makeop{Corr}  \makeop{PSL}   \makeop{sgn}   \makeop{Spf}
\makeop{Spec}  \makeop{Map}    \makeop{Nm}    \makeop{Fr}    \makeop{disc}
\makeop{Proj}  \makeop{supp}  \makeop{ker}   \makeop{im}    \makeop{dom}
\makeop{coker} \makeop{Stab}  \makeop{SO}    \makeop{SL}    \makeop{O}
\makeop{Cl}    \makeop{cond}  \makeop{Br}    \makeop{inv}   \makeop{rank}
\makeop{id}    \makeop{Fil}   \makeop{Frac}  \makeop{GL}    \makeop{SU}
\makeop{Trd}   \makeop{Sp}    \makeop{Tr}    \makeop{Trd}   \makeop{diag}
\makeop{Res}   \makeop{ind}   \makeop{depth} \makeop{Tr}    \makeop{st}
\makeop{Ad}    \makeop{Int}   \makeop{tr}    \makeop{Sym}   \makeop{can}
\makeop{length}\makeop{SO}    \makeop{torsion} \makeop{GSp} \makeop{Ker}
\makeop{Adm}   \makeop{Frob}  \makeop{id}    \makeop{Tor}   \makeop{Ind}
\makeop{CoInd} \makeop{Inf}   \makeop{Cor}   \makeop{CoInf} \makeop{coLie}
\makeop{Ann} \makeop{rec} \makeop{Image} \makeop{Rep} \makeop{Inn} \makeop{Out}
\makeop{Def} \makeop{B} \makeop{gr}
\makeop{Art}
\makeop{codim}

\newcommand{\Z}{\mathbb Z}
\newcommand{\Q}{\mathbb Q}

\newcommand{\F}{\mathbb F}

\newcommand{\m}{\mathfrak m}

\def\bal#1\nal{\begin{align*}#1\end{align*}}
\def\lbal#1\lnal{\begin{flalign*}#1\end{flalign*}}

\newcommand{\vp}{\varphi}

\newcommand{\goto}{\rightarrow}

\newcommand{\M}{\mathfrak{M}}

\newcommand{\G}{\mathcal{G}}
\newcommand{\Ieta}{\text{\normalfont I}_\eta}
\newcommand{\Ietaa}{\text{\normalfont I}_{\eta'}}
\newcommand{\II}{\text{\normalfont II}}

\newcounter{sarrow}

\allowdisplaybreaks

\begin{document}

\title[Components of the Emerton-Gee stack]{Smoothness of components of the Emerton-Gee stack for \texorpdfstring{$\text{GL}_2$}{}}

\author{Anthony Guzman}

\author{Kalyani Kansal}

\author{Iason Kountouridis}

\author{Ben Savoie}

\author{Xiyuan Wang}

\address[Anthony Guzman]{Department of Mathematics, The University of Arizona, Tucson, AZ 85721, USA}
\email{awguzman@math.arizona.edu}

\address[Kalyani Kansal]{Department of Mathematics, Johns Hopkins University, Baltimore, MD 21218, USA}
\email{kkansal2@jhu.edu}

\address[Iason Kountouridis]{Department of Mathematics, The University of Chicago, Chicago, IL 60637, USA}
\email{iasonk@math.uchicago.edu}

\address[Ben Savoie]{Department of Mathematics, Rice University, Houston, TX 77005, USA}
\email{Bs83@rice.edu}

\address[Xiyuan Wang]{Department of Mathematics, The Ohio State University, Columbus, OH 43210, USA}
\email{wang.15476@osu.edu}

\unskip

\maketitle
\vspace*{-.5\baselineskip}
\begin{abstract}\vspace*{-1.1\baselineskip}
Let $K$ be a finite unramified extension of $\mathbb{Q}_p$, where $p>2$. \cite{CEGS-local-geometry} and \cite{EG20} construct a moduli stack of two dimensional mod $p$ representations of the absolute Galois group of $K$. We show that most irreducible components of this stack (including several non-generic components) are isomorphic to quotients of smooth affine schemes. We also use this quotient presentation to compute global sections on these components.\vspace*{-\baselineskip}
 \unskip

\end{abstract}\mbox{}\unskip

\unskip

\vspace*{-3.5\baselineskip}
\unskip

\tableofcontents\unskip \vspace{-\baselineskip}
\clearpage
\unskip

\section{Introduction}\label{intro sec}

Let $K/\mathbb{Q}_p$ be a finite extension. Following \cite[\S~3]{CEGS-local-geometry}, define the stack $\mathcal{R}$ to be the moduli stack of \'etale $\varphi$-modules of rank two. By \cite{Fon91}, there is an equivalence of categories between \'etale $\varphi$-modules and $p$-adic Galois representations of $G_\infty\coloneqq\Gal(\overline{K}/K_\infty)$ allowing us to view $\mathcal{R}$ as a moduli stack of said representations. 

The theory of Breuil-Kisin modules developed in \cite{Kis06} gives us a way to associate to any Breuil-Kisin module an \'{e}tale $\varphi$-module. Indeed, by denoting the  moduli stack of rank two Breuil-Kisin modules of finite height $h$ by $\mathcal{C}_h$, then there is a morphism $\mathcal{C}_h \rightarrow \mathcal{R}$ given (topologically) by 
$$\mathfrak{M}\longmapsto\mathfrak{M}[1/u],$$
where $u$ is a formal variable. The stack of Breuil-Kisin modules of height 1, denoted $\mathcal{C}_1$, admits a scheme-theoretic image $\mathcal{Z}_1$ via the above morphism. Breuil-Kisin modules of height 1 correspond to \'etale $\varphi$-modules which admit natural extensions to representations of $G_K=\Gal(\overline{K}/K)$ so the substack $\mathcal{Z}_1$ may be viewed as a moduli stack of representations of $G_K$. 

One method to study such objects is to introduce Breuil-Kisin modules with descent datum. This approach allows the consideration of those Galois representations that arise from generic fibers of finite flat group schemes after restriction to a finite tamely ramified extension of $K$. In fact, all representations except tr\`es ramifi\'ee ones arise in this way. Here, the tr\`es ramifi\'ee representations are the twists of certain extensions of the trivial character by the mod $p$ cyclotomic character.

Let $K'$ be a tamely ramified extension of $K$. Endowing our Breuil-Kisin modules and \'etale $\varphi$-modules with descent datum from $K'$ to $K$ allows us to define the morphism $\mathcal{C}_1^{dd}\rightarrow\mathcal{R}^{dd}$. To focus on those non tr\`es ramifi\'ee representations, we enforce a Barsotti-Tate condition on points in the scheme-theoretic image. The Barsotti-Tate condition on representations with coefficients in a characteristic $p$ field turns out to correspond to a strong determinant condition on Breuil-Kisin modules, see \cite[Theorem~5.1.2]{CEGS-local-geometry}. Let $\mathcal{C}^{dd,\mathrm{BT}}_1$ denote the stack of Breuil-Kisin modules of height 1 with descent data that satisfy the strong determinant condition. With this, we attain a morphism $\mathcal{C}^{dd,\mathrm{BT}}_1\rightarrow\mathcal{R}^{dd}$ whose scheme-theoretic image is denoted $\mathcal{Z}^{dd}_1$, the stack of non tr\`es ramifi\'ee $G_K$-representations. We will take these stacks to be defined over $\mathbb{F}$, a finite field extension of $\mathbb{F}_p$. To reduce notation, we will suppress the decorations $dd$ and $1$ in the symbols for our stacks with the assumption that all of our objects have descent data and correspond to height 1 Breuil-Kisin modules. Both $\mathcal{C}^{\mathrm{BT}}$ and $\mathcal{Z}$ are algebraic stacks of finite presentation. The stack $\mathcal{C}^{\mathrm{BT}}$ admits a decomposition
\[\prod_{\tau}\mathcal{C}^{\tau,\mathrm{BT}}\]
over tame inertial types $\tau:I_K\rightarrow\GL_2(\overline{\Q}_p)$ where we interpret each substack $\mathcal{C}^{\tau,\mathrm{BT}}$ as consisting of Breuil-Kisin modules whose descent data is of type $\tau$. For each such $\mathcal{C}^{\tau,\mathrm{BT}}$, we write $\mathcal{Z}^\tau$ to be the scheme-theoretic image of $\mathcal{C}^{\tau,\mathrm{BT}}$ inside $\mathcal{Z}$.

Such stacks are of increasing interest in the study of $p$-adic Galois representations. Related to the stack constructed by \cite{CEGS-components} 
 is a stack of so-called \'{e}tale $(\varphi,\Gamma)$-modules of rank $2$, constructed and studied in \cite{EG-Scheme-Image}. The Emerton-Gee stack for $\GL_2$, as it is called, is denoted by $\mathcal{X}_2$ and is defined over the formal spectrum of a discrete valuation ring with residue field $\F$. The stack $\mathcal{X}_2$ has proven fruitful in the study of Galois representations, and in particular, the rational lifts of $\overline{\rho}:G_K\rightarrow\GL_2(\F)$. Moreover, it is conjectured to play a role in a categorical $p$-adic Langlands correspondence, as explained in \cite{egh-ihes}. While theoretically important, $\mathcal{X}_2$ is difficult to work with by hand. On the contrary, the stack $\mathcal{Z}$ is much easier to work with. Based on their constructions, one should suspect that there is a connection between $\mathcal{Z}$ and the reduced part of $\mathcal{X}_2$, denoted $\mathcal{X}_{2, \text{red}}$. Indeed, for each tame inertial type $\tau$, \cite[Thm.~1.4]{Irreg_Loci} demonstrates an isomorphism between $\mathcal{Z}^{\tau}$ and a closed substack of $\mathcal{X}_{2, \text{red}}$. As $\tau$ varies, the isomorphism induces isomorphisms between the irreducible components of $\mathcal{Z}$ and $\mathcal{X}_{2, \text{red}}$ identifying the irreducible component of $\mathcal{Z}$ labelled by a Serre weight $\sigma$ with the irreducible component of $\mathcal{X}_{2, \text{red}}$ labelled by $\sigma$. Hence, our main result which is technically a statement about the irreducible components of $\mathcal{Z}$ can be interpreted as a result on irreducible components of $\mathcal{X}_{2, \text{red}}$.
 

\subsection{Main result}
By \cite[\S~5]{CEGS-local-geometry}, the irreducible components of $\mathcal{Z}$ can be described in terms of non-Steinberg Serre weights. Indeed, for such a Serre weight $\sigma$, the associated irreducible component $\mathcal{Z}(\sigma)$ is such that the $\overline{\mathbb{F}}_p$-points of $\mathcal{Z}(\sigma)$ are those representations having $\sigma$ as a Serre weight. Such irreducible components serve as the main objects of study in this paper. In particular, our main result is a smooth presentation of the irreducible component $\mathcal{Z}(\sigma)$ and a computation of its global functions.
\begin{thm*}[Theorem~ \ref{main-thm}]
Let $p>2$. Let $K$ be an unramified extension of $\Q_{p}$ of degree $f$ with residue field $k$. Let $\mathcal{Z}(\sigma)$ be the irreducible component of $\mathcal{Z}$ indexed by a non-Steinberg Serre weight $\sigma = \sigma_{\vec{a}, \vec{b}} = \bigotimes_{i=0}^{f-1} ({\det}^{a_{i}} \Sym^{b_i} k^2 ) \otimes_{k, \kappa_{i}} \F$, where $\{\kappa_{i}\}_{i=0}^{f-1}$ is the set of the distinct embeddings of $k$ into $\F$, and $\kappa_{i+1}^p = \kappa_i$. Suppose $\sigma$ satisfies the following properties:
 \begin{enumerate}
    \item $\vec{b} \neq (0, 0, \dots, 0)$.
    \item $\vec{b} \neq (p-2, p-2, \dots, p-2)$.
    \item Extend the indices of $b_i$'s to all of $\mathbb{Z}$ by setting $b_{i+f} = b_i$. Then $(b_i)_{i \in \mathbb{Z}}$ does not contain a contiguous subsequence of the form $(0, p-2, \dots, p-2, p-1)$ of length $\geq 2$.
\end{enumerate}

Then $\mathcal{Z}(\sigma)$ is smooth and isomorphic to a quotient of $\GL_2 \times \SL_2^{f-1}$ by $\mathbb{G}_m^{f+1} \times \mathbb{G}_a^{f}$. The ring of global functions of $\mathcal{Z}(\sigma)$ is isomorphic to $\F[x, y][\frac{1}{y}]$.
\end{thm*}

Note that the case $\vec{b} = (0, 0, \dots, 0)$ can be studied using Fontaine-Lafaille methods. The key utility of this paper is in providing a description of components indexed by non-Fontaine-Lafaille Serre weights.

\subsection{Outline of the article}\label{outline subsec}

We begin in Section \ref{BK sec}, by providing a concrete classification of certain Breuil-Kisin modules defined over Artinian local $\F$-algebras with descent data classified by a tame inertial $\F$-type $\tau$. This closely follows the ideas of \cite{CDM} which describe a method to classify such Breuil-Kisin modules into one of three forms which are convenient for computations. We then describe the automorphisms of such Breuil-Kisin modules that are needed to study the stack $\mathcal{C}^{\mathrm{BT}}$. We find that this method is not foolproof however, and identify \textit{obstruction} conditions on $\tau$ for which our methods fail.

In Section \ref{component sec}, for a given tame intertial $\F$-type $\tau$, we identify an irreducible component $\mathcal{X}(\tau)$ of $\mathcal{C}^{\mathrm{BT}}$ that can be written as a quotient of smooth affine schemes. This utilizes the classification from Section \ref{BK sec}. We then use this quotient presentation to compute the global functions on $\mathcal{X}(\tau)$.


Finally, Section \ref{passage sec} is where we analyze the relationship between $\mathcal{X}(\tau)$'s and the irreducible components of $\mathcal{Z}$. We first explicitly identify the irreducible component of $\mathcal{Z}$ that turns out to be the image of $\mathcal{X}(\tau) \subset \mathcal{C}^{\mathrm{BT}}$. Subject to the aforementioned obstruction conditions, we allow $\tau$ to vary in order to obtain as many irreducible components of $\mathcal{Z}$ as possible in the images of $\mathcal{X}(\tau)$'s. We then show that the map from $\mathcal{X}(\tau)$ to $\mathcal{Z}$ is fully faithful, and conclude finally that $\mathcal{X}(\tau)$ is isomorphic to its scheme-theoretic image. 


\subsection{Acknowledgements} This work was born out of the NSF-FRG Collaborative Grant DMS-1952556. We would like to thank Brandon Levin for suggesting this problem and guiding us throughout our research. We would also like to thank Agn\`es David, Matthew Emerton, Bao V. Le Hung and David Savitt for helpful conversations, and the anonymous referee for numerous suggestions that improved both accuracy and exposition. 

\subsection{Notation and conventions}\label{notation subsec}

Fix a prime $p>2$. Let $K$ be a finite, unramified extension of $\mathbb{Q}_p$, with ring of integers $\mathcal{O}_K$ and residue field $k$. Let $f := f(K/\mathbb{Q}_p)$.
Upon fixing an algebraic closure $\overline{K}$, we let $G_K=\Gal(\overline{K}/K)$ denote the absolute Galois group of $K$ and define $I_K$ and $I_K^{\mathrm{w}}$ to be the inertia and wild inertia subgroups of $G_K$ respectively. Fix a uniformizer $\pi$ of $K$ and a $p$-power compatible sequence $(\pi_n)_{n\ge0}$ so that $\pi_0=\pi$ and $\pi_{n+1}^p=\pi_n$. We define the field $K_\infty$ to be the compositum $K_\infty=\bigcup_n K(\pi_n)$ with associated Galois group $G_\infty=\Gal(\overline{K}/K_\infty)\subset G_K$.

Serving to capture our descent data, let $K'/K$ be a totally tamely ramified extension of degree $e\coloneqq p^f-1$ obtained by adjoining an $e$-th root of $p$ to $K$ which we denote $\pi_{K'}$. Let $k'=k$ be the residue field of $K'$.


$\Gal(K'/K)$ is cyclic of order $e$, and is isomorphic to $\mu_e(K)$, the group of $e$th roots of unity in $W(k') = \mathcal{O}_K$. This isomorphism is given by $h: \Gal(K'/K)\rightarrow  \mu_{e}(K)$, defined by $
h(g) = \tfrac{g(\pi_{K'})}{\pi_{K'}}$.

Serving as coefficients, let $\F$ be a finite extension of $\F_p$ which we assume to be large enough such that $\F$ contains all embeddings of $k$ into $\overline{\F}_p$. We fix an embedding $\kappa_{0}: k \hookrightarrow \F$, and recursively define $\kappa_{i}: k\hookrightarrow \F$ for any $i\in \Z$ such that $\kappa_{i+1}^{p}=\kappa_{i}$ and so $\kappa_{i+f}=\kappa_{i}$ for any $i\in \Z$. Since there are $f$ such embeddings, we will commonly take the index to be $i\in\Z/f\Z$.

\subsubsection{Serre Weights}

Recall $k/\F_{p}$ is a degree $f$ extension. A \textit{Serre weight} is an isomorphism class of irreducible $\F$-representations of $\GL_{2}(k)$. Any such representation is, up to isomorphism, of the form 
\[
\sigma_{\vec{a}, \vec{b}}:=\bigotimes_{i=0}^{f-1} ({\det}^{a_{i}} \Sym^{b_i} k^2 ) \otimes_{k, \kappa_{i}} \F,
\]
where $0\leq a_{i}, b_{i}\leq p-1$ and not all $a_{i}$ equal to $p-1$. We say $\sigma_{\vec{a}, \vec{b}}$ is \textit{Steinberg} if each $b_i$ equals $p-1$.

\subsubsection{Tame Inertial $\F$-Types}

An \textit{inertial $\F$-type} is (the isomorphism
class of) a representation $\tau : I_{K}\rightarrow \GL_2(\F)$ with open kernel which extends to $G_{K}$. An inertial $\F$-type is called \textit{tame} if $\tau|_{I_{K}^{\mathrm{w}}}$ is trivial.


Let $\tau : I_{K}\rightarrow \GL_2(\F)$ be a tame inertial $\F$-type. Then $\tau\cong \eta \oplus \eta^{'}$, where $\eta, \eta' : I_{K}\rightarrow \F^{\times}$ are tamely ramified characters. We say $\tau$ is a \textit{tame principal series $\F$-type} if $\eta, \eta'$ both extend to characters of $G_{K}$.
We will assume that the tame inertial $\F$-types $\tau$ factor via $I(K'/K) = \Gal(K'/K)$.

Given a tame principal series $\F$-type $\tau \cong \eta \oplus \eta'$, let $\gamma_i$ be the unique integer in $[0, e)$ such that $\eta \eta'^{-1} (g)= \kappa_i \circ h(g)^{\gamma_i}$. We also define $z_j \in \{0,\dots, p-1\}$ for $j \in \mathbb{Z}/f\mathbb{Z}$ so as to satisfy
\begin{align}\label{defn-z_i}
    \gamma_i= \sum\limits_{j=0}^{f-1} z_{i-j}p^j.
\end{align}

Note that this implies $p \gamma_i = z_{i+1} e + \gamma_{i+1}$. At times, we will assume the indexing set for ${z_i}$ to be $\mathbb{Z}$ via the natural quotient map $\mathbb{Z} \to \mathbb{Z}/f\mathbb{Z}$.


\section{Breuil-Kisin modules with descent}\label{BK sec}

In this section, we introduce some basic definitions and properties about Breuil-Kisin modules and their moduli space. To begin, we introduce the relevant notions which we will use for the rest of the paper. While the general definition of Breuil-Kisin modules works over characteristic $0$, we will restrict to characteristic $p$ and the specific case where descent data is from $K'$ to $K$.

Define $\mathfrak{S}_{\F_p}:= k[\![u]\!]$. The ring $\mathfrak{S}_{\F_p}$ is equipped with a Frobenius endomorphism $\varphi$ such that $u\mapsto u^p$ which is semilinear with respect to the (arithmetic) Frobenius on $k$. The Galois group $\Gal(K'/K)$ acts on $\mathfrak{S}_{\F_p}$ via $g(\sum a_i u^{i})=\sum a_{i} h(g)^{i} u^i$,
where $g\in \Gal(K'/ K)$, $\sum a_i u^i \in \mathfrak{S}$ and $h:\Gal(K'/K)\rightarrow k^{\times}$ is the map given by $h(g) = \tfrac{g(\pi_{K'})}{\pi_{K'}}$ mod $\pi_{K'}$. For later convenience, we also let $v=u^e$.


Let $R$ be an $\F$-algebra. We let $\mathfrak{S}_R:= (k\otimes_{\mathbb{F}_p} R)[\![u]\!]$ be the extension of scalars equipped with $R$-linear actions of $\varphi$ and $\Gal(K'/K)$ naturally extended from the $\varphi$ and $\Gal(K'/K)$ actions on $\mathfrak{S}_{\F_p}$. Let $\mathfrak e_i$ denote the idempotent of 
$$k' \otimes_{\F_p} R \cong \prod\limits_{\kappa_i:k'\hookrightarrow \F} R$$ corresponding to $(0,\dots,0,1,0,\dots,0)$ with $1$ in the $i$th coordinate. We can then write
\[
\mathfrak{S}_{R}\cong \bigoplus_{i\in \mathbb{Z}/f\mathbb{Z}}R[\![u]\!]\mathfrak{e}_{i}.
\]
Note that $\varphi(\mathfrak{e}_i) = \mathfrak{e}_{i+1}$, and if $g \in \Gal(K'/K)$, then $g(\mathfrak{e}_i) = \mathfrak{e}_i$.

\begin{defn}\label{BK-Def}
    A Breuil-Kisin module with $R$-coefficients and descent data from $K'$ to $K$ is a triple $(\M, \varphi_{\M}, \{\widehat{g}\}_{g\in \Gal(K'/K)})$, consisting of a finitely generated projective $\mathfrak{S}_{R}$-module $\M$ such that
    \begin{itemize}
        \item $\M$ admits a $\varphi$-semilinear map $ \varphi_{\M}: \M\rightarrow\M$ such that the induced map $\Phi_{\M} = 1\otimes \varphi_{\M}: \mathfrak{S}_{R}\otimes_{\varphi, \mathfrak{S}_{R}}\M\rightarrow \M$ is an isomorphism after inverting $v$.
        
        \item $\M$ admits an additive bijection $\widehat{g}: \M\rightarrow \M$ for each $g\in \Gal(K'/K)$ which commutes with $\varphi_{\M}$, respects the group structure $\widehat{g_{1}\circ g_{2}}=\widehat{g_1}\circ \widehat{g_2}$, and satisfies $\widehat{g}(sm) = g(s) \widehat{g}(m)$ for all $s \in \mathfrak{S}_{R}$ and $m \in \M$.
    \end{itemize}
\end{defn}
    We say $\M$ has \textit{height} at most $h\ge0$ if the cokernel of $\Phi_{\M}$ is killed by $v^{h}$. We say a Breuil-Kisin module $\M$ is rank $d$ if the underlying $\mathfrak{S}_{R}$-module has constant rank $d$. A morphism of Breuil-Kisin modules with $R$-coefficients and descent data is a morphism of $\mathfrak{S}_{R}$-modules that commutes with the $\varphi$-action and the Galois action. 
    
    Localizing a Breuil-Kisin module $\M$ by inverting $u$ gives rise to an \'etale-$\varphi$ module which we define below.

\begin{defn}
        An \'etale-$\varphi$ module with $R$-coefficients and descent data from $K'$ to $K$ is a triple $(M, \varphi_{M}, \{\widehat{g}\}_{g\in \Gal(K'/K)})$, consisting of a finitely generated projective $\mathfrak{S}_{R}[1/u]$-module $M$ such that
    \begin{itemize}
        \item $M$ admits a $\varphi$-semilinear map $ \varphi_{M}: M \rightarrow M$ such that the induced map $\Phi_{M} = 1\otimes \varphi_{M}: \mathfrak{S}_{R}[1/u]\otimes_{\varphi, \mathfrak{S}_{R}[1/u]}M\rightarrow M$ is an isomorphism.
        
        \item $M$ admits an additive bijection $\widehat{g}: M\rightarrow M$ for each $g\in \Gal(K'/K)$ which commutes with $\varphi_{M}$, respects the group structure $\widehat{g_{1}\circ g_{2}}=\widehat{g_1}\circ \widehat{g_2}$, and satisfies $\widehat{g}(sm) = g(s) \widehat{g}(m)$ for all $s \in \mathfrak{S}_{R}[1/u]$ and $m \in M$.
    \end{itemize}
\end{defn}

Let $\M$ be a Breuil-Kisin module with $R$-coefficients. We may decompose $\M$ in terms of idempotents
\[
\M = \bigoplus_{i=0}^{f-1}\M_i,
\]
where $\M_{i} = \mathfrak{e}_{i}\M$. Similarly, let $M$ be an \'etale-$\varphi$ module with $R$-coefficients. We may again decompose $M$ in terms of idempotents
\[
M = \bigoplus_{i=0}^{f-1}M_i,
\]
where $M_{i} = \mathfrak{e}_{i}M$. It follows from the action of $\varphi$ and $\mathrm{Gal}(K'/K)$ on $\mathfrak{e}_i$ that $\varphi_{\mathfrak{M}}$ maps $\mathfrak{M}_i$ to $\mathfrak{M}_{i+1}$ while $\mathrm{Gal}(K'/K)$ maps $\mathfrak{M}_i$ to $\mathfrak{M}_{i}$. In this paper, we will be using this decomposition by idempotents repeatedly and without further comment. 

\subsection{Inertial descent datum}\label{inertia subsec}
Let $\M$ (resp. $M$) be a Breuil-Kisin module (resp. \'etale-$\varphi$ module) of rank two with $R$ coefficients. Let $\tau =\eta\oplus \eta'$ be a fixed tame principal series $\F$-type. 
\begin{defn}
We say $\M$ has tame principal series $\F$-type $\tau$ if (Zariski locally on $\Spec R$ if necessary) there exists a $\Gal(K'/K)$-equivariant isomorphism $\M_i/u\M_i \cong R\eta \oplus R\eta'$ for each $i$.
\end{defn}


Many of the computations carried out in the rest of this paper depend on ensuring that any base change respects this principle series type structure. For this reason, we define such special bases and base changes in the following way.

\begin{defn}\label{inertial-base-change} 
An inertial basis of $\M_i$ (resp. $M_i$) is an ordered basis with respect to which the Galois action is given diagonally by $\eta \oplus \eta'$.

Base change matrices that switch a set of inertial bases (comprising a basis for each $\M_i$, or $M_i$ as the case may be) to another set of inertial bases will be called inertial base change matrices, and the corresponding change of bases will be called an inertial base change.
\end{defn}

Our first order of business is to show that we can always find an inertial basis for a Breuil-Kisin module of tame $\mathbb{F}$-type $\tau$.

\begin{lemma}\label{InertiaForm}
    Let $\M$ be a Breuil-Kisin module of tame $\mathbb{F}$-type $\tau$. For all $i\in\Z/f\Z$, (Zariski locally on $R$ if necessary), there exists an ordered $R[\![u]\!]$-basis $(e_i,f_i)$ of $\mathfrak M_i$ such that the action of $g\in\Gal(K'/K)$ is given by
    \begin{equation}\label{eqn-inertia}
        \widehat{g} (e_{i})=\eta(g)e_i, \ \ \ \ 
        \widehat{g} (f_{i})=\eta'(g)f_i.
    \end{equation}
\end{lemma}
\begin{proof}

Fix $i\in\Z/f\Z$. Without loss of generality (after restricting to an affine open cover of $\text{Spec }R$ if necessary), we have $$\M_i/u\M_i \cong Rx_i \oplus Ry_i$$
where $x_i$ is an eigenvector for $\Gal(K'/K)$ with eigencharacter $\eta$ and $y_i$ is an eigenvector for $\Gal(K'/K)$ with eigencharacter $\eta'$. Fix lifts $\tilde{x}_i$ and $\tilde{y}_i$ of $x_i$ and $y_i$ respectively in $\M_i$.
Set
\begin{align*}
&e_i := \dfrac{1}{|\Gal(K'/K)|}\sum\limits_{j = 0}^{|\Gal(K'/K)| - 1} g^{j}(\tilde{x}_i)\eta(g^{-j}), \text{ and } \\
&f_i := \dfrac{1}{|\Gal(K'/K)|}\sum\limits_{j = 0}^{|\Gal(K'/K)| - 1} g^{j}(\tilde{y}_i)\eta^{\prime}(g^{-j})
\end{align*}
where $g$ is a generator of $\Gal(K'/K)$. Clearly $e_i$ and $f_i$ lift $x_i$ and $y_i$ and are eigenvectors for $\Gal(K'/K)$ with eigenvalues given by $\eta$ and $\eta'$ respectively. We now show that they give an $R[[u]]$-basis of $\M_i$.

Suppose $(\sum_{k \geq 0}{a_k u^k}) e_i + (\sum_{k \geq 0}{b_k u^k}) f_i = 0$ for some $\sum_{k \geq 0}{a_k u^k}, \sum_{k \geq 0}{b_k u^k} \in R[\![u]\!]$. Suppose $n$ is the smallest degree so that either $a_n$ or $b_n$ is nonzero. As $\mathfrak{M}_i$ is $u$-torsion free, we can divide the equation by $u^n$ and assume, without loss of generality, that $a_0 \neq 0$. As $e_i$ and $f_i$ are linearly independent mod $u$, $a_0$ is forced to be $0$, giving a contradiction. Therefore, $e_i$ and $f_i$ are linearly independent, and we have an inclusion $R[\![u]\!] e_i \oplus R[\![u]\!] f_i \hookrightarrow \mathfrak M_i$, which is an equality mod $u$. By Nakayama, $R[\![u]\!] e_i \oplus R[\![u]\!] f_i \hookrightarrow \mathfrak M_i$ is an equality.
\end{proof}

The Frobenius $\Phi_{\M}$ restricts under idempotent decomposition to a map $\varphi^{*}\M_{i-1} \to \M_{i}$ which we will call the $i$-th Frobenius map and denote by $\Phi_{\M,i}$. After fixing an inertial basis for each $i$, let the $i$-th Frobenius map be represented by 
\[
F_i = \begin{pmatrix}
A_1^{(i)}& A_2^{(i)} \\
A_3^{(i)}& A_4^{(i)}
\end{pmatrix},
\] 
such that $A_j^{(i)} \in R[\![u]\!]$ with
\[
\Phi_{\M,i}(1 \otimes e_{i-1}) = A_1^{(i)} e_{i} + A_3^{(i)} f_{i}, \ \ \
\Phi_{\M,i}(1 \otimes f_{i-1}) = A_2^{(i)} e_{i} + A_{4}^{(i)} f_{i},
\]
for any $i\in \Z/f\Z$. The principle series type structure on $\M_i$ allows us to put $F_i$ in a particular form.

\begin{lemma}\label{lem_linearization}
Suppose $\eta \neq \eta'$. After fixing an inertial basis for each $i$, each Frobenius linearization $\Phi_{\M, i}: \varphi^*(\M_{i-1}) \goto \M_{i}$ has a matrix of the form:
\begin{align}\label{equ-frobmatrix}
F_i = \begin{pmatrix}
s_1^{(i)}&u^{e - \gamma_{i}} s_2^{(i)} \\
u^{\gamma_{i}} s_3^{(i)}& s_4^{(i)}
\end{pmatrix},
\end{align}
where $s_j^{(i)} \in R[\![v]\!]$, for $i\in \Z/f\Z$.
\end{lemma}

\begin{proof}
This follows easily from the commutative condition between $\varphi$ and $\Gal(K'/K)$ actions.
\end{proof}
We will further refine the form of the Frobenius action in a subsequent section but first, we must find a description of inertial base change matrices.

\begin{lemma}\label{constraints-base-change-matrix} Suppose $\M$ (resp. $M$) is a Breuil-Kisin module (resp. \'{e}tale-$\varphi$ module) with an inertial basis for each $\M_i$ (resp. $M_i$). 

For each $i$, let $P_i = \begin{pmatrix}
b_1^{(i)}& b_2^{(i)}\\
b_3^{(i)}& b_4^{(i)}
\end{pmatrix}$ be an inertial base change matrix in $\GL_2(R[\![u]\!])$ (resp. $\GL_2(R(\!(u)\!))$. Then $b_1^{(i)}, b_4^{(i)} \in R(\!(v)\!)$, $b_2^{(i)} \in u^{e - \gamma_i} R(\!(v)\!)$ and $b_3^{(i)} \in u^{\gamma_i} R(\!(v)\!)$.

\end{lemma}

\begin{proof}

$P_i$ is an inertial base change matrix if and only if for all $g \in \Gal(K'/K)$, we have:
\begin{align*}
P_{i}^{-1} \cdot \begin{pmatrix}
\eta(g)& 0\\
0 & \eta'(g)
\end{pmatrix} \cdot  g( P_i) = \begin{pmatrix}
\eta(g)& 0\\
0 & \eta'(g)
\end{pmatrix} \iff\\
\begin{pmatrix}
\eta(g) g(b_1^{(i)})& \eta(g) g(b_2^{(i)})\\
\eta'(g) g(b_3^{(i)}) & \eta'(g) g(b_4^{(i)})
\end{pmatrix} = \begin{pmatrix}
\eta(g) b_1^{(i)}& \eta'(g) b_2^{(i)}\\
\eta(g) b_3^{(i)} & \eta'(g) b_4^{(i)}
\end{pmatrix}.
\nonumber
\end{align*}
Comparing the entries and letting $\chi = \eta \eta'^{-1}$, we get $g(b_1^{(i)}) = b_1^{(i)}$, $g(b_2^{(i)}) =\chi(g)^{-1} b_2^{(i)}$,  $g(b_3^{(i)}) =\chi(g) b_3^{(i)}$ and $g(b_4^{(i)}) = b_4^{(i)}$. The statement of the lemma follows immediately.
\end{proof}

\subsection{Moduli Stacks}\label{stack-definitions} We introduced several stacks of Breuil-Kisin modules and Galois representations in Section \ref{intro sec}. In this section, we recall more precisely the definitions of some of these stacks from \cite{CEGS-local-geometry} for later reference. We will be suppressing the superscripts ``$dd$'' (for descent data) and ``$1$'' (indicating that the stack is defined over $\F$) from the original notation used in \cite{CEGS-local-geometry}.

\begin{defn}
    We define $\mathcal{C}$ to be the \textit{fppf} stack defined over $\F$ characterized by the following property: For an $\F$-algebra $R$, $\mathcal{C}(R)$ is the groupoid of rank two Breuil-Kisin modules defined over $R$ of height at most one and descent data from $K'$ to $K$.
\end{defn}

In the  previous section, we focused on Breuil-Kisin modules of tame principal series $\F$-type $\tau$ which motivates the definition of a substack encoding such modules.

\begin{defn}\label{defn-C-tau}
    We define $\mathcal{C}^{\tau}$ to be the closed substack of $\mathcal{C}$ corresponding to Breuil-Kisin modules of tame inertial $\F$-type $\tau$. 
\end{defn}

In the next section, we will restrict to Breuil-Kisin modules satisfying an additional determinant condition which will allow us to further understand the structure of the Frobenius matrices. Looking forward to this, we define the following substack.

\begin{defn}\label{defn-C-BT}
    We define $\mathcal{C}^{\mathrm{BT}}$ (resp. $\mathcal{C}^{\tau, \mathrm{BT}})$ to be the closed substack of $\mathcal{C}$ (resp. $\mathcal{C}^{\tau}$) corresponding to Breuil-Kisin modules that additionally satisfy the the strong determinant condition in \cite[Definition~4.2.2]{CEGS-local-geometry}. 
\end{defn}
We will not recall the precise definition of the strong determinant condition because it is technical and not important for this article. The essential idea is that if $\F'$ is a finite extension of $\F$, then by \cite[Lemma~4.2.16]{CEGS-local-geometry}, the $\F'$ points of $\mathcal{C}^{\mathrm{BT}}$ are precisely those Breuil-Kisin modules whose corresponding Galois representations become Barsotti-Tate over $K'$. 
The $\mathcal{C}^{\mathrm{BT}}$ is the reduced closure of such points.

By inverting the formal power series variable $u$ in $\mathfrak{S}$, we can transform a Breuil-Kisin module into an \'etale-$\varphi$ module. This gives us a morphism from $\mathcal{C}$ to the stack of \'etale-$\varphi$ modules with descent data, denoted $\mathcal{R}$.

\begin{defn}\label{defn-Z}
   We define $\mathcal{Z}$ (resp. $\mathcal{Z}^{\tau}$) to be the scheme-theoretic image of the natural morphism $\mathcal{C}^{\mathrm{BT}} \to \mathcal{R}$ (resp. $\mathcal{C}^{\tau, \mathrm{BT}} \to \mathcal{R}$), in the sense of \cite{EG-Scheme-Image}.
\end{defn}

By \cite[Thm~5.1.2]{CEGS-local-geometry}, the $\overline{\mathbb{F}}_p$-points of $\mathcal{Z}$ are the continuous representations $\overline{r}: G_K \to \GL_2(\overline{\mathbb{F}}_p)$ that are not tr\`es ramifi\'ee. The irreducible components of $\mathcal{Z}$ are labelled by Serre weights, so that if $\sigma$ is a Serre weight and $\mathcal{Z}(\sigma)$ is the corresponding irreducible component, then the $\overline{\mathbb{F}}_p$-points of $\mathcal{Z}(\sigma)$ are precisely the representations $\overline{r}: G_K \to \GL_2(\overline{\mathbb{F}}_p)$ having $\sigma\in W^{\mathrm{cris}}(\overline{r})$ (see \cite[Definition~4.1.7]{BLGG13} for a precise definition of $W^{\mathrm{cris}}(\overline{r})$).

We note that the stacks we are calling $\mathcal{C}^{BT}$, $\mathcal{C}^{\tau, BT}$, $\mathcal{Z}$ and $\mathcal{Z}^{\tau}$ are respectively called $\mathcal{C}^{\text{dd}, BT, 1}$, $\mathcal{C}^{\tau, BT, 1}$, $\mathcal{Z}^{1}$ and $\mathcal{Z}^{\tau, 1}$ in \cite{CEGS-local-geometry}, where \cite[Lem.~5.1.8]{CEGS-local-geometry} shows them to be the underlying reduced substacks of certain formal $p$-adic stacks defined over the formal spectrum of a characteristic $0$ local ring with residue field $\F$. We omit giving details on these formal $p$-adic stacks and refer the interested reader to \textit{loc. cit.} instead.

\subsection{Classification in rank two}\label{classification subsec}
The objective of this section is to classify and describe rank two Breuil-Kisin modules with descent data that satisfy some additional conditions, which we now introduce.
\begin{defn}\cite[Definition~3.1.1]{CDM}\label{defn-hodge-type}
A Breuil-Kisin module $\M$ defined over an $\F$-algebra $R$ with descent data is said to be of of Hodge type $\textbf{v}_0$ if it is of rank two, height at most one, and the $u$-adic valuation of the determinant of each Frobenius map $\Phi_{\M, i}: \varphi^*(\M_{i-1}) \goto \M_{i}$ is $e$.
\end{defn}

\begin{lemma}\label{hodge-type-equiv-strong-det}
Suppose $R$ is a field and $\M$ is a rank two Breuil-Kisin module over $R$ with tame $\F$-type $\tau$ and height at most one. Then Hodge type $\textbf{v}_0$ condition is equivalent to the strong determinant condition of \cite[\textsection~4.2]{CEGS-local-geometry}.
\end{lemma}
\begin{proof}
We will use \cite[Lem.~3.5.11, Prop.~ 4.2.12]{CEGS-local-geometry} for the proof. Although these results are stated for coefficients in finite fields, the proofs in fact work for all fields over $\mathbb{F}$.

One direction (strong determinant condition implies Hodge type $\textbf{v}_0$ condition) follows from \cite[Lem.~4.2.11 (2)]{CEGS-local-geometry}. For the other direction, we observe firstly that $\M_i/\mathrm{im}(\Phi_{\M, i})$ is a finitely generated torsion $R[\![u]\!]$ module, being a quotient of $\M_i/u^e \M_i$.
The determinant of $\Phi_{\M, i}$ is the product of the invariants of $\M_i/\mathrm{im}(\Phi_{\M, i})$ times a unit. Therefore, the sum of $u$-adic valuations of the invariants is $e$, implying that the dimension of $\M_i/\mathrm{im}(\Phi_{\M, i})$ is $e$. 
By \cite[Lem. 4.2.11 (1), Lem. 4.2.12]{CEGS-local-geometry}, the strong determinant condition is satisfied if and only if $\mathrm{dim}_{R}\left(\tfrac{\mathrm{im}(\Phi_{\M, i})}{u^e\M_i}\right) = e(K/\mathbb{Q}_p) \cdot e(K'/K) = e$. As $\M_i$ is a rank two free module over $R[\![u]\!]$, $\M_i/u^e\M_i$ has dimension $2e$ over $R$, and thus $\mathrm{dim}_{R}\left(\tfrac{\mathrm{im}(\Phi_{\M, i})}{u^e\M_i}\right) = e$. 
\end{proof}

Suppose $\M$ is Breuil-Kisin module over $R$ satisfying the Hodge type $\textbf{v}_0$ condition and is of tame principal series $\F$-type $\tau = \eta \oplus \eta'$ with $\eta \neq \eta'$. By Lemma \ref{lem_linearization}, we know that with respect to inertial bases for $\M_{i-1}$ and $\M_{i}$, the $i$-th Frobenius map $\Phi_{\M, i}$ is represented by a matrix $F_i$ of the form (\ref{equ-frobmatrix}). Since
$\det(F_i) = s_1^{(i)} s_4^{(i)} - v s_2^{(i)} s_3^{(i)}$,
the Hodge type $\textbf{v}_0$ condition implies $v \mid s_1^{(i)} s_4^{(i)}$. 
This gives us three cases:
\begin{enumerate}
    \item If $s_1^{(i)}$ is a non-unit and $s_4^{(i)}$ is a unit, then $F_i$ is of \textit{genre }$\Ieta$, denoted by $\mathcal{G}(F_i) = \Ieta$.
    \item If $s_1^{(i)}$ is a unit and $s_4^{(i)}$ is a non-unit, then $F_i$ is of \textit{genre }$\Ietaa$, denoted by $\mathcal{G}(F_i) = \Ietaa$.
    \item If both $s_1^{(i)}$ and $s_4^{(i)}$ are non-units, then $F_i$ is of \textit{genre }II, denoted by $\mathcal{G}(F_i) = \II$. 
\end{enumerate}
A direct calculation shows that if $\{P_i\}_{i\in \mathbb{Z}/f\mathbb{Z}}$ is a set of inertial base change matrices, then
$
\mathcal{G} (P_{i}^{-1} \cdot F_i \cdot \varphi(P_{i-1})) = \mathcal{G} (F_i)
$. We are therefore justified in defining the \textit{genre of }$\M_i$ to be $\mathcal{G}(\M_i) := \mathcal{G} ( F_i)$.

\begin{defn}\label{eta-eta'-form}
For a Breuil-Kisin module $\M$ over $R$ of rank two, height at most one and tame principal series $\F$-type $\tau = \eta \oplus \eta'$ where $\eta \neq \eta'$, let $\{F_i\}$ be the Frobenius matrices written with respect to a choice of inertial basis for each $i$. We will say that $F_i$ is in $\eta$-form if its top left entry is divisible by $v$. If the bottom right entry is divisible by $v$, we will say it is in $\eta'$-form.
\end{defn}
We note that the property of being in $\eta$-form or in $\eta'$-form is preserved by inertial base change, and can be seen as a property of the $i$-th Frobenius map $\Phi_{\M, i}$.

\begin{defn}\label{regular}
    A Breuil-Kisin module $\M$ over an $\F$-algebra $R$ with descent data is regular if it is of Hodge type $\textbf{v}_0$, tame principal series $\F$-type $\tau = \eta \oplus \eta'$ such that $\eta \neq \eta'$, and with each Frobenius map either in $\eta$-form or in $\eta'$-form.
\end{defn}

For the rest of this section, our Breuil-Kisin modules will be defined over an $\F$-algebra $R$, and will be regular. We now turn to show that when $R$ is Artinian local, we can choose inertial bases such that each Frobenius matrix $F_i$ takes one of three forms depending on the genre $\G(\M_i)$.

\begin{defn}\label{defn-CDM-form}
Let $R$ be an Artinian local ring over $\F$ with maximal ideal $\m$. Let $\M$ over $R$ be regular. We say that the Frobenius matrices $\{F_i\}$ (written with respect to an inertial basis for each $\M_i$) are in CDM form if for $i \in \{1, ..., f-1\}$, we have:
\bal
F_i &= \begin{cases}
\begin{pmatrix}
v& 0\\
A_i u ^{\gamma_{i}}& 1
\end{pmatrix}& \text{if  } \G(F_i) = \Ieta,\\
\hspace{1cm}\\
\begin{pmatrix}
0& - u^{e- \gamma_{i}}\\
u^{\gamma_{i}}& A_i'
\end{pmatrix}& \text{if  } \G(F_i) = \II \text{ and } F_i \text{ is in } \eta \text{-form},\\
\hspace{1cm}\\
\begin{pmatrix}
1& A'_i u^{e- \gamma_{i}}\\
0& v
\end{pmatrix}& \text{if  } \G(F_i) = \Ietaa,\\
\hspace{1cm}\\
\begin{pmatrix}
A_i & - u^{e- \gamma_{i}}\\
u^{\gamma_{i}}& 0
\end{pmatrix}& \text{if  } \G(F_i) = \II \text{ and } F_i \text{ is in } \eta' \text{-form},\\
\end{cases}
\nal
while for $i = 0$, 
\bal
F_{0} &= \begin{cases}
\begin{pmatrix} \alpha & 0 \\
0 & \alpha' \end{pmatrix}
\begin{pmatrix}
v& 0\\
A_{0} u ^{\gamma_{0}}& 1
\end{pmatrix}& \text{if  } \G(F_{0}) = \Ieta,\\
\hspace{1cm}\\
\begin{pmatrix} \alpha & 0 \\
0 & \alpha' \end{pmatrix}
\begin{pmatrix}
0& - u^{e- \gamma_{0}}\\
u^{\gamma_{0}}& A_{0}'
\end{pmatrix}& \text{if  } \G(F_{0}) = \II \text{ and } F_{0} \text{ is in } \eta \text{-form},\\
\hspace{1cm}\\
\begin{pmatrix} \alpha & 0 \\
0 & \alpha' \end{pmatrix}
\begin{pmatrix}
1& A'_{0} u^{e- \gamma_{0}}\\
0& v
\end{pmatrix}& \text{if  } \G(F_{0}) = \Ietaa,\\
\hspace{1cm}\\
\begin{pmatrix} \alpha & 0 \\
0 & \alpha' \end{pmatrix}
\begin{pmatrix}
A_{0} & - u^{e- \gamma_{0}}\\
u^{\gamma_{0}}& 0
\end{pmatrix}& \text{if  } \G(F_{0}) = \II \text{ and } F_{0} \text{ is in } \eta' \text{-form}.\\
\end{cases}
\nal
where each of $\alpha$, $\alpha'$, $A_i$, $A'_i$ are elements of $R$. Moreover, when $\mathcal{G}(F_i) = \II$ and $F_i$ is in $\eta$-form, $A'_i \in \m$. Similarly, when $\mathcal{G}(F_i) = \II$ and $F_i$ is in $\eta'$-form, $A_i \in \m$.\\
\end{defn}

We will describe the matrices of Definition \ref{defn-CDM-form} in terms of the parameters $$(\alpha, \alpha', A_0, A'_0, ..., A_{f-1}, A'_{f-1}).$$ If no $A_i$ shows up in the description of $F_i$, we set it equal to $0$. Similarly, if no $A_i'$ shows up in the description of $F_i$, we set it equal to $0$. Note that, in general, $(\alpha, \alpha', A_0, A'_0, ..., A_{f-1}, A'_{f-1})$ do not uniquely determine the Frobenius matrices. They do, however, if we know which Frobenius maps are in $\eta$-form and which are in $\eta'$-form.

\begin{defn}\label{define-modified-bad-genre}
Let $R$ be an Artinian local ring over $\F$ with maximal ideal $\m$. A regular Breuil-Kisin module $\M$ over $R$ is of bad genre if the following conditions are satisfied:

\begin{enumerate}
    \item 
    $\forall i, (\mathcal{G}(F_i), z_{i}) \in \{(\II, 0), (\II, p-1), (\Ieta, 1), (\Ieta, p-1), (\Ietaa, 0), (\Ietaa, p-2)\}$.
    \item 
    If $(\mathcal{G}(F_i), z_{i}) \in \{(\II, 0), (\Ietaa, 0), (\Ietaa, p-2)\}$, then \\
    $(\mathcal{G}(F_{i+1}), z_{i+1}) \in \{(\II, p-1), (\Ieta, p-1), (\Ietaa, p-2)\}$.
    \item 
    If $(\mathcal{G}(F_i), z_{i}) \in \{(\II, p-1), (\Ieta, 1), (\Ieta, p-1)\}$, then \\
    $(\mathcal{G}(F_{i+1}), z_{i+1}) \in \{(\II, 0), (\Ieta, 1), (\Ietaa, 0)\}$.
\end{enumerate}

The $(z_j)_j$ are as defined in (\ref{defn-z_i}).
\end{defn}

The existence of Breuil-Kisin modules of bad genre is the primary contributor to the failure of our methods in some cases. The following proposition is the first indication of this.
 
\begin{prop}\label{convergence-CDM-matrix}
Let $R$ be an Artinian local ring over $\F$ with maximal ideal $\m$. Let $\M$ be a regular Breuil-Kisin module over $R$, not of bad genre.
Then there exists an inertial basis for each $\M_i$ with respect to which the Frobenius matrices are in CDM form (see Definition \ref{defn-CDM-form}).
\end{prop}

The proof of Proposition \ref{convergence-CDM-matrix} is very similar to that of \cite[Lem.~3.1.20]{CDM}, with slight differences to accommodate Artinian local algebras. Before showing the proof, we first state some technical lemmas and definitions required in the proof.

\begin{defn}\label{defn-B-operation}(As in \cite[Lem.~3.1.16]{CDM}).
Let $R$ be an Artinian local ring over $\F$ with maximal ideal $\m$. Define an operation $\mathcal{B}$ as follows. For any 
\bal 
G = \begin{pmatrix} 
v s_1 & u^{e - \gamma} s_2\\
u^{\gamma} s_3& s_4
\end{pmatrix}
\nal 
with $s_j \in R[\![v]\!]$ and $\det(G) = v \alpha$ for some $\alpha \in R[\![v]\!]^*$, define $A, A' \in R$ by $A \equiv s_3/s_4 \mod v$ (if $s_4$ is invertible) and $A' \equiv s_4/s_3 \mod v$ (if $s_4$ is not invertible). Then $\mathcal{B}(G)$ is defined as follows:
\bal 
\mathcal{B}(G) := \begin{cases} \begin{pmatrix}
s_1 - A s_2&
u^{e - \gamma} s_2\\
u^{\gamma}  \frac{s_3 - A s_4}{u^e}& s_4
\end{pmatrix}& \text{if  }s_4 \in R[\![v]\!]^*, \\
\hspace{1cm}\\
\begin{pmatrix}
s_2 - A' s_1&
u^{e-\gamma} s_1\\
u^\gamma \frac{s_4 - A' s_3}{u^e}&
s_3
\end{pmatrix}& \text{if  } s_4 \not\in R[\![v]\!]^*.
\end{cases}
\nal 
Note that $\det(\mathcal{B}(G) ) = \pm \alpha$, so $\mathcal{B}(G)$ is invertible. Furthermore,
\begin{align}\label{G-is-a-product-B(G)M}
G =\begin{cases} \mathcal{B}(G) \begin{pmatrix}
v& 0\\
A u ^\gamma& 1
\end{pmatrix}& \text{if  }s_4 \in R[\![v]\!]^*,\\
\hspace{1cm}\\
\mathcal{B}(G)\begin{pmatrix}
0& u^{e-\gamma}\\
u^\gamma& A'
\end{pmatrix} & \text{if  } s_4 \not\in R[\![v]\!]^*.
\end{cases}
\end{align}

For any 
\[
G = \begin{pmatrix} 
s_1 & u^{e - \gamma} s_2\\
u^{\gamma} s_3& v s_4
\end{pmatrix}
\] 
with $s_j \in R[\![v]\!]$ and $\det(G) = v \alpha$ for some $\alpha \in R[\![v]\!]^*$, define $A, A' \in R$ by $A \equiv s_1/s_2 \mod v$ (if $s_1$ is not invertible) and $A' \equiv s_2/s_1 \mod v$ (if $s_1$ is invertible). Then $\mathcal{B}(G)$ is defined in a way that is compatible with the above definitions if we want to interchange $\eta$ and $\eta'$. That is,
\[
\mathcal{B}(G) := \begin{cases} \begin{pmatrix}
s_1 &
u^{e - \gamma} \frac{s_2 - A' s_1}{u^e}\\
u^{\gamma} s_3& s_4 - A's_3
\end{pmatrix}& \text{if  }s_1 \in R[\![v]\!]^*,\\
\hspace{1cm}\\
\begin{pmatrix}
s_2&
u^{e-\gamma} \frac{s_1 - A s_2}{u^e}\\
u^\gamma s_4&
s_3 - As_4
\end{pmatrix}& \text{if  }s_1 \not\in R[\![v]\!]^*.
\end{cases}
\]
Note that $\det(\mathcal{B}(G) ) = \pm \alpha$, so $\mathcal{B}(G)$ is invertible. Furthermore,
\bal 
G =\begin{cases} \mathcal{B}(G) \begin{pmatrix}
1& A' u ^{e-\gamma}\\
0 & u^e
\end{pmatrix}& \text{if  }s_1 \in R[\![v]\!]^*,\\
\hspace{1cm}\\
\mathcal{B}(G)\begin{pmatrix}
A& u^{e-\gamma}\\
u^\gamma& 0
\end{pmatrix} & \text{if  }s_1 \not\in R[\![v]\!]^*.
\end{cases}
\nal 
\end{defn}

\begin{lemma}\label{reduction-step}
Consider the matrix
\bal
P = \begin{pmatrix}
\sigma_1& u^{e-\gamma_{i-1}} \sigma_2\\
u ^{\gamma_{i-1}} \sigma_3 & \sigma_4
\end{pmatrix},
\nal with $\sigma_j \in R[\![v]\!]$ and $\det(P) \in R[\![u]\!]^{*}$. Let 
\bal
F =  \begin{pmatrix}
v a& u^{e-\gamma_{i}} b\\
u ^{\gamma_{i}}c& d
\end{pmatrix} \text{   or   } \begin{pmatrix}
a& u^{e-\gamma_{i}} b\\
u ^{\gamma_{i}}c& v d
\end{pmatrix}
\nal
with $a, b, c, d \in R[\![v]\!]$ and $ad-bc  \in R[\![v]\!]^{*}$. If $M$ is the matrix such that $F = \mathcal{B}(F)M$, where $\mathcal{B}$ is the operation in Definition \ref{defn-B-operation}, then $\mathcal{B}(F\varphi(P)) = \mathcal{B}(F)\mathcal{B}(M\varphi(P))$.
\end{lemma}
\begin{proof}
Let $X$ be such that $M\varphi(P) = \mathcal{B}(M\varphi(P))X$ and $Y$ be such that $F\varphi(P) = \mathcal{B}(F\varphi(P))Y$. It suffices to show that $X=Y$, because if so, by inverting $X$ and $Y$ in $\GL_2(R(\!(v)\!))$, we can show that $\mathcal{B}(F\varphi(P)) = F\varphi(P)Y^{-1} = \mathcal{B}(F)M\varphi(P) X^{-1} = \mathcal{B}(F)\mathcal{B}(M\varphi(P))$.

We first consider the case where $F =  \begin{psmallmatrix}
v a& u^{e-\gamma_{i}} b\\
u ^{\gamma_{i}}c& d
\end{psmallmatrix}$. Note that for any $G = \begin{psmallmatrix} 
v s_1& u^{e - \gamma_{i}} s_2\\
u^{\gamma_{i}} s_3& s_4
\end{psmallmatrix}$
with $s_j \in R[\![v]\!]$ and $\det(G) = v \alpha$ for some $\alpha \in R[\![v]\!]^*$, we can calculate $\mathcal{B}(G)$ and scalars $A$ or $A'$ so that (\ref{G-is-a-product-B(G)M}) holds.
It suffices to show that $A$ and $A'$ do not depend on whether $G=F\varphi(P)$ or $G=M\varphi(P)$. We have 
\bal
F\varphi(P) = \begin{pmatrix}
v(a \varphi(\sigma_1) + b u^{e z_{i}} \varphi(\sigma_3))&
u^{e - \gamma_{i}} (a u^{e ( p - z_{i} ) } \varphi(\sigma_2) + b \varphi(\sigma_4) )  \\
u^{\gamma_{i}} (c \varphi(\sigma_1) + d u^{e z_{i} } \varphi(\sigma_3 ))& 
c u^{e ( p - z_{i} ) } \varphi(\sigma_2 )+ d \varphi(\sigma_4) 
\end{pmatrix}.
\nal 

Let $\overline{x}$ denote the constant part of any $x \in R[\![v]\!]$. By Definition \ref{defn-B-operation}, $M$ is given by: 
\bal
M =\begin{cases} \begin{pmatrix}
v& 0\\
u ^{\gamma_{i}} \frac{\overline{c}}{\overline{d}} & 1
\end{pmatrix}& \text{if  } d \in R[\![v]\!]^*,\
\vspace{0.5cm}\\
\begin{pmatrix}
0& u^{e-\gamma_{i}}\\
u^{\gamma_{i}}& \frac{\overline{d}}{\overline{c}}
\end{pmatrix} & \text{if  }d \not\in R[\![v]\!]^*.
\end{cases}
\nal
Therefore,
 \bal
M\varphi(P) =\begin{cases} \begin{pmatrix}
v \varphi(\sigma_1) &
u^{e - \gamma_{i}} u^{e ( p - z_{i}) } \varphi(\sigma_2) \\
u^{\gamma_{i}} (\frac{\overline{c}}{\overline{d}} \varphi(\sigma_1) + u^{e z_{i} } \varphi(\sigma_3 ))& 
\frac{\overline{c}}{\overline{d}} u^{e ( p - z_{i} ) } \varphi(\sigma_2 )+ \varphi(\sigma_4) 
\end{pmatrix}& \text{if  } d \in R[\![v]\!]^*,\
\vspace{0.5cm}\\
\begin{pmatrix}
v u^{e z_{i}} \varphi(\sigma_3)&
u^{e - \gamma_{i}} \varphi(\sigma_4)  \\
u^{\gamma_{i}} (\varphi(\sigma_1) + \frac{\overline{d}}{\overline{c}} u^{e z_{i} } \varphi(\sigma_3 ))& 
u^{e ( p - z_{i} ) } \varphi(\sigma_2 )+ \frac{\overline{d}}{\overline{c}} \varphi(\sigma_4) 
\end{pmatrix} & \text{if  }d \not\in R[\![v]\!]^*.
\end{cases}
\nal
Then $X=Y$ follows immediately from Definition \ref{defn-B-operation}. Similar considerations hold for the case $F =  \begin{psmallmatrix}
a& u^{e-\gamma_{i}} b\\
u ^{\gamma_{i}}c& vd
\end{psmallmatrix}$.
\end{proof}

\begin{defn}
    Let $R$ be an Artinian local ring over $\F$ with maximal ideal $\m$. Let $n \in \mathbb{Z}_{\geq 0}$ be the maximum such that $\mathfrak{m}^n \neq 0$.
    For $t \in \mathbb{Z}_{\geq 0}$, define the ideal $I_t \subset R[\![v]\!]$ as follows:
    \bal
    &I_t = \{\sum_{i= \max{\{t-n, 0\}}}^{\infty} a_i v^i \in R[\![v]\!]:  a_0 \in \mathfrak{m}^{t} \}.
    \nal

    In other words, for $t \leq n$, $I_t = \{\sum_{i=0}^{\infty} a_i v^i \in R[\![v]\!]:  a_0 \in \mathfrak{m}^{t} \}$. For $t > n$, $I_t = \{\sum_{i=t-n}^{\infty} a_i v^i \in R[\![v]\!]\}$.
\end{defn}


\begin{defn}\label{defn-t-close}
Let $R$ be an Artinian local ring over $\F$ with maximal ideal $\m$. Let $P^{(0)}, P^{(1)} \in \mathrm{GL}_2(R[[u]])$ be such that for $j \in \{0, 1\}$,
\bal
P^{(j)} = \begin{pmatrix}
\sigma^{(j)}_1& u^{e-\gamma_{i-1}} \sigma^{(j)}_2\\
u ^{\gamma_{i-1}} \sigma^{(j)}_3 & \sigma^{(j)}_4
\end{pmatrix},
\nal with $\sigma^{(j)}_k \in R[\![v]\!]$ for $k \in \{1, 2, 3, 4\}$.

We say that $P^{(0)}$ and $P^{(1)}$ are $t$-close if there exists a matrix 
\bal
Y = \begin{pmatrix}
y_1 & u^{e-\gamma_{i}}y_2 \\
u^{\gamma_{i}}y_3 & y_4
\end{pmatrix}
\nal satisfying $y_1 \equiv y_2 \equiv y_3 \equiv y_4 \equiv 0 \mod I_t$, such that $P^{(1)} = P^{(0)} + Y$.
\end{defn}



\begin{lemma}\label{inductive-step}(c.f. \cite[Lem.~3.1.19]{CDM})
Let $R$ be an Artinian local ring over $\F$ with maximal ideal $\m$. Let 
\bal
P = \begin{pmatrix}
\sigma_1& u^{e-\gamma_{i-1}} \sigma_2\\
u ^{\gamma_{i-1}} \sigma_3 & \sigma_4
\end{pmatrix}
\nal
be an inertial base change matrix with $\sigma_1 - 1 \equiv \sigma_4 - 1 \equiv 0 \mod v$. Let $P'$ be $t$-close to $P$ with $Y = P' -P = \begin{pmatrix}
y_1 & u^{e-\gamma_{i}}y_2 \\
u^{\gamma_{i}}y_3 & y_4
\end{pmatrix}$ and $y_1 \equiv y_4 \equiv 0 \mod v$.

\begin{enumerate}
\item Let $M = \begin{pmatrix}
v & 0\\
u ^{\gamma_{i}}a& 1
\end{pmatrix}$ for $a \in R$. Then $Y' = \mathcal{B}(M\varphi(P')) - \mathcal{B}(M\varphi(P)) = \begin{pmatrix}
y'_1 & u^{e-\gamma_{i}}y'_2 \\
u^{\gamma_{i}}y'_3 & y'_4
\end{pmatrix}$ satisfies:
\bal
&y'_1 \equiv 0 \mod I_{t+1}, \\
&y'_2 \equiv 0 \mod I_{t+1},\\
&y'_3 \equiv \begin{cases} 
0 &\text{ if } z_{i} \neq 1, p-1 ,\\
\varphi(y_3) &\text{ if } z_{i} = 1, \\
-a^2\varphi(y_2) &\text{ if } z_{i} = p-1,
\end{cases} \mod I_{t+1}, \\
&y'_4 \equiv 0 \mod I_{t+1}.
\nal
The congruences also hold true mod $v$.

\item Let $M = \begin{pmatrix}
1 & u^{e-\gamma_{i}} a'\\
0& v
\end{pmatrix}$ for $a' \in R$. Then $Y' = \mathcal{B}(M\varphi(P')) - \mathcal{B}(M\varphi(P)) = \begin{pmatrix}
y'_1 & u^{e-\gamma_{i}}y'_2 \\
u^{\gamma_{i}}y'_3 & y'_4
\end{pmatrix}$ satisfies:
\bal
&y'_1 \equiv 0 \mod I_{t+1}, \\
&y'_2 \equiv \begin{cases} 
0 &\text{ if } z_{i} \neq 0, p-2, \\
-a'^{2} \varphi(y_3) &\text{ if } z_{i} = 0, \\
\varphi(y_2) &\text{ if } z_{i} = p-2,
\end{cases} \mod I_{t+1}, \\
&y'_3 \equiv 0 \mod I_{t+1}, \\
&y'_4 \equiv 0 \mod I_{t+1}.
\nal
The congruences also hold true mod $v$. 

\item Let $M = \begin{pmatrix}
0 & u^{e-\gamma_{i}}\\
u ^{\gamma_{i}}& a'
\end{pmatrix}$ for $a' \in \mathfrak{m}$ or $M = \begin{pmatrix}
a & u^{e-\gamma_{i}}\\
u ^{\gamma_{i}}& 0
\end{pmatrix}$ for $a \in \mathfrak{m}$. Then $Y' = \mathcal{B}(M\varphi(P')) - \mathcal{B}(M\varphi(P)) = \begin{pmatrix}
y'_1 & u^{e-\gamma_{i}}y'_2 \\
u^{\gamma_{i}}y'_3 & y'_4
\end{pmatrix}$ satisfies:
\bal
&y'_1 \equiv 0 \mod I_{t+1}, \\
&y'_2 \equiv \begin{cases} 0 &\text{ if } z_{i} \neq 0,\\
\varphi(y_3) &\text{ if } z_{i} = 0,
\end{cases} \mod I_{t+1}, \\
&y'_3 \equiv \begin{cases} 
0 &\text{ if } z_{i} \neq p-1, \\
\varphi(y_2) &\text{ if } z_{i} = p-1,
\end{cases} \mod I_{t+1}, \\
&y'_4 \equiv 0 \mod I_{t+1}.
\nal
\end{enumerate}
\end{lemma}

\begin{proof}
\begin{enumerate}
\item The matrix $M\varphi(P')$ equals
\bal
\begin{pmatrix}
v\varphi(\sigma_1 + y_1) & u^{e-\gamma_{i}}(u^{e(p - z_{i})}\varphi(\sigma_2 + y_2))\\
u^{\gamma_{i+1}}(u^{ez_{i}} \varphi(\sigma_3 + y_3) + a \varphi(\sigma_1 + y_1)) & \varphi(\sigma_4 + y_4) + au^{e(p-z_{i})}\varphi(\sigma_2 + y_2)
\end{pmatrix}.
\nal

Let $b \in R$ such that $b = \frac{a(\sigma_1 + y_1) + u^{ez_{i}}(\sigma_3 + y_3)}{(\sigma_4 + y_4)}$ mod $v$. Then $\mathcal{B}(M\varphi(P')) = \begin{psmallmatrix}
\sigma'_1& u^{e-\gamma_{i}} \sigma'_2\\
u ^{\gamma_{i}} \sigma'_3 & \sigma'_4
\end{psmallmatrix}$ where
\bal
&\sigma'_1 = \varphi(\sigma_1 + y_1) - v^{(p-z_{i})}b\varphi(\sigma_2 + y_2),\\
&\sigma'_2 = v^{(p-z_{i})}\varphi(\sigma_2 + y_2),\\
&v \sigma'_3 = a\varphi(\sigma_1 + y_1) - b\varphi(\sigma_4 + y_4) + v^{ z_{i}}\varphi(\sigma_3 + y_3) - abv^{(p-z_{i})} \varphi(\sigma_2 + y_2),\\
&\sigma'_4 = av^{(p-z_{i})}\varphi(\sigma_2 + y_2) + \varphi(\sigma_4 + y_4).
\nal 
The congruences for $y'_1$, $y'_2$ and $y'_4$ are immediate from the above. For $y'_3$, consider first the case where $t>n$. Then $\varphi(y_1) \equiv \varphi(y_2) \equiv \varphi(y_3) \equiv \varphi(y_4) \equiv 0 \mod I_{t+2}$ (as $p>2$). Moreover, $a=b$. We thus have
$$y'_3 = a v^{-1} \left( \varphi(y_1) - \varphi(y_4)\right) + v^{z_{i} - 1} \varphi(y_3) - a^2 v^{(p-1-z_{i})} \varphi(y_2)$$
which is $0$ mod $I_{t+1}$.

Now, let $t \leq n$. First consider the case where $z_{i} = 0$.  The right side of the equality $v \sigma'_3 = a\varphi(\sigma_1 + y_1) - b\varphi(\sigma_4 + y_4) + v^{ z_{i}}\varphi(\sigma_3 + y_3) - abv^{(p-z_{i})} \varphi(\sigma_2 + y_2)$ can have no constant terms. Therefore, $v y'_3$ depends on $v^p \varphi(y_2)$ and the nonconstant parts of $\varphi(y_1)$, $\varphi(y_3)$ and $\varphi(y_4)$, which are all $0$ mod $v^2$. Therefore, $y'_3 \equiv 0$ mod $v$ and therefore, mod $I_{t+1}$.
If $z_{i} \neq 0$, $v y'_3$ depends on $v^{p-z_{i}} \varphi(y_2)$, $v^{z_{i}}\varphi(y_3)$ and the nonconstant parts of $\varphi(y_1)$ and $\varphi(y_4)$. The latter two terms are $0$ mod $v^2$. This gives us the following equivalence mod $v$ (and hence, mod $I_{t+1}$):
$$y'_3 \equiv v^{z_{i} - 1} \varphi(y_3) - a^2 v^{p-1-z_{i}} \varphi(y_2).$$
The desired congruences follow from the same reasoning as for the first case.

\item The matrix $M\varphi(P')$ equals
\bal
\begin{pmatrix}
\varphi(\sigma_1 + y_1) + a'u^{e(1+z_{i})}\varphi(\sigma_3 + y_3) & u^{e-\gamma_{i}}(u^{e(p-1-z_{i})} \varphi(\sigma_2 + y_2) + a' \varphi(\sigma_4 + y_4)) \\
u^{\gamma_{i}}(u^{e(1 + z_{i})}\varphi(\sigma_3 + y_3)) & v\varphi(\sigma_4 + y_4)
\end{pmatrix}.\nal
Let $b' \in R$ such that $b' = \frac{u^{e(p-1-z_{i})} (\sigma_2 + y_2) + a'(\sigma_4 + y_4)}{\sigma_1 + y_1}$ mod $v$. Then $\mathcal{B}(M\varphi(P')) = \begin{psmallmatrix}
\sigma'_1& u^{e-\gamma_{i}} \sigma'_2\\
u ^{\gamma_{i}} \sigma'_3 & \sigma'_4
\end{psmallmatrix}$ where
\begin{align*}
&\sigma'_1 = \varphi(\sigma_1 + y_1) - a' v^{(1+z_{i})}\varphi(\sigma_3 + y_3),\\
&v \sigma'_2 = v^{(p-1-z_{i})}\varphi(\sigma_2 + y_2) + a'\varphi(\sigma_4 + y_4) - b'\varphi(\sigma_1 + y_1) + a'b' v^{(1+ z_{i})}\varphi(\sigma_3 + y_3),\\
&\sigma'_3 = v^{(1+z_{i})} \varphi(\sigma_3 + y_3),\\
&\sigma'_4 = b'v^{(1 + z_{i})}\varphi(\sigma_3 + y_3) + \varphi(\sigma_4 + y_4).
\end{align*} 
The desired congruences follow using the same arguments as for the first part.

\item 
Suppose $M = \begin{psmallmatrix}
0 & u^{e-\gamma_{i}}\\
u ^{\gamma_{i}}& a'
\end{psmallmatrix}$ for $a' \in \mathfrak{m}$. Then, $M\varphi(P')$ equals
$$ \begin{pmatrix}
u^{e(1+z_{i})}\varphi(\sigma_3 + y_3) & u^{e-\gamma_{i}}\varphi(\sigma'_4) \\
u^{\gamma_{i}}(\varphi(\sigma_1 + y_1) + a'u^{e z_{i}}\varphi(\sigma_3 + y_3)) & u^{e(p-z_{i})}\varphi(\sigma_{2} + y_2) + a'\varphi(\sigma_4 + y_4)
\end{pmatrix}.$$
Let $b' \in R$ such that $b' = \frac{a'(\sigma_4 + y_4)}{\sigma_1 + y_1 + a' v^{z_{i}}(\sigma_3 + y_3)}$ mod $v$. 
Note that $b' \in \mathfrak{m}$ because $a' \in \mathfrak{m}$, and that $b' \equiv a' \mod \mathfrak{m}^2 $ \; because $b' - a' \equiv \frac{a'(\sigma_4 - \sigma_1 + y_4 - y_1) - a'^2 v^{z_{i}}(\sigma_3 + y_3)}{\sigma_1 + y_1 +a' v^{z_{i}}(\sigma_3 + y_3)} \equiv 0 \mod (\mathfrak{m}^2, v)$.
Then
$\mathcal{B}(M\varphi(P')) = \begin{psmallmatrix}
\sigma'_1& u^{e-\gamma_{i}} \sigma'_2\\
u ^{\gamma_{i}} \sigma'_3 & \sigma'_4
\end{psmallmatrix}$, where
\bal
&\sigma'_1  = \varphi(\sigma_4 + y_4) - b' v^{z_{i}}\varphi(\sigma_3 + y_3),\\
&\sigma'_2 = v^{z_{i}}\varphi(\sigma_3 + y_3),\\
&v \sigma'_3 = v^{(p-z_{i})} \varphi(\sigma_2 + y_2) + a' \varphi(\sigma_4 + y_4) - b'(\varphi(\sigma_1 + y_1) + a'v^{z_{i}}\varphi(\sigma_3 + y_3)),\\
&\sigma'_4 = \varphi(\sigma_1 + y_1) + a'v^{z_{i}}\varphi(\sigma_3 + y_3).
\nal 

On the other hand, suppose $M = \begin{psmallmatrix}
a & u^{e-\gamma_{i}}\\
u ^{\gamma_{i}}& 0
\end{psmallmatrix}$ for $a \in \mathfrak{m}$. Let $b \in R$ such that $b = \frac{a (\sigma_1 + y_1)}{a v^{(p-1-z_{i})}(\sigma_2 + y_2)+ (\sigma_4 + y_4)}$ mod $v$. Again, note that $b \in \mathfrak{m}$, since $a \in \mathfrak{m}$ and that $b \equiv a \mod \mathfrak{m}^2$. By symmetry (we can interchange $\eta$ and $\eta'$ to convert this to a previous computed case), we have $\mathcal{B}(M\varphi(P')) = \begin{psmallmatrix}
\sigma'_1& u^{e-\gamma_{i}} \sigma'_2\\
u ^{\gamma_{i}} \sigma'_3 & \sigma'_4
\end{psmallmatrix}$, where
\bal
&\sigma'_1 = av^{(p-1-z_{i})} \varphi(\sigma_2 + y_2) + \varphi(\sigma_4),\\
&v \sigma'_2 = v^{(1+z_{i})}\varphi(\sigma_3 + y_3) + a\varphi(\sigma_1 + y_1) -b(\varphi(\sigma_4 + y_4) + av^{(p-1-z_{i})}\varphi(\sigma_2 + y_2)), \\
&\sigma'_3 = v^{(p-1-z_{i})} \varphi(\sigma_2 + y_2),\\
&\sigma'_4 = \varphi(\sigma_1 + y_1) - bv^{(p-1-z_{i})}\varphi(\sigma_2 + y_2).
\nal 
The congruences follow immediately.
\end{enumerate}
\end{proof}

\begin{cor}\label{cdm-algo-multiplicative}
Let $R$ be an Artinian local ring over $\F$ with maximal ideal $\m$. Let 
\bal
P' = \begin{pmatrix}
\sigma_1& u^{e-\gamma_{i-1}} \sigma_2\\
u ^{\gamma_{i-1}} \sigma_3 & \sigma_4
\end{pmatrix}
\nal
be an inertial base change matrix which is $t$-close to $Id$, with diagonal entries congruent to $1$ mod $v$. 

\begin{enumerate}
\item Let $M = \begin{pmatrix}
v & 0\\
u ^{\gamma_{i}}a& 1
\end{pmatrix}$ for $a \in R$. Then $\mathcal{B}(M\varphi(P')) = \begin{pmatrix}
\sigma'_1& u^{e-\gamma_{i}} \sigma'_2\\
u ^{\gamma_{i}} \sigma'_3 & \sigma'_4
\end{pmatrix}$ satisfies:
\bal
&\sigma'_1 - 1 \equiv 0 \mod I_{t+1}, \\
&\sigma'_2 \equiv 0 \mod I_{t+1},\\
&\sigma'_3 \equiv \begin{cases} 
0 &\text{ if } z_{i} \neq 1, p-1 ,\\
\varphi(\sigma_3) &\text{ if } z_{i} = 1, \\
-a^2\varphi(\sigma_2) &\text{ if } z_{i} = p-1,
\end{cases} \mod I_{t+1}, \\
&\sigma'_4 - 1 \equiv 0 \mod I_{t+1}.
\nal
The congruences also hold true mod $v$.

\item Let $M = \begin{pmatrix}
1 & u^{e-\gamma_{i}} a'\\
0& v
\end{pmatrix}$ for $a' \in R$. Then $\mathcal{B}(M\varphi(P')) = \begin{pmatrix}
\sigma'_1& u^{e-\gamma_{i}} \sigma'_2\\
u ^{\gamma_{i}} \sigma'_3 & \sigma'_4
\end{pmatrix}$ satisfies:
\bal
&\sigma'_1 - 1 \equiv 0 \mod I_{t+1}, \\
&\sigma'_2 \equiv \begin{cases} 
0 &\text{ if } z_{i} \neq 0, p-2, \\
-a'^{2} \varphi(\sigma_3) &\text{ if } z_{i} = 0, \\
\varphi(\sigma_2) &\text{ if } z_{i} = p-2,
\end{cases} \mod I_{t+1}, \\
&\sigma'_3 \equiv 0 \mod I_{t+1}, \\
&\sigma'_4 - 1 \equiv 0 \mod I_{t+1}.
\nal
The congruences also hold true mod $v$. 

\item Let $M = \begin{pmatrix}
0 & u^{e-\gamma_{i}}\\
u ^{\gamma_{i}}& a'
\end{pmatrix}$ for $a' \in \mathfrak{m}$ or $M = \begin{pmatrix}
a & u^{e-\gamma_{i}}\\
u ^{\gamma_{i}}& 0
\end{pmatrix}$ for $a \in \mathfrak{m}$. Then $\mathcal{B}(M\varphi(P')) = \begin{pmatrix}
\sigma'_1& u^{e-\gamma_{i}} \sigma'_2\\
u ^{\gamma_{i}} \sigma'_3 & \sigma'_4
\end{pmatrix}$ satisfies:
\bal
&\sigma'_1 - 1 \equiv 0 \mod I_{t+1}, \\
&\sigma'_2 \equiv \begin{cases} 0 &\text{ if } z_{i} \neq 0,\\
\varphi(\sigma_3) &\text{ if } z_{i} = 0,
\end{cases} \mod I_{t+1}, \\
&\sigma'_3 \equiv \begin{cases} 
0 &\text{ if } z_{i} \neq p-1, \\
\varphi(\sigma_2) &\text{ if } z_{i} = p-1,
\end{cases} \mod I_{t+1}, \\
&\sigma'_4 - 1 \equiv 0 \mod I_{t+1}.
\nal
\end{enumerate}
\end{cor}

\begin{proof}
Apply Lemma \ref{inductive-step} with
$P = Id$, $Y = \begin{pmatrix}
\sigma_1 - 1 & u^{e-\gamma_{i}} \sigma_2 \\
u^{\gamma_{i}} \sigma_3 & \sigma_4 - 1
\end{pmatrix}$ and $P' = P + Y$. 
\end{proof}

\begin{proof}[Proof of Proposition \ref{convergence-CDM-matrix}]
We set $P_0 = Id$ and construct $P_s = \begin{psmallmatrix}
\sigma^{(s)}_1& u^{e-\gamma_{s}} \sigma^{(s)}_2\\
u ^{\gamma_{s}} \sigma^{(s)}_3 & \sigma^{(s)}_4
\end{psmallmatrix}$ inductively by letting $\mathcal{B}(F_{s+1} \varphi(P_s)) = P_{s+1}\Delta_{s+1} $, where we choose $\Delta_{s+1}$ to be a diagonal matrix in $\GL_2(R)$ such that the diagonal entries of $P_{s+1}$ are $1$ mod $v$. Here, the indexing set of the Frobenius matrices $F_s$ is extended to all natural numbers via the natural map $\mathbb{Z} \to \mathbb{Z}/f\mathbb{Z}$.

We let $M_s = \mathcal{B}(F_s)^{-1} F_s$ (so that $M_{s+f} = M_s$). Trivially, $P_s$ and $P_{s+f}$ are $0$-close (see Definition \ref{defn-t-close}).
Suppose $P_s$ and $P_{s+f}$ are $t$-close for $t \geq 0$. Let $Y = \begin{psmallmatrix}
y_1 & u^{e-\gamma_{s}}y_2 \\
u^{\gamma_{s}}y_3 & y_4
\end{psmallmatrix}$ be such that $P_{s} = P_{s+f} + Y$. 

We use Lemma \ref{inductive-step} to calculate $\mathcal{B}(M_{s+1} \varphi(P_s)) - \mathcal{B}(M_{s+f+1} \varphi(P_{s+f}))$ mod $I_{t+1}$. Let $\mathcal{B}(M_{s+1} \varphi(P_s)) = \mathcal{B}(M_{s+f+1} \varphi(P_{s+f})) + Y'$ where $Y' = \begin{psmallmatrix}
y'_1 & u^{e-\gamma_{s+1}}y'_2 \\
u^{\gamma_{s+1}}y'_3 & y'_4
\end{psmallmatrix}$. Then $y'_1$ and $y'_4$ are $0$ mod $I_{t+1}$.
Moreover, at least one of $y'_2$ and $y'_3$ is $0$ mod $I_{t+1}$. Either (but not both) of $y'_2$ and $y'_3$ can depend on either (but not both) of $y_2$ and $y_3$ mod $I_{t+1}$ depending on the genre of $F_s$ and the value of $z_{s+1}$. Since $y_2$ and $y_3$ are $0$ mod $I_t$, the same ends up being true for $y'_2$ and $y'_3$. 

Using Lemma \ref{reduction-step}, we have:
\bal
&\mathcal{B}(F_{s+f+1} \varphi(P_{s+f}))^{-1} \mathcal{B}(F_{s+1} \varphi(P_{s}))\\
&= \mathcal{B}(M_{s+1} \varphi(P_{s+f}))^{-1}\mathcal{B}(F_{s+1})^{-1} \mathcal{B}(F_{s+1}) \mathcal{B}(M_{s+1} \varphi(P_{s}))\\
&= \mathcal{B}(M_{s+1} \varphi(P_{s+f}))^{-1}(\mathcal{B}(M_{s+1} \varphi(P_{s+f}) + Y')\\
&= Id + \mathcal{B}(M_{s+1} \varphi(P_{s+f}))^{-1} Y'.
\nal

Since $P_{s+f}$ is $0$-close to $Id$ and the diagonal terms of $P_{s+f}$ are $1$ mod $v$ , we can use the congruences in Corollary \ref{cdm-algo-multiplicative} to conclude that $\mathcal{B}(M_{s+1} \varphi(P_{s+f}))^{-1}$ has the form $Id + \begin{psmallmatrix}
x_1 & u^{e-\gamma_{s+1}}x_2 \\
u^{\gamma_{s+1}} x_3 & x_4 
\end{psmallmatrix}$, with $x_1, x_4 \in I_{1}$. Therefore, $$\mathcal{B}(M_{s+1} \varphi(P_{s+f}))^{-1} Y' = Y' + \begin{pmatrix}
x_1 y'_1 + v x_2 y'_3 & u^{e-\gamma_{s+1}}(x_1 y'_2 + x_2 y'_4)\\
u^{\gamma_{s+1}} (x_3 y'_1 + x_4 y'_3) & v x_3 y'_2 + x_4 y'_4
\end{pmatrix}$$ is $(t+1)$-close to $Y'$. This is evident when $t \leq n$, because in that case, $I_1 I_{t} \subset I_{t+1}$. On the other hand, when $t > n$, $y'_2$ and $y'_3$ are already $0$ mod $I_{t+1}$, and the assertion follows.


This implies that $\mathcal{B}(F_{s+f+1} \varphi(P_{s+f}))$ and $\mathcal{B}(F_{s+1} \varphi(P_{s}))$  have the same diagonal entries mod $I_{t+1}$ and consequently, $\Delta_{s+1} \equiv \Delta_{s+f+1} \mod \mathfrak{m}^{t+1}$. Further,
\bal
P_{s+f+1}^{-1} P_{s+1} &= \Delta_{s+f+1} \mathcal{B}(F_{s+f+1} \varphi(P_{s+f}))^{-1} \mathcal{B}(F_{s+1} \varphi(P_{s})) \Delta_{s+1}^{-1} \\
&=Id + (\Delta_{s+f+1}\Delta_{s+1}^{-1} - Id)  + \Delta_{s+f+1} \mathcal{B}(M_{s+1} \varphi(P_{s+f}))^{-1} Y' \Delta_{s+1}^{-1}.
\nal
The entries of $\Delta_{s+f+1} \mathcal{B}(M_{s+1} \varphi(P_{s+f}))^{-1} Y' \Delta_{s+1}^{-1}$ differ from those of $\mathcal{B}(M_{s+1} \varphi(P_{s+f}))^{-1} Y'$ by some scalars congruent to $1$ mod $I_{t+1}$. As a result, $P_{s+f+1}^{-1} P_{s+1} - Id$ is $(t+1)$-close to $Y'$. Let $P_{s+f+1}^{-1} P_{s+1} - Id = Y' + \begin{psmallmatrix} z_1 & u^{e - \gamma_{s+1}} z_2 \\ u^{\gamma_{s+1}} z_3 & z_4 \end{psmallmatrix}$ with each of $z_1, z_2, z_3, z_4$ congruent to $0 \mod I_{t+1}$.

Let $Y'' = P_{s+1} - P_{s+f+1} = P_{s+f+1}(P_{s+f+1}^{-1} P_{s+1} - Id)$. We claim that $Y''$ is $(t+1)$-close to $Y'$, equivalently to $P_{s+f+1}^{-1} P_{s+1} - Id$. To see this, write 
\begin{align*}
&Y'' - (P_{s+f+1}^{-1} P_{s+1} - Id)
= (P_{s+f+1} - Id)(P_{s+f+1}^{-1} P_{s+1} - Id) \\
&= \begin{pmatrix}
\sigma^{(s+f+1)}_1 - 1& u^{e-\gamma_{s+1}} \sigma^{(s+f+1)}_2\\
u ^{\gamma_{s+1}} \sigma^{(s+f+1)}_3 & \sigma^{(s+f+1)}_4 - 1
\end{pmatrix} \left( \begin{pmatrix}
y'_1 & u^{e-\gamma_{s+1}}y'_2 \\
u^{\gamma_{s+1}}y'_3 & y'_4
\end{pmatrix} +  \begin{pmatrix} z_1 & u^{e - \gamma_{s+1}} z_2 \\ u^{\gamma_{s+1}} z_3 & z_4 \end{pmatrix} \right).
\end{align*}

Since $\sigma^{(s+f+1)}_1 - 1 \equiv \sigma^{(s+f+1)}_4 - 1 \equiv 0 \mod v$, it is immediately verified that $Y''$ is $(t+1)$-close to $Y'$. Thus, $Y'$ measures the difference between $P_{s+1}$ and $P_{s+f+1}$ upto an error term which is $(t+1)$-close to $0$.

We now induct on $s$ and use Lemma \ref{inductive-step} to examine the dependency of $Y'$ on $Y$ (specifically of $y'_2$ and $y'_3$ on $y_2$ and $y_3$ mod $I_{t+1}$), which in turn gives the dependency of $P_{s+1} - P_{s+f+1}$ on 
$P_{s} - P_{s+f}$. It is evident that if $\M$ is not of bad genre (see Definition \ref{define-modified-bad-genre}) and $P_{s+f}$ is $t$-close to $P_s$, then $P_{s+2f}$ is at least $(t+1)$-close to $P_{s+f}$, making $\{P_{s+nf}\}_n$ a Cauchy sequence.

Therefore, we can set $P^{(i)} = \lim\limits_{n \to \infty} P_{i+nf}$ and let $F'_i := (P^{(i)})^{-1} F_{i} \varphi(P^{(i-1)})$. Then $F'_i$ has the following form:

\begin{align}\label{intermediate-CDM-form}
F'_i &= \begin{cases}
\Delta_i \begin{pmatrix}
v& 0\\
A_i u ^{\gamma_{i}}& 1
\end{pmatrix}& \text{if  } \G(F_i) = \Ieta,\\
\hspace{1cm}\\
\Delta_i \begin{pmatrix}
0& u^{e- \gamma_{i}}\\
u^{\gamma_{i}}& A_i'
\end{pmatrix}& \text{if  } \G(F_i) = \II \text{ and } F_i \text{ is in } \eta \text{-form},\\
\hspace{1cm}\\
\Delta_i \begin{pmatrix}
1& A'_i u^{e- \gamma_{i}}\\
0& v
\end{pmatrix}& \text{if  } \G(F_i) = \Ietaa,\\
\hspace{1cm}\\
\Delta_i \begin{pmatrix}
A_i & u^{e- \gamma_{i}}\\
u^{\gamma_{i+1}}& 0
\end{pmatrix}& \text{if  } \G(F_i) = \II \text{ and } F_i \text{ is in } \eta' \text{-form},\\
\end{cases}
\end{align}
where $\Delta_i$ are diagonal matrices with entries in $R^*$ and $A_i, A'_i \in R$. Note that when $\G(F_i) = \II$, the diagonal terms of $F'_i$ are in $\m$.

Now we do one final base change by diagonal scalar matrices $Q_i$ where $Q_0 = Id$ and $Q_1, ..., Q_{f-1}$ are defined inductively so that $G_i = Q_{i}^{-1} F_i Q_{i-1}$ is in CDM form.
\end{proof}

\begin{remark}
Note that the definition of bad genre in Definition \ref{define-modified-bad-genre} is minimal in the sense that for any choice of bad genre, it is easy to write Frobenius matrices that prohibit the convergence of the algorithm in the proof of Proposition \ref{convergence-CDM-matrix}.
\end{remark}

\subsection{Base changes}\label{automorphism subsec}
We continue to assume that our Breuil-Kisin modules are regular (see Definition \ref{regular}). Having classified the Breuil-Kisin modules that make up the stack $\mathcal{C}^{\tau,BT}$, we now need to understand their automorphisms which we know take the form of inertial base changes which need to respect the affermentioned classification. Our first order of business is to understand what these  base changes look like.

\begin{prop}\label{automorphisms}\cite[Prop.~3.1.22]{CDM}
Let $R$ be an Artinian local ring over $\F$ with maximal ideal $\m$. Let $\M$ be a regular Breuil-Kisin module over $R$, not of bad genre so that by Lemma \ref{convergence-CDM-matrix}, the Frobenius matrices $G_i$ can be put in the CDM form with parameters $(\alpha, \alpha', A_0, A'_0, ..., A_{f-1}, A'_{f-1})$. Suppose there exist inertial base change matrices $P_i$, so that if $F_i = P_{i}^{-1}G_i \varphi(P_{i-1})$, then $\{F_i\}_i$ are also in the CDM form with some parameters $(\beta, \beta', B_0, B'_0, ..., B_{f-1}, B'_{f-1})$. Then the following hold true:

\begin{enumerate}
    \item For all $i$, $P_i$ are necessarily of the form: \bal
    P_i = \begin{pmatrix}
    \lambda_i & 0 \\
    0 & \mu_i
    \end{pmatrix} \in \GL_2(R).
    \nal
    \item For $i \in \{1,\dots,  f-1\}$, if $\mathcal{G}(F_{i}) \in \{\Ieta, \Ietaa\}$, then $\lambda_{i} = \lambda_{i-1}$ and $\mu_{i} = \mu_{i-1}$.
    \item For $i \in \{1, \dots, f-1\}$, if $\mathcal{G}(F_{i}) = \II$, then  $\lambda_{i} = \mu_{i-1}$ and $\mu_{i} = \lambda_{i-1}$.
\end{enumerate}
Therefore the parameters $(\beta, \beta', B_0, B'_0, ..., B_{f-1}, B'_{f-1})$ are obtained by suitably scaling the parameters $(\alpha, \alpha', A_0, A'_0, ..., A_{f-1}, A'_{f-1})$ as dictated by the base change matrices.
\end{prop}

The proof of this proposition uses a few more technical lemmas given below.

\begin{lemma}\label{reduction-of-automorphisms-to-CDM-3.1.19-1}
Assume the setting of Proposition \ref{automorphisms}. Let $C_{i}=\mathcal{B}(G_{i} \varphi(P_{i-1}))^{-1}G_{i}\varphi(P_{i-1})$. Then $C_{i} = \Delta_{i}^{-1} F_{i}$ where $\Delta_{i}$ is the identity matrix for $i \in \{1,\dots, f-1\}$ and equals $\begin{psmallmatrix}
\beta & 0 \\
0 & \beta'
\end{psmallmatrix}$, for $i=0$.
\end{lemma}

\begin{proof}
Let $c_i, c'_i$ be such that:
\begin{align}
C_i &= \begin{dcases}
\begin{pmatrix}
v& 0\\
c_i u ^{\gamma_{i}}& 1
\end{pmatrix}& \text{if  } \G(G_i) = \Ieta,\\
\begin{pmatrix}
0& u^{e- \gamma_{i}}\\
u^{\gamma_{i}}& c_i'
\end{pmatrix}& \text{if  } \G(G_i) = \II \text{ and } G_i \text{ is in } \eta \text{-form},\\
\begin{pmatrix}
1& c'_i u^{e- \gamma_{i}}\\
0& v
\end{pmatrix}& \text{if  } \G(G_i) = \Ietaa,\\
\begin{pmatrix}
c_i & u^{e- \gamma_{i}}\\
u^{\gamma_{i}}& 0
\end{pmatrix}& \text{if  } \G(G_i) = \II \text{ and } G_i \text{ is in } \eta' \text{-form}.\\
\end{dcases}
\end{align}
Since $P_{i}^{-1} G_i \varphi(P_{i-1}) = F_i$, we have $P_{i}^{-1} \mathcal{B}(G_i \varphi(P_{i-1}))C_i = F_i$. Inverting $C_i$ in $\GL_2(R(\!(u)\!))$, we obtain that $F_i C_i^{-1} = P_{i}^{-1}\mathcal{B}(G_i \varphi(P_{i-1}))$. Notice that $P_{i}^{-1}\mathcal{B}(G_i \varphi(P_{i-1}))$ is in $\GL_2(R[\![u]\!])$ and therefore all the entries of $F_i C_i^{-1}$ must be in $R[\![u]\!]$.

Now, consider the case where $\G(G_i) = \Ieta$.
\bal
F_i C_i^{-1} &= \Delta_i \begin{pmatrix}
v & 0\\
B_i u^{\gamma_{i}} & 1
\end{pmatrix}
\begin{pmatrix}
v^{-1} & 0 \\
-c_i u^{\gamma_{i} - e} & 1
\end{pmatrix} \\
&= \Delta_i \begin{pmatrix}
1 & 0\\
(B_i - c_i)u^{-e + \gamma_{i}} & 1
\end{pmatrix}
\nal

We conclude that the entries of $F_i C_i^{-1}$ are in $R[\![u]\!]$ if and only if $c_i = B_i$ or in other words, $C_i = \Delta_{i}^{-1} F_i$. The other three cases involve similar computations and conclusions, and are omitted.
\end{proof}

\begin{lemma}\label{reduction-of-automorphisms-to-CDM-3.1.19-2}
Assume the setting of Proposition \ref{automorphisms}. If $i \in \{1,\dots,  f-1\}$, then $P_{i} = \mathcal{B}(G_i \varphi(P_{i-1}))$. Furthermore, $P_0 = \mathcal{B}(G_{0} \varphi(P_{f-1})) \begin{psmallmatrix}
\beta^{-1} & 0 \\
0 & \beta'^{-1}
\end{psmallmatrix}$.
\end{lemma}
\begin{proof}
By Lemma \ref{reduction-of-automorphisms-to-CDM-3.1.19-1}, $P_{i}^{-1}\mathcal{B}(G_i \varphi(P_{i-1}))\Delta_{i}^{-1} F_i = P_{i}^{-1}G_i \varphi(P_{i-1}) =  F_i$. Inverting $F_i$ in $\GL_2(R(\!(u)\!))$, we have $P_{i}^{-1}\mathcal{B}(G_i \varphi(P_{i-1}))\Delta_{i}^{-1} = Id$, and therefore, $P_{i} = \mathcal{B}(G_i \varphi(P_{i-1}))\Delta_{i}^{-1}$.
\end{proof}

\begin{lemma}\label{reduction-of-automorphisms-to-CDM-3.1.19-3}
Assume the setting of Proposition \ref{automorphisms}. Suppose both the diagonal entries of $P_0$ equal $1$ mod $v$. Then $P_i = Id$ for all $i$.
\end{lemma}
\begin{proof}
Suppose that $P_0$ is $t$-close to $Id$ (this is automatically true for $t = 0$ from the hypothesis in the statement of the Lemma). We apply Lemma \ref{inductive-step} successively to compute the congruences for $P_{i} = \mathcal{B}(G_i \varphi(P_{i-1}))$ as $i$ goes from $1$ to $f-1$, and then finally for $\mathcal{B}(\begin{psmallmatrix}
\alpha^{-1} & 0 \\
0 & \alpha'^{-1}
\end{psmallmatrix} G_{0} \varphi(P_{f-1}))$. 

We obtain that
\[
 \mathcal{B}(\begin{pmatrix}
\alpha^{-1} & 0 \\
0 & \alpha'^{-1}
\end{pmatrix} G_{0} \varphi(P_{f-1})) = \begin{pmatrix}
\sigma_1& u^{e-\gamma_{0}} \sigma_2\\
u ^{\gamma_{0}} \sigma_3 & \sigma_4
\end{pmatrix},
\]
where
\bal
&\sigma_1 - 1 \equiv \sigma_2 \equiv \sigma_3 \equiv \sigma_4 - 1 \equiv 0 \mod I_{t+1}.
\nal
By Lemma \ref{reduction-of-automorphisms-to-CDM-3.1.19-2},
\bal
P_0 &= \mathcal{B}(G_{0} \varphi(P_{f-1})) \begin{pmatrix}
\beta^{-1} & 0 \\
0 & \beta'^{-1}
\end{pmatrix} \\
&= \begin{pmatrix}
\alpha & 0 \\
0 & \alpha'
\end{pmatrix} \mathcal{B}(\begin{pmatrix}
\alpha^{-1} & 0 \\
0 & \alpha'^{-1}
\end{pmatrix} G_{0} \varphi(P_{f-1}))\begin{pmatrix}
\beta^{-1} & 0 \\
0 & \beta'^{-1}
\end{pmatrix} &\text{(using Lemma \ref{reduction-step})}\\
&= \begin{pmatrix}
\alpha & 0 \\
0 & \alpha'
\end{pmatrix} \begin{pmatrix}
\sigma_1& u^{e-\gamma_{0}} \sigma_2\\
u ^{\gamma_{0}} \sigma_3 & \sigma_4
\end{pmatrix} \begin{pmatrix}
\beta^{-1} & 0 \\
0 & \beta'^{-1}
\end{pmatrix}.
\nal \\
Recalling that $P_0$ has diagonal entries equal to $1$ mod $v$, we have:
\bal
&\alpha \beta^{-1} \sigma_1  - 1 \equiv \alpha' \beta'^{-1}\sigma_4 -1 \equiv 0 \mod v,\\
&\alpha \beta^{-1} \sigma_1 - 1 \equiv \alpha' \beta'^{-1} \sigma_4 - 1 \equiv 0 \mod I_{t+1}, \\
&\alpha \beta'^{-1}\sigma_2 \equiv \alpha' \beta^{-1}\sigma_3 \equiv 0 \mod I_{t+1}.
\nal
The mod $v$ congruence shows that $\alpha = \beta$ and $\alpha' = \beta'$. Therefore, $P_0$ is $(t+1)$-close to $Id$. Induction on $t$ gives us the desired proof.
\end{proof}

\begin{proof}[Proof of Proposition \ref{automorphisms}]
Suppose the top left entry of $P_0$ is $\lambda_0$ mod $v$, while the bottom right entry is $\mu_0$ mod $v$, where $\lambda_0, \mu_0 \in R^{*}$. Let $Q_i :=  \begin{psmallmatrix}
\lambda_{i}^{-1} & 0 \\
0 & \mu_{i}^{-1}
\end{psmallmatrix}$ where $\lambda_i$ and $\mu_i$ are defined in the following manner for $i \in \{1,\dots,  f-1\}$: If $\mathcal{G}(G_{i}) \in \{\Ieta, \Ietaa\}$, then we let $\lambda_{i} = \lambda_{i-1}$ and $\mu_{i} = \mu_{i-1}$. If $\mathcal{G}(G_{i}) = \II$, we let $\lambda_{i} = \mu_{i-1}$ and $\mu_{i} = \lambda_{i-1}$. To prove the proposition, we must show that $P_i = Q_{i}^{-1}$.

Observe that the matrices $H_{i} = Q_{i}^{-1} F_{i} \varphi(Q_{i-1})$ are still in CDM form (see Definition \ref{defn-CDM-form}). We now consider the base change given by the matrices $P_i Q_{i}$, that transforms $G_i$ to $H_i$. By the choice of $\lambda_0$ and $\mu_0$, the diagonal entries of $P_0 Q_0$ equal $1$ mod $v$. Applying Lemma \ref{reduction-of-automorphisms-to-CDM-3.1.19-3}, we have $P_i Q_i = Id$ for all $i$, and therefore $P_i = Q_{i}^{-1}$. 
\end{proof}

\begin{cor}\label{unique-base-change-upto-scalar}
Let $R$ be an Artinian local $\mathbb{F}$-algebra and let $\M$ be a regular Breuil-Kisin module defined over $R$ and not of bad genre. Suppose $\{F_i\}_i$ and $\{G_i\}_i$ are two sets of Frobenius matrices for $\M$ written with respect to different sets of inertial bases. Then the base change matrices $\{P_i\}_i$ to go from $\{F_i\}_i$ to $\{G_i\}_i$ are unique up to multiplying each of the $P_i$ by a fixed scalar matrix.
\end{cor}
\begin{proof}
Since each set of Frobenius matrices can be transformed into CDM form, it suffices to check the assertion when $\{F_i\}_i$ and $\{G_i\}_i$  are assumed to be in CDM form. From the way the parameters for the Frobenius matrices transform under base change, it is immediate that the base change matrices are uniquely determined up to scalar multiples.
\end{proof}

For the remainder of this section, we will make the following assumption for a Breuil-Kisin module $\M$ defined over $R$.
\begin{ass}\label{eta-form-ass}
$\M$ is a regular Breuil-Kisin module over $R$, not of bad genre. Each of its Frobenius maps is in $\eta$-form, and none are in $\eta'$-form.
\end{ass}

The assumption is justified because allowing some Frobenius matrices to be in $\eta'$-form will offer very little advantage in our eventual conclusions but inundate the text with significantly more notation - a discussion of the effect of allowing some Frobenius matrices to be in $\eta'$-form is in the Appendix.
\\

Via Proposition \ref{convergence-CDM-matrix}, we can now describe Frobenius maps very parsimoniously using matrices in CDM form. Base changes between CDM forms also have an easy description using Proposition \ref{automorphisms}. This bring us one step closer to finding a finite presentation of the stack of Breuil-Kisin modules. We now turn our attention to furthering this process, specifically to understanding the base changes that allowed us to write the Frobenius matrices in the CDM form. Specifically, we will be studying the matrices $P^{(i)} = \lim \limits_{n \to \infty} P_{i+nf}$ showing up in the proof of Proposition \ref{convergence-CDM-matrix}. We will also analyze obstructions to a parsimonious description, one of which we have already seen show up as a 'bad genre' condition. We have seen that $\M$ can be of bad genre only if the infinite sequence $(z_i)_{i\in \mathbb{Z}}$ is made up entirely of the building blocks $1$ and $(0, p-1)$. On the other hand, if $(z_i)_{i\in \mathbb{Z}}$ is such, we can find an $\M$ of bad genre by choosing the entries of the Frobenius matrices suitably. This motivates the following definition.

\begin{defn}\label{defn-first-obstruction}
We say that a tame principal series $\F$-type $\tau$ faces the first obstruction if $(z_i)_{i\in \mathbb{Z}}$ is made up entirely of the building blocks $1$ and $(0, p-1)$.
\end{defn}

\begin{prop}\label{unipotent-action-gives-CDM-form-I-eta}
Let $R$ be an Artinian local ring over $\F$ with maximal ideal $\m$ and let $\M$ be a Breuil-Kisin module over $R$ satisfying Assumption \ref{eta-form-ass}. Suppose with respect to an inertial basis, $F_i$ has the form $$\begin{pmatrix}
v a_i & u^{e - \gamma_{i}}b_i \\
u^{\gamma_{i}}c_i & d_i
\end{pmatrix}$$ with $a_i, b_i, c_i, d_i \in R$. Let $P^{(j)} = \lim\limits_{n \to \infty} P_{j+nf}$ denote the base change matrices described in the proof of Proposition \ref{convergence-CDM-matrix} and let  $$F'_i = (P^{(i)})^{-1} F_i \varphi(P_{i-1})=\begin{pmatrix}
v a'_i & b'_i u^{e - \gamma_{i}}\\
c'_i u^{\gamma_{i}}& d'_i
\end{pmatrix}$$ be the matrix in (\ref{intermediate-CDM-form}). Define a left action of upper unipotent matrices on $\eta$-form Frobenius matrices in the following manner:

\[
\begin{pmatrix}
1 & y \\
0 & 1
\end{pmatrix} \star \begin{pmatrix}
v a & u^{e - \gamma}b \\
u^{\gamma}c & d
\end{pmatrix} = \begin{pmatrix}
v (a + y c) & u^{e - \gamma}(b + y d) \\
u^{\gamma}c & d
\end{pmatrix}.
\]
The following statements are true:
\begin{enumerate}
    \item Suppose $(z_{i})_{i \in \mathbb{Z}}$ does not contain the subsequence $(p-1, 1, ..., 1, 0)$ (where the number of $1$'s is allowed to be zero). Then there exists an upper unipotent $U_i$ for each $i$ satisfying $U_i \star F_i = F'_i$.
\item If $(z_{i})_i$ contains the subsequence $(p-1, 1, ..., 1, 0)$, then there exists a set of Frobenius matrices $\{F_i\}$ such that, for some $l$, no unipotent matrix $U$ satisfies $U \star F_l = F'_l$.
\end{enumerate}
\end{prop}
The proof will use the following lemma.
\begin{lemma}\label{Explicit-CDM-form-I-eta}
Consider the setup of Proposition \ref{unipotent-action-gives-CDM-form-I-eta}. Suppose that the base change matrices $P^{(j)}$ are given by $$\begin{pmatrix}
q_j & u^{e - \gamma_{j}} r_j\\
u^{\gamma_{j}}s_j& t_j
\end{pmatrix}.$$ For any $\sigma \in R[\![v]\!]$, denote by $\overline{\sigma}$ the constant part of $\sigma$. Then 
\bal
F'_i &= \begin{dcases}
Ad \begin{pmatrix}
1 & 0\\
0 & \frac{c_i + d_i \overline{s_{i-1}}}{c_i}
\end{pmatrix} \left(\begin{pmatrix}
(a_i - \frac{c_i}{d_i}b_i) v& 0\\
u^{\gamma_{i}} c_i & d_i
\end{pmatrix}\right) &\text{ if } \mathcal{G}(F_i) = \Ieta, z_{i} = 0 ,\\
    \begin{pmatrix}
(a_i - \frac{c_i}{d_i}b_i) v& 0\\
u^{\gamma_{i}} c_i & d_i
\end{pmatrix} &\text{ if } \mathcal{G}(F_i) = \Ieta, z_{i} \neq 0 ,\\
 Ad \begin{pmatrix}
1 & 0\\
0 & \frac{c_i + d_i \overline{s_{i-1}}}{c_i}
\end{pmatrix} \left(\begin{pmatrix}
0& u^{e-\gamma_{i}}(b_i - \frac{d_i}{c_i} a_i)\\
u^{\gamma_{i}}c_i& d_i 
\end{pmatrix}\right) &\text{ if } \mathcal{G}(F_i) = \II, z_{i} = 0 ,\\
 \begin{pmatrix}
0& u^{e-\gamma_{i}}(b_i - \frac{d_i}{c_i} a_i)\\
u^{\gamma_{i}}c_i& d_i 
\end{pmatrix} &\text{ if } \mathcal{G}(F_i) = \II, z_{i} \neq 0 ,
\end{dcases}
\nal
where $Ad \: M \: (N)$ denotes the matrix $MNM^{-1}$.
\end{lemma}
\begin{proof}
Using the definition of the operator $\mathcal{B}$ in Definition \ref{defn-B-operation} and our calculations in Lemma \ref{reduction-step}, we have
\bal
    &\mathcal{B}(F_i)^{-1} F_i = M_i = \begin{cases}
    \vspace{0.4cm}
    \begin{pmatrix}
v& 0\\
u^{\gamma_{i}} \frac{c_i}{d_i} & 1
\end{pmatrix} &\text{ if } \mathcal{G}(F_i) = \Ieta,\\
\vspace{0.4cm}
    \begin{pmatrix}
0& u^{e-\gamma_{i}}\\
u^{\gamma_{i}}& \frac{d_i}{c_i} 
\end{pmatrix} &\text{ if } \mathcal{G}(F_i) = \II,
    \end{cases}\\
\nal
\bal
&\mathcal{B}(F_i) = 
    \begin{dcases}
    \begin{pmatrix}
a_i - \frac{c_i}{d_i}b_i & u^{e - \gamma_{i}} b_i \\
0& d_i
\end{pmatrix} 
&\text{if } \mathcal{G}(F_i) = \Ieta, \\
    \begin{pmatrix}
b_i - \frac{d_i}{c_i} a_i & u^{e - \gamma_{i}} a_i\\
0 & c_i 
\end{pmatrix} 
&\text{if } \mathcal{G}(F_i) = \II,
    \end{dcases}\\
\nal
and 
\bal
&\mathcal{B}(M_i \varphi(P^{(i-1)}))^{-1} M_i \varphi(P^{(i-1)})=
    \begin{cases}
    \vspace{0.4cm}
    \begin{pmatrix}
v& 0\\
u^{\gamma_{i}} (\frac{c_i}{d_i} + \overline{s_{i-1}}) & 1
\end{pmatrix} &\text{ if } \mathcal{G}(F_i) = \Ieta, z_{i} = 0,\\
    \vspace{0.4cm}
    \begin{pmatrix}
v& 0\\
u^{\gamma_{i}} \frac{c_i}{d_i} & 1
\end{pmatrix} &\text{ if } \mathcal{G}(F_i) = \Ieta, z_{i} \neq 0,\\
\vspace{0.4cm}
    \begin{pmatrix}
0& u^{e-\gamma_{i}}\\
u^{\gamma_{i}}& \frac{\frac{d_i}{c_i}}{1+\frac{d_i}{c_i} \overline{s_{i-1}}} 
\end{pmatrix} &\text{ if } \mathcal{G}(F_i) = \II, z_{i} = 0,\\
    \begin{pmatrix}
0& u^{e-\gamma_{i}}\\
u^{\gamma_{i}}& \frac{d_i}{c_i} 
\end{pmatrix} &\text{ if } \mathcal{G}(F_i) = \II, z_{i} \neq 0. \end{cases}
\nal

Recall that by Lemma \ref{reduction-step}, $\mathcal{B}(F_i \varphi(P^{(i-1)})) = \mathcal{B}(F_i)\mathcal{B}(M_i \varphi(P^{(i-1)}))$  and by the calculations in Lemma \ref{inductive-step}, $\mathcal{B}(M_i \varphi(P^{(i-1)}))$ is $Id$ mod $u$ if $\mathcal{G}(F_i) \neq \II$ or if $\mathcal{G}(F_i) = \II$ but $z_{i} \neq 0$. By the algorithm in the proof of Proposition \ref{convergence-CDM-matrix}, we find that $\mathcal{B}(F_i \varphi(P^{(i-1)})) = P^{(i)}\Delta^{(i)}$ for a suitable diagonal scalar matrix $\Delta^{(i)}$ chosen such that the diagonal entries of $P^{(i)}$ are $1$ mod $v$ or, in other words, such that $P^{(i)}$ is $Id$ mod $u$. Therefore if $\mathcal{G}(F_i) \neq \II$ or if $\mathcal{G}(F_i) = \II$ but $z_{i} \neq 0$, $\Delta^{(i)} \equiv \mathcal{B}(F_i) \mod u$. If $\mathcal{G}(F_i) = \II$ and $z_{i} = 0$, Lemma \ref{inductive-step} gives us the following equivalence mod $u$:

\bal
\Delta^{(i)} &\equiv \mathcal{B}(F_i) \begin{pmatrix}
1 - \frac{\frac{d_i}{c_i} \overline{s_{i-1}}}{1 + \frac{d_i}{c_i} \overline{s_{i-1}}} & 0 \\
0 & 1 + \frac{d_i}{c_i} \overline{s_{i-1}}
\end{pmatrix}\\
&\equiv \mathcal{B}(F_i) \begin{pmatrix}
\frac{c_i}{c_i + d_i \overline{s_{i-1}}} & 0 \\
0 & \frac{c_i + d_i \overline{s_{i-1}}}{c_i} 
\end{pmatrix}.
\nal

Letting $D_i = a_i d_i - b_i c_i$,
\bal
\Delta^{(i)} = 
    \begin{cases}
\vspace{0.3cm}
    \begin{pmatrix}
\frac{D_i}{d_i} & 0 \\
0& d_i
\end{pmatrix} 
&\text{if } \mathcal{G}(F_i) = \Ieta , \\
\vspace{0.3cm}
    \begin{pmatrix}
\frac{-D_i}{c_i + d_i \overline{s_{i-1}}}& 0\\
0 & c_i + d_i \overline{s_{i-1}}
\end{pmatrix} 
&\text{if } \mathcal{G}(F_i) = \II, z_{i} = 0 ,\\
    \begin{pmatrix}
\frac{-D_i}{c_i} & 0\\
0 & c_i 
\end{pmatrix} 
&\text{if } \mathcal{G}(F_i) = \II, z_{i} \neq 0 .
    \end{cases}
\nal
Now we compute $F'_i$:
\bal
F'_i &= (P^{(i)})^{-1} F_i \varphi(P^{(i-1)}) \\
&= \Delta^{(i)} \mathcal{B}(F_i \varphi(P^{(i-1)}))^{-1}F_i \varphi(P^{(i-1)}) \\
&= \Delta^{(i)} \mathcal{B}(M_i \varphi(P^{(i-1)}))^{-1}M_i \varphi(P^{(i-1)}) \\
&= \begin{cases}
\begin{pmatrix}
D_i/d_i & 0 \\
0& d_i
\end{pmatrix} \begin{pmatrix}
v& 0\\
u^{\gamma_{i}} (\frac{c_i + d_i \overline{s_{i-1}}}{d_i}) & 1
\end{pmatrix} &\text{ if } \mathcal{G}(F_i) = \Ieta, z_{i} = 0 ,\\
\hspace{1cm}\\
    \begin{pmatrix}
D_i/d_i & 0 \\
0& d_i
\end{pmatrix} \begin{pmatrix}
v& 0\\
u^{\gamma_{i}} \frac{c_i}{d_i} & 1
\end{pmatrix} &\text{ if } \mathcal{G}(F_i) = \Ieta, z_{i} \neq 0 ,\\
\hspace{1cm}\\
   \begin{pmatrix}
\frac{-D_i}{c_i + d_i \overline{s_{i-1}}}& 0\\
0 & c_i + d_i \overline{s_{i-1}}
\end{pmatrix} \begin{pmatrix}
0& u^{e-\gamma_{i}}\\
u^{\gamma_{i}}& \frac{d_i}{c_i + d_i \overline{s_{i-1}}} 
\end{pmatrix} &\text{ if } \mathcal{G}(F_i) = \II, z_{i} = 0 ,\\
\hspace{1cm}\\
\begin{pmatrix}
-D_i/c_i & 0\\
0 & c_i 
\end{pmatrix} \begin{pmatrix}
0& u^{e-\gamma_{i}}\\
u^{\gamma_{i}}& \frac{d_i}{c_i} 
\end{pmatrix} &\text{ if } \mathcal{G}(F_i) = \II, z_{i} \neq 0 ,
\end{cases}\\
\hspace{2cm}\\
&= \begin{cases}
Ad \begin{pmatrix}
1 & 0\\
0 & \frac{c_i + d_i \overline{s_{i-1}}}{c_i}
\end{pmatrix} \left(\begin{pmatrix}
(a_i - \frac{c_i}{d_i}b_i) v& 0\\
u^{\gamma_{i}} c_i & d_i
\end{pmatrix}\right) &\text{ if } \mathcal{G}(F_i) = \Ieta, z_{i} = 0 ,\\
\hspace{1cm}\\
    \begin{pmatrix}
(a_i - \frac{c_i}{d_i}b_i) v& 0\\
u^{\gamma_{i}} c_i & d_i
\end{pmatrix} &\text{ if } \mathcal{G}(F_i) = \Ieta, z_{i} \neq 0 ,\\
\hspace{1cm}\\
 Ad \begin{pmatrix}
1 & 0\\
0 & \frac{c_i + d_i \overline{s_{i-1}}}{c_i}
\end{pmatrix} \left(\begin{pmatrix}
0& u^{e-\gamma_{i}}(b_i - \frac{d_i}{c_i} a_i)\\
u^{\gamma_{i}}c_i& d_i 
\end{pmatrix}\right) &\text{ if } \mathcal{G}(F_i) = \II, z_{i} = 0, \\
\hspace{1cm}\\
 \begin{pmatrix}
0& u^{e-\gamma_{i}}(b_i - \frac{d_i}{c_i} a_i)\\
u^{\gamma_{i}}c_i& d_i 
\end{pmatrix} &\text{ if } \mathcal{G}(F_i) = \II, z_{i} \neq 0 .
\end{cases}
\nal
\end{proof}

\begin{proof}[Proof of Proposition \ref{unipotent-action-gives-CDM-form-I-eta}]
By Lemma \ref{Explicit-CDM-form-I-eta}, $F'_i$ can be obtained via left unipotent action whenever $z_{i} \neq 0$. If $z_{i} = 0$, then $F'_i$ can be obtained via left unipotent action if and only if $s_i \equiv 0$ mod $v$.

Now, suppose $z_{i} = 0$ and $s_{i-1} \not\equiv 0$ mod $v$. Recall that $P^{(i-1)} = \mathcal{B}(F_{i-1} \varphi(P^{(i-2)})) (\Delta^{i-1})^{-1} = \mathcal{B}(F_{i-1}) \mathcal{B}(M_{i-1} \varphi(P^{(i-2)})) (\Delta^{i-1})^{-1}$.

By the explicit calculations in Lemma \ref{Explicit-CDM-form-I-eta}, $\mathcal{B}(F_{i-1})$ is upper triangular. Therefore, $s_{i-1} \not\equiv 0$ if and only if $\mathcal{B}(M_{i-1} \varphi(P^{(i-2)}))$ is not upper triangular mod $u^eR[\![u]\!]$. By the calculations in Lemma 
\ref{inductive-step}, this can happen only if one of the following two statements holds:

\begin{enumerate}
    \item $z_{i-1} = 1$ and $s_{i-2} \not\equiv 0$ mod $v$. In this situation, $s_{i-1}$ is a multiple of $s_{i-2}$ mod $v$.
    \item $z_{i-1} = p-1$ and $r_{i-2} \not\equiv 0$ mod $v$. In this situation, $s_{i-1}$ is a multiple of $r_{i-2}$ mod $v$.
\end{enumerate}

Going backward, we conclude that $z_{i} = 0$ and $s_{i-1} \not\equiv 0$ can happen only if $z_{i}$ is preceded by a subsequence $(z_{i-k-1}, z_{i-k}, ..., z_{i-1}) = (p-1, 1, ..., 1)$ with $r_{i-k-2}, s_{i-k-1}, ..., s_{i-2} \not\equiv 0$ mod $v$. In other words, if $(z_i)_i$ does not contain a contiguous subsequence of the form $(p-1, 1, ..., 1, 0)$, we can always obtain $F'_i$ via a left unipotent action on $F_i$.

On the other hand, if there exist $k \geq 0$ and $i \in \mathbb{Z}$ such that $(z_{i-k-1}, z_{i-k}, ..., z_{i-1}, z_{i}) = (p-1, 1, ..., 1, 0)$, we may choose $F_j$'s so that $\mathcal{G}(F_j) =\Ieta$ for all $j$. Choose $F_{i-k-2}$ so that $b_{i-k-2} \neq 0$ and $F_{i-k-1}$ so that $c_{i-k-1}$ is a unit. By Lemma \ref{inductive-step}, $\mathcal{B}(M_{i-k-2} \varphi(P^{(i-k-3)}))$ must be lower triangular mod $u^eR[\![u]\!]$. Therefore, $r_{i-k-2}$ is a unit multiple of $b_{i-k-2}$ mod $v$. In turn, $s_{i-k-1}$ is a unit times $c_{i-k-1}^2$ times $r_{i-k-2}$ mod $v$. Inductively, we see that $s_{i-1}$ is a unit times $b_{i-k-2}$ mod $v$, and therefore, non-zero mod $v$. Thus, no unipotent action can give $F'_i$ from $F_i$.
\end{proof}

Proposition \ref{unipotent-action-gives-CDM-form-I-eta} motivates the following definition.
\begin{defn}\label{defn-second-obstruction}
We say that a tame principal series $\F$-type $\tau$ faces the second obstruction if $(z_i)_{i\in \mathbb{Z}}$ contains a contiguous subsequence $(p-1, 1, ..., 1, 0)$ of length $\geq 2$, with the number of $1$'s allowed to be zero.
\end{defn}

Our next step is to analyze when left unipotent action of the type described in Proposition \ref{unipotent-action-gives-CDM-form-I-eta} can be functorially associated to inertial base change data. The eventual goal is to quotient the data of Frobenius matrices by unipotent action, and encode that as a point of the stack of Breuil-Kisin modules. In particular, the unipotent action will be encoded as base change data.

For each $i$, let $(e_i, f_i)$ be an inertial basis of $\M_i$. The $\eta'$-eigenspace of $\mathfrak{M}_i$ is a free module over $R[\![v]\!]$ with an ordered basis given by $(u^{e-\gamma_i} e_i, f_i)$. The $\eta'$-eigenspace of $\varphi^{*} \mathfrak{M}_i$ is a free module over $R[\![v]\!]$ with an ordered basis given by $(u^{e- \gamma_{i+1}} \otimes e_i, 1 \otimes f_i)$. Written with respect to our choice of inertial bases, let the $i$-th Frobenius matrix be given as follows:
\begin{align*} F_i = 
    \begin{pmatrix}
    a_i & u^{e-\gamma_{i}} b_i  \\
    u^{\gamma_{i}} c_i & d_i
    \end{pmatrix}.
\end{align*}

Let $\{P_i\}_i$ be a set of inertial base change matrices, where
\begin{align*} P_{i} = 
    \begin{pmatrix}
    q_i & u^{e-\gamma_i} r_i   \\
    u^{\gamma_i} s_i & t_i
    \end{pmatrix}.
\end{align*}

The Frobenius map $F_i$, when restricted to the $\eta'$-eigenspace part and written with respect to the ordered $\eta'$-eigenspace basis of $\varphi^{*} \mathfrak{M}_{i-1}$ and $\mathfrak{M}_{i}$ has the following matrix:
\begin{align}\label{eta'-frob} G_{i} = 
    \begin{pmatrix}
    a_i & b_i  \\
    v c_i & d_i
    \end{pmatrix}.
\end{align}

Base change of $G_{i}$ is given by:
\begin{align}\label{eta'-base-change}
    J_{i}^{-1} G_i \left( \text{Ad} \begin{pmatrix}
    v^{p-1-z_{i}} & 0 \\
    0 & 1
    \end{pmatrix} (\varphi(J_{i-1})) \right),
\end{align}
where the matrices $J_{i}$ are defined as follows:
\begin{align*}
    J_{i} = 
    \begin{pmatrix}
    q_i & r_i   \\
    v s_i & t_i
    \end{pmatrix}.
\end{align*}

\begin{defn}\label{eigenspace-frobenii}
    When a choice of an inertial basis for each $i$ is understood, $(G_i)_i$ and $(J_i)_i$ as above will be called the Frobenius and base change matrices (respectively) for the $\eta'$-eigenspace. 
    
    We say that the $G_i$'s are in CDM form if the $F_i$'s, which are the matrices for the unrestricted Frobenius maps, are in CDM form (see Definition \ref{defn-CDM-form}).
\end{defn}
It is clear that knowing the data of Frobenius and base change on the $\eta'$-eigenspace part is equivalent to knowing it for the entire Breuil-Kisin module.

\begin{prop}\label{straightening}
Fix an inertial basis $(e_i, f_i)$ for each $i$. Suppose that each $F_i$ is of the form $$\begin{pmatrix}
v a_i& u^{e - \gamma_{i}}b_i \\
u^{\gamma_{i}}c_i & d_i
\end{pmatrix}$$ with $a_i, b_i, c_i, d_i \in R$. For each $i$, denote by $G_i$ the Frobenius matrices for restriction to $\eta'$-eigenspaces so that $G_i = \begin{psmallmatrix}
    v a_i & b_i  \\
    v c_i & d_i
    \end{psmallmatrix}$. Let $U_i = \begin{psmallmatrix}
    1 & y_i \\
    0 & 1
    \end{psmallmatrix}$ for each $i \in \mathbb{Z}/f\mathbb{Z}$.

Then, whenever $\tau$ does not face the second obstruction (Definition \ref{defn-second-obstruction}), there exists a functorially constructed inertial-base-change matrices, in $(U_i)_i$, given by $$P_i = \begin{pmatrix}
q_i & u^{e-\gamma_i} r_i   \\
u^{\gamma_i} s_i & t_i
\end{pmatrix}$$
satisfying $F_i = P_{i}^{-1} (U_i \star F_i)\varphi(P_{i-1})$ where $U_i \star F_i$ is as defined in Proposition \ref{unipotent-action-gives-CDM-form-I-eta}.
Equivalently, 
 \begin{align*}
    G_i = J_{i}^{-1} U_i G_i \left( \text{Ad} \begin{pmatrix}
    v^{p-1-z_{i}} & 0 \\
    0 & 1
    \end{pmatrix} (\varphi(J_{i-1})) \right)
\end{align*} where $J_i = \begin{psmallmatrix}
    q_i & r_i   \\
    v s_i & t_i
    \end{psmallmatrix}$.

\end{prop}

\begin{proof}
We will build $J_i$ as a $v$-adic limit of a sequence $J_{i}^{(n)}$. First, let $J_{i}^{(0)}$ be the identity matrix and define $J_{i+1}^{(n+1)}$ to be
\begin{align*}
    J_{i+1}^{(n+1)} = U_{i+1} G_{i+1} \left( \text{Ad} \begin{pmatrix}
    v^{p-1-z_{i+1}} & 0 \\
    0 & 1
    \end{pmatrix} (\varphi(J_i^{(n)})) \right) G_{i+1}^{-1},
\end{align*} where we are inverting $G_{i+1}$ in $\GL_2(R(\!(u)\!))$.
Therefore,
\begin{align*}
    J_{i+1}^{(n+1)} - J_{i+1}^{(n)} &= U_{i+1} G_{i+1} \left( \text{Ad} \begin{pmatrix}
    v^{p-1-z_{i+1}} & 0 \\
    0 & 1
    \end{pmatrix} (\varphi(J_{i}^{(n)} - J_{i}^{(n-1)})) \right) G_{i+1}^{-1}.
\end{align*}
Let $D_i = a_i d_i - b_i c_i$. Evidently, $J_{i+1}^{(1)} = U_{i+1}$. Further, \\
$$J_{i}^{(2)} - J_{i}^{(1)} = 
v^{p-z_{i}} \frac{y_{i-1}}{D_{i}}\begin{pmatrix}
-a_{i} c_{i} - y_{i}c_{i}^{2} & a_{i}^2 + y_{i}a_{i} c_{i}\\
-c_{i}^{2} & a_{i} c_{i}
\end{pmatrix}.$$
For $X \in M_2(R[\![v]\!])$, denote by $\text{val}_v(X)$ the highest power of $v$ that divides $X$. Let $\alpha_i = \text{val}_v(J_{i}^{(2)} - J_{i}^{(1)})$. Then $\alpha_i \geq p-z_{i}$. 

Now we compute the dependence of the valuation of $J_{i+1}^{(n)} - J_{i+1}^{(n-1)}$ on $J_{i}^{(n)} - J_{i}^{(n-1)}$.

 If $v^r$ divides $J_{i}^{(n)} - J_{i}^{(n-1)}$, then $v^{pr - (p-1-z_{i+1})}$ divides $\text{Ad} \begin{psmallmatrix}
    v^{p-1-z_{i+1}} & 0 \\
    0 & 1    \end{psmallmatrix} \varphi(J_{i}^{(n)} - J_i^{(n-1)})$. After taking into account an extra factor of $v$ coming from the determinant of $G_{i+1}$ which we will need to divide by when inverting $G_{i+1}$, we conclude that $v^{pr - (p-z_{i+1})} = v^{p(r-1) + z_{i+1}}$ divides $J_{i+1}^{(n+1)} - J_{i+1}^{(n)}$. 

Therefore,
\bal
\text{val}_v (J_{i+1}^{(3)} - J_{i+1}^{(2)}) \quad &\geq p(\alpha_i - 1) + z_{i+1}, \\
\text{val}_v (J_{i+1}^{(4)} - J_{i+1}^{(3)}) \quad &\geq p^2(\alpha_{i-1} - 1) + (pz_{i} + z_{i+1}) - p,\\
\text{val}_v (J_{i+1}^{(5)} - J_{i+1}^{(4)}) \quad &\geq p^3(\alpha_{i-2} - 1) + (p^2 z_{i-1} + pz_{i} + z_{i+1}) - (p^2 + p),\\
...\\
\text{val}_v (J_{i+1}^{(n)} - J_{i+1}^{(n-1)}) &\geq p^{n-2}(\alpha_{i-(n-3)} - 1) + \sum_{j=1}^{n-3}p^j (z_{i-(j-1)} - 1) + z_{i+1} \\
&\geq p^{n-2}(p-1-z_{i-(n-3)})+ \sum_{j=1}^{n-3}p^j (z_{i-(j-1)} - 1) + z_{i+1}.
\nal

We have the following scenarios:
\begin{itemize}
    \item Suppose $z_{i-(n-3)} < p-1$. Let $m := \lfloor \frac{n-3}{f} \rfloor$. Then
\bal
\text{val}_v (J_{i+1}^{(n)} - J_{i+1}^{(n-1)}) & \geq p^{n-2} + \sum_{j=1}^{n-3}p^j (z_{i-j+1}- 1) + z_{i+1}\\
&= p + \sum_{j=1}^{n-3}(p-1)p^j + \sum_{j=1}^{n-3}p^j (z_{i-(j-1)}- 1) + z_{i+1}\\
&= p + \sum_{j=1}^{n-3}(p-2)p^j + \sum_{j=0}^{n-3}p^j z_{i+1-j}\\
&\geq \sum_{j=1}^{n-3}(p-2)p^j + \sum_{k=0}^{m-1}p^k \gamma_{i+1} &\text{ (using (\ref{defn-z_i}))}\\
&> p^{m-1}.
\nal
\item Suppose $z_{i-(n-3)} = p-1$ and $z_j \neq 0$ for each $j$. Then
\bal
\text{val}_v (J_{i+1}^{(n)} - J_{i+1}^{(n-1)}) &\geq \sum_{j=1}^{n-3}p^j (z_{i-(j-1)} - 1) + z_{i+1} \\
&\geq \sum_{\substack{j \in [1, n-3] \text{ and } \\
j-1 \equiv n-3 \mod f}}p^j \\
&\geq p^{n-2-f}.
\nal
The second to last step uses $p>2$.
\\
\item Suppose $z_{i-(n-3)} = p-1$ and there exists a $k \in [0, n-3]$ such that $z_{i-(k-1)} = 0$. Take $k$ to be as large as possible. As $\{z_{j}\}_j$ is $f$-periodic, $k \in [n-2-f, n-3]$. Since $\tau$ does not face the second obstruction, there exists a largest possible $l \in (k, n-3)$ such that $z_{i-(l-1)} > 1$. Then
\bal
\text{val}_v (J_{i+1}^{(n)} - J_{i+1}^{(n-1)}) &\geq \sum_{j=k+1}^{l}p^j (z_{i-(j-1)} - 1) + \sum_{j=1}^{k} (z_{i-(j-1)} - 1) + z_{i+1} \\
&> p^{k+1} - \sum_{j=1}^{k} p^j \\
&= p + \sum_{j=1}^k (p-1)p^j - \sum_{j=1}^{k} p^j \\
&= p + \sum_{j=1}^k (p-2)p^j \\
&> p^k\\
&\geq p^{n-2-f}.
\nal
The second to last step uses that $p>2$.
\end{itemize}

The above calculations show that whenever $\tau$ does not face the second obstruction, $(J_{i+1}^{(n)})_n$ is a Cauchy sequence for all $i$.


We set $J_{i} = \lim_{n \to \infty} (J_{i}^{(n)})_n$, and construct the base change matrices $P_{i}$ using the data of $J_i$. Since $J_{i}^{(1)} \in M_2(R[\![v]\!])$ and $\text{val}_v (J_{i}^{(n)} - J_{i}^{(n-1)}) \geq 0$ for each $i$ and $n \geq 2$, $J_{i} \in M_2(R[\![v]\!])$. Since each $G_i$ is invertable in $\GL_2(M(\!(v)\!)$ then we may repeat the argument for $J_i^{-1}$ to see that $J_i\in\GL_2(M[\![v]!])$. Since $v$ divides the upper and lower left entries of $G_i$, it can be shown by direct computation that the lower left entry of $J_{i}$ is $0$ mod $v$. Therefore, $P_{i} \in \GL_2(R[\![u]\!])$ for each $i$.
\end{proof} 

\begin{defn}\label{mathcal-F}
    Denote the inverses of $J_i$'s constructed in Proposition \ref{straightening} by $J_i^{-1} = \mathcal{F}_i (\mathbf{U})$ to indicate the functorial dependence on the tuple of unipotent matrices $\mathbf{U} = (U_j)_j$. Then $(\mathcal{F}_i (\mathbf{U}))_i$ capture the base change data to go from $(G_i)_i \to (U_iG_i)_i$.
\end{defn}


\section{A component of \texorpdfstring{$\mathcal{C}^{\tau, \mathrm{BT}}$}{} as a quotient of a scheme}\label{component sec}


At this point, via Proposition \ref{convergence-CDM-matrix}, we have an easy way of describing the Frobenius maps for certain Breuil-Kisin modules by writing the matrices in CDM form (see Definition \ref{defn-CDM-form}). We also have a complete description of base changes between such Frobenius matrices in Proposition \ref{automorphisms}. Finally, in some cases, we have a way of obtaining Frobenius matrices in CDM form through a particular group action (see Proposition \ref{unipotent-action-gives-CDM-form-I-eta}). The goal of this section is to use these results to write a certain irreducible component of $\mathcal{C}^{\tau, \mathrm{BT}}$ (Definition \ref{defn-C-BT}) as a quotient stack $[X/G]$ for some scheme $X$ and group scheme $G$ acting on $X$. We will use this presentation to compute global functions on the component.

In order to allow us to use Propositions \ref{convergence-CDM-matrix}, \ref{automorphisms} and \ref{unipotent-action-gives-CDM-form-I-eta}, we make the following assumption for the entirety of this section. 
\begin{ass}\label{tau}
The tame principal series $\F$-type $\tau = \eta \oplus \eta'$ satisfies:
\begin{itemize}
    \item $\eta \neq \eta'$, and
    \item $\tau$ does not face either the first obstruction (in the sense of Definition \ref{defn-first-obstruction}) or the second obstruction (in the sense of Definition \ref{defn-second-obstruction}).
\end{itemize}
\end{ass}

\subsection{A smooth map from a scheme to \texorpdfstring{$\mathcal{C}$}{}}\label{smooth sec}



Let $G = (\mathbb{G}_{m})^{f+1}_{\mathbb{F}} \times_{\mathbb{F}} U_{\mathbb{F}}^{f}$ and $X = (\GL_{2})_{\mathbb{F}} \times_{\mathbb{F}} (\SL_{2})_{\mathbb{F}}^{f-1}$, where $U \cong \mathbb{G}_a$ is the upper unipotent subgroup of $\GL_2$. Define a $G$-action on $X$ in the following way: 

Let $(\lambda, \mu, r_1, r_2, ..., r_{f-1}, m_0, ..., m_{f-1}) \in G$ and $ (A_0, ..., A_{f-1}) \in X$. Then
\begin{align}\label{gaction}
&(\lambda, \mu, r_1, r_2, ..., r_{f-1}, m_0, ..., m_{f-1}) \boldsymbol{\cdot} (A_0, ..., A_{f-1}) :=\\
&\Bigg(\begin{pmatrix} \lambda^{-1} & 0 \\ 0 & \mu^{-1}
\end{pmatrix}
m_{0} A_{0} \begin{pmatrix}
r_{f-1}^{-1} & 0 \\ 0 & r_{f-1}
\end{pmatrix} \begin{pmatrix}
\lambda & 0 \\ 0 & \mu
\end{pmatrix}, \nonumber\\
&\begin{pmatrix} \lambda^{-1} & 0 \\ 0 & \mu^{-1}
\end{pmatrix}\begin{pmatrix}
r_1 & 0 \\ 0 & r_1^{-1}
\end{pmatrix}m_1 A_1 \begin{pmatrix}
\lambda & 0 \\ 0 & \mu
\end{pmatrix} , \nonumber\\
&\begin{pmatrix} \lambda^{-1} & 0 \\ 0 & \mu^{-1}
\end{pmatrix}
\begin{pmatrix}
r_2 & 0 \\ 0 & r_2^{-1}
\end{pmatrix}
m_2 A_2 \begin{pmatrix}
r_1^{-1} & 0 \\ 0 & r_1
\end{pmatrix} \begin{pmatrix}
\lambda & 0 \\ 0 & \mu
\end{pmatrix}, \nonumber\\
&\hspace{3cm}\dots \nonumber\\
&\begin{pmatrix} \lambda^{-1} & 0 \\ 0 & \mu^{-1}
\end{pmatrix}
\begin{pmatrix}
r_{f-1} & 0 \\ 0 & r_{f-1}^{-1}
\end{pmatrix}
m_{f-1} A_{f-1} \begin{pmatrix}
r_{f-2}^{-1} & 0 \\ 0 & r_{f-2}
\end{pmatrix} \begin{pmatrix}
\lambda & 0 \\ 0 & \mu
\end{pmatrix}\Bigg) \nonumber.
\end{align}

\begin{defn}\label{defn-functor-T}
Define a functor $\mathcal{T}: X \to \mathcal{C}^{\tau}$ by sending
$$\left(\begin{pmatrix}
a_0 & b_0 \\
c_0 & d_0
\end{pmatrix}, \dots, \begin{pmatrix}
a_{f-1} & b_{f-1} \\
c_{f-1} & d_{f-1}
\end{pmatrix}\right) \in X(R)$$ to the Breuil-Kisin module $\M \in \mathcal{C}^{\tau}(R)$ constructed as follows:
\begin{enumerate}
    \item $\M_i = R[\![u]\!]e_i \oplus R[\![u]\!]f_i$.
    \item With respect to the basis $\{e_i, f_i\}$, the action of $g \in Gal(K'/K)$ is given by the diagonal matrix $\begin{psmallmatrix}
    \eta(g) & 0\\
    0 & \eta'(g)
    \end{psmallmatrix}$.
    \item With respect to the basis $\{u^{e - \gamma_{i}} \otimes e_{i-1}, 1 \otimes f_{i-1}\}$ (resp. $\{u^{e-\gamma_{i}} e_{i}, f_{i}\}$) of the $\eta'$-eigenspace of $\varphi^{*} \M_{i-1}$ (resp. $\M_{i}$), the matrix of the restriction of the $i$-th Frobenius map $\varphi^{*} \M_{i-1} \to \M_{i}$ to the $\eta'$-eigenspace is $\begin{psmallmatrix}
va_{i} & b_{i} \\
vc_{i} & d_{i}
\end{psmallmatrix}$.
\end{enumerate}
\end{defn}

Consider the pullback of $\mathcal{T}$ by the closed embedding $\mathcal{C}^{\tau, \mathrm{BT}} \hookrightarrow \mathcal{C}^{\tau}$. The pullback is a closed subscheme of $X$ that contains all the closed points of $X$ by Lemma \ref{hodge-type-equiv-strong-det}. Since $X$ is reduced, the pullback must be all of $X$ and $\mathcal{T}$ must map $X$ into $\mathcal{C}^{\tau, \mathrm{BT}}$. Choose an irreducible component $\mathcal{X}(\tau) \subset \mathcal{C}^{\tau, \mathrm{BT}}$ containing the image of $\mathcal{T}$. Such an irreducible component must exist because $X$ is irreducible, although a priori, it is not unique (we will see later in Proposition \ref{finite-presentation-of-irred-component} that in fact it is unique). Henceforth, we will see $\mathcal{T}$ as a functor from $X$ to $\mathcal{X}(\tau)$.

\begin{defn}\label{defn-F}
    Suppose $\tau$ satisfies Assumption \ref{tau}. We define a functor $F: G \times X \to X\times_{\mathcal{X}(\tau)}X$ in the following way:
    
    Let $g = (\lambda, \mu, r_1, ..., r_{f-1}, m_0, ..., m_{f-1}) \in G(R)$ and $x \in X(R)$. 
    Then $F((g,x))$ is the triple $(x, \: g \boldsymbol{\cdot} x, \: \{J_i\}_i)$ where $(x, \: g \boldsymbol{\cdot} x) \in X(R) \times X(R)$ and $\{J_i\}_i$
    are base change matrices for $\eta'$-eigenspaces (in the sense of Definition \ref{eigenspace-frobenii}) that encode transformation of the Frobenius matrices of $\mathcal{T}(x)$ to those of $\mathcal{T}(g \boldsymbol{\cdot} x)$.
    They are given by:
    \begin{align}\label{base-change-F}
&&J_i := \begin{cases}
\mathcal{F}_i((m_j)_j) \begin{pmatrix}
\lambda & 0 \\
0 & \mu
\end{pmatrix}  & \text{ for } i = 0,\\
\mathcal{F}_i((m_j)_j) \begin{pmatrix}
r^{-1}_{i}& 0 \\
0 & r_{i}
\end{pmatrix} \begin{pmatrix}
\lambda & 0 \\
0 & \mu
\end{pmatrix} & \text{ for } i \in \{1,\dots, f-1\}.\\
\end{cases} 
\end{align}

Here, $\mathcal{F}_i((m_j)_j)$ are described in Definition \ref{mathcal-F}.
\end{defn}

There exists a prestack $[X/G]^{\text{pre}}$ over $\Spec \F$ whose fiber category over $\Spec R$ is the groupoid with objects given by elements of the set $X(R)$ and morphisms given in the following way: for each $x \in X(R)$ and $g \in G(R)$, there exists a morphism $x \xmapsto{g}  g \boldsymbol{\cdot} x$. The functor $F$ induces a functor $[X/G]^{\text{pre}} \to \mathcal{X}(\tau)$ given by mapping $x \in X(R)$ to $\mathcal{T}(x)$ and $x \xmapsto{g} g \boldsymbol{\cdot} x$ to the isomorphism $\mathcal{T}(x) \to \mathcal{T}(g \boldsymbol{\cdot} x)$ encoded in the data of $F(g, x)$. Thus, by stackifying in the \textit{fppf} topology, one obtains a functor $[X/G] \to \mathcal{X}(\tau)$.
\begin{defn}\label{T-tilde}
Suppose $\tau$ satisfies Assumption \ref{tau}. We let $\widetilde{\mathcal{T}}: [X/G] \to \mathcal{X}(\tau)$ be the functor induced by $F$, as explained above. 
\end{defn}
\begin{prop}\label{finite-presentation-of-irred-component} The functor $\widetilde{\mathcal{T}}$ is an isomorphism. \end{prop}

The proof of Proposition \ref{finite-presentation-of-irred-component} will be given in several steps outlined below.

\begin{lemma}\label{diagonal-is-surj-mon}
The functor $F$ in Definition \ref{defn-F} is surjective on points valued in Artinian local $\mathbb{F}$-algebras and a monomorphism.
\end{lemma}
\begin{proof}
Let $R$ be an Artinian local $\mathbb{F}$-algebra.
Let $(x, y, \{J_i\}_i) \in (X \times_{\mathcal{C}^{\tau}}X)(R)$ where $(x, y) \in X(R) \times X(R)$ and $\{J_i\}_i$ are the base change matrices for $\eta'$-eigenspaces to transform $\mathcal{T}(x)$ to $\mathcal{T}(y)$.

Let $(A_i)_{i=0}^{f-1}$ be the Frobenius matrices for the $\eta'$-eigenspace for $\mathcal{T}(x)$. Because $\tau$ does not face the first obstruction, $\mathcal{T}(x)$ is not of bad genre and with respect to a suitable choice of inertial bases, the Frobenius matrices of $\mathcal{T}(x)$ will be in CDM form (see Proposition \ref{convergence-CDM-matrix} and Definition \ref{eigenspace-frobenii}). Because $\tau$ also does not face the second obstruction, using Proposition \ref{unipotent-action-gives-CDM-form-I-eta} we can uniquely determine $(r_1, \dots, r_{f-1}) \in \mathbb{G}_m^{f-1}(R)$ and $(m_0, \dots, m_{f-1}) \in U^{f}(R)$ so that the tuple $(A'_i)_{i=0}^{f-1}$ defined below is in CDM form:
\bal
A'_i := \begin{cases}
\vspace{0.3cm}
m_{i} A_{i} \begin{pmatrix}
r_{i-1}^{-1} & 0 \\ 0 & r_{i-1}
\end{pmatrix} &\text{ if } i =0,\\
\vspace{0.3cm}
\begin{pmatrix}
r_i & 0 \\ 0 & r_i^{-1}
\end{pmatrix}m_i A_i &\text{ if } i = 1, \\
\begin{pmatrix} 
r_{i} & 0 \\ 0 & r_{i}^{-1}
\end{pmatrix}
m_i A_i \begin{pmatrix}
r_{i-1}^{-1} & 0 \\ 0 & r_{i-1}
\end{pmatrix} &\text{ if } i \in \{2,\dots, f-1\}.
\end{cases}
\nal

Similarly, let $(B_i)_{i=0}^{f-1}$ be the Frobenius matrices for the $\eta'$-eigenspace corresponding to the data of $\mathcal{T}(y)$. We can uniquely determine $(s_1, \dots, s_{f-1}) \in \mathbb{G}_m^{f-1}(R)$ and $(n_0, \dots, n_{f-1}) \in U^{f-1}(R)$ so that the tuple $(B'_i)_{i=0}^{f-1}$ defined below is in CDM form:
\bal
B'_i = \begin{cases}
\vspace{0.3cm}
n_{i} B_{i} \begin{pmatrix}
s_{i-1}^{-1} & 0 \\ 0 & s_{i-1}
\end{pmatrix} &\text{ if } i =0, \\
\vspace{0.3cm}
\begin{pmatrix}
s_i & 0 \\ 0 & s_i^{-1}
\end{pmatrix}n_i B_i &\text{ if } i = 1, \\
\begin{pmatrix} 
s_{i} & 0 \\ 0 & s_{i}^{-1}
\end{pmatrix}
n_i B_i \begin{pmatrix}
s_{i-1}^{-1} & 0 \\ 0 & s_{i-1}
\end{pmatrix} &\text{ if } i \in \{2,\dots, f-1\}. \\
\end{cases}
\nal

Since $(A'_i)_i$ and $(B'_i)_i$ are base changes of $(A_i)_i$ and $(B_i)_i$ respectively, there exist base change matrices $(P_i)_i$ that allow us to transform $(A'_i)_i$ to $(B'_i)_i$. By Proposition \ref{automorphisms}, there exist $\lambda, \mu \in \mathbb{G}_m(R)$ so that $P_0 = \begin{psmallmatrix}
 \lambda & 0 \\
 0 & \mu
 \end{psmallmatrix}$ and for $i \in \{1,\dots, f-1\}$, $P_i = \begin{psmallmatrix}
 \lambda & 0 \\
 0 & \mu
 \end{psmallmatrix}\begin{psmallmatrix}
 k_i & 0 \\
 0 & k_i^{-1}
 \end{psmallmatrix}$, where $k_i = \mu\lambda^{-1}$ if $|\{j \in [1, i] \mid \mathcal{G}(A'_{j}) = \II\}| \text{ is odd}$, and $1$ otherwise.
 
 We now use $(r_i)_i$, $(s_i)_i$, $(m_i)_i$, $(n_i)_i$ and $(P_i)_i$ to write $(B_i)_i$ in terms of $(A_i)_i$.
 \bal 
B_i = \begin{cases}
\begin{aligned}
&n_i^{-1}
\begin{pmatrix}
 \lambda^{-1} & 0 \\
 0 & \mu^{-1}
 \end{pmatrix}
 m_i A_i
 \begin{pmatrix}
 r_{i-1}^{-1} & 0 \\
 0 & r_{i-1}
 \end{pmatrix} \begin{pmatrix}
 \lambda & 0 \\
 0 & \mu
 \end{pmatrix}\\
 &\hspace{4.4cm}\begin{pmatrix}
 k_{i-1} & 0 \\
 0 & k_{i-1}^{-1}
 \end{pmatrix}\begin{pmatrix}
 s_{i-1} & 0 \\
 0 & s_{i-1}^{-1}
 \end{pmatrix} 
\end{aligned}&\text{ for } i = 0,\\
\vspace{0.1cm}\\
\begin{aligned}
&n_i^{-1} \begin{pmatrix}
 s_i^{-1} & 0 \\
 0 & s_i
 \end{pmatrix} 
 \begin{pmatrix}
 k_{i}^{-1} & 0 \\
 0 & k_{i}
 \end{pmatrix}\begin{pmatrix}
 \lambda^{-1} & 0 \\
 0 & \mu^{-1}
 \end{pmatrix}\\ 
 &\hspace{4cm}
 \begin{pmatrix}
 r_i & 0 \\
 0 & r_i^{-1}
 \end{pmatrix} m_i A_i \begin{pmatrix}
 \lambda & 0 \\
 0 & \mu
 \end{pmatrix}\end{aligned} &\text{ for } i=1,\\
 \vspace{0.1cm}\\
\begin{aligned}
&n_i^{-1} \begin{pmatrix}
 s_i^{-1} & 0 \\
 0 & s_i
 \end{pmatrix} 
 \begin{pmatrix}
 k_{i}^{-1} & 0 \\
 0 & k_{i}
 \end{pmatrix}\begin{pmatrix}
 \lambda^{-1} & 0 \\
 0 & \mu^{-1}
 \end{pmatrix}
 \begin{pmatrix}
 r_i & 0 \\
 0 & r_i^{-1}
 \end{pmatrix} \\ 
&\hspace{1cm}m_i A_i \begin{pmatrix}
 r_{i-1}^{-1} & 0 \\
 0 & r_{i-1}
 \end{pmatrix} \begin{pmatrix}
 \lambda & 0 \\
 0 & \mu
 \end{pmatrix}\begin{pmatrix}
 k_{i-1} & 0 \\
 0 & k_{i-1}^{-1}
 \end{pmatrix}\begin{pmatrix}
 s_{i-1} & 0 \\
 0 & s_{i-1}^{-1}
 \end{pmatrix} 
\end{aligned} &\text{ for } i\in \{2, \dots,f-1\}.\\
\end{cases}
\nal\\
 
 Simplifying,
\bal 
B_i = \begin{cases}
\begin{pmatrix}
 \lambda^{-1} & 0 \\
 0 & \mu^{-1}
 \end{pmatrix}
 \tilde{m}_i A_i
 \begin{pmatrix}
 s_{i-1} k_{i-1} r_{i-1}^{-1} & 0 \\
 0 & s_{i-1}^{-1} k_{i-1}^{-1} r_{i-1}
 \end{pmatrix} \begin{pmatrix}
 \lambda & 0 \\
 0 & \mu
 \end{pmatrix} &\text{ for } i = 0,\\
 \vspace{0.1cm}\\
\begin{pmatrix}
 \lambda^{-1} & 0 \\
 0 & \mu^{-1}
 \end{pmatrix} \begin{pmatrix}
 s_i^{-1}   k_{i}^{-1}  r_i & 0 \\
 0 & s_i k_{i} r_i^{-1}
 \end{pmatrix} \tilde{m}_i A_i \begin{pmatrix}
 \lambda & 0 \\
 0 & \mu
 \end{pmatrix} &\text{ for } i=1,\\
 \vspace{0.1cm}\\
\begin{aligned}
&\begin{pmatrix}
 \lambda^{-1} & 0 \\
 0 & \mu^{-1}
 \end{pmatrix}
\begin{pmatrix}
s_i^{-1} k_{i}^{-1} r_i & 0 \\
 0 & s_i k_{i} r_i^{-1}
 \end{pmatrix} 
 \tilde{m}_i A_i \\
&\hspace{3cm}\begin{pmatrix}
 s_{i-1} k_{i-1} r_{i-1}^{-1} & 0 \\
 0 & s_{i-1}^{-1} k_{i-1}^{-1} r_{i-1}
 \end{pmatrix} \begin{pmatrix}
 \lambda & 0 \\
 0 & \mu
 \end{pmatrix}
 \end{aligned} &\text{ for } i\in \{2,\dots, f-1\}.\\
\end{cases}
\nal\\
where $\tilde{m}_i$ are suitably chosen unipotent matrices.

This implies the existence of a $g \in G$ such that $y = g \boldsymbol{\cdot} x$. By (\ref{base-change-F}), $F((g, x))$ contains the data of some base change matrices to go from $\{A_i\}$ to $\{B_i\}$ . These can only differ by a fixed scalar multiple from the original base change matrices $\{J_i\}_i$ (by Corollary \ref{unique-base-change-upto-scalar}). Scaling $\lambda$ and $\mu$ by this fixed multiple gives us a $g'$ such that $F((g', x)) = (x, y, \{J_i\}_i)$. This shows surjectivity on Artinian local points.\\

Now suppose that $R$ is \textit{any} $\mathbb{F}$-algebra. Let $(g, x), (g', x') \in (G \times X)(R)$ such that $F((g, x)) = F((g', x')) = (x, y, \{J_i\}_i)$. Then $x = x'$ and $y = g \boldsymbol{\cdot} x = g' \boldsymbol{\cdot} x$. Let $(A_i)_{i=0}^{f-1}$ be the Frobenius matrices for $\eta'$-eigenspaces in the data of $\mathcal{T}(x)$ (described in Definition \ref{defn-functor-T}) and $(B_i)_{i=0}^{f-1}$ be the corresponding matrices for $\mathcal{T}(y)$. Let 
\bal
&g = (\lambda, \mu, r_1, r_2, ..., r_{f-1}, m_0, ..., m_{f-1}), \\
&g' = (\lambda', \mu', r'_1, r'_2, ..., r'_{f-1}, m'_0, ..., m'_{f-1}).
\nal
By (\ref{base-change-F}),
\begin{align*}
&J_0 = \mathcal{F}_0((m_j)_j) \begin{pmatrix}
\lambda & 0 \\
0 & \mu
\end{pmatrix}  = \mathcal{F}_0((m'_j)_j) \begin{pmatrix}
\lambda' & 0 \\
0 & \mu'
\end{pmatrix}.
\end{align*}
All inertial base change matrices for $\eta'$-eigenspaces, including $\mathcal{F}_i((m_j)_j)$, are upper unipotent mod $v$. Reducing mod $v$, we get $\lambda = \lambda'$ and $\mu = \mu'$.

For $i \in \{1,\dots, f-1\}$,
\begin{align*}
J_i = \mathcal{F}_i((m_j)_j) \begin{pmatrix}
r^{-1}_{i}& 0 \\
0 & r_{i}
\end{pmatrix} \begin{pmatrix}
\lambda & 0 \\
0 & \mu
\end{pmatrix} = \mathcal{F}_i((m'_j)_j) \begin{pmatrix}
r'^{-1}_{i}& 0 \\
0 & r'_{i}
\end{pmatrix} \begin{pmatrix}
\lambda & 0 \\
0 & \mu
\end{pmatrix}.
\end{align*}

Again reducing mod $v$, we get $(r_1, ..., r_{f-1}) = (r'_1, ..., r'_{f-1})$. Finally we use (\ref{gaction}) to write $B_i$ in terms of $A_i$ and $g$, and compare it to $B_i$ written in terms of $A_i$ and $g'$. It is immediate that for each $i$, $m_i = m'_i$.
\end{proof}

\begin{lemma}\label{diagonal-is-isom}
The functor $F$ is an isomorphism.
\end{lemma}

\begin{proof}

We note that the diagonal of $F$ is an isomorphism because $F$ is a monomorphism (by Lemma \ref{diagonal-is-surj-mon}). This implies via \cite[\href{https://stacks.math.columbia.edu/tag/0AHJ}{Tag 0AHJ}]{stacks-project} that $F$ is representable by algebraic spaces.

To show $F$ is an isomorphism, we will show that $F$ is \'{e}tale since it is already known to be a surjective monomorphism and \'etale monomorphisms are open immersions. The property of being \'{e}tale is \'{e}tale-smooth local on the source-and-target by \cite[\href{https://stacks.math.columbia.edu/tag/0CG3}{Tag 0CG3}]{stacks-project}. Therefore, it suffices by \cite[\href{https://stacks.math.columbia.edu/tag/0CIF}{Tag 0CIF}]{stacks-project} to show the top arrow in the following diagram is \'{etale}, where $T$ is a smooth cover of $X \times_{\mathcal{X}(\tau)} X$ and $f: W \to T \times_{(X \times_{\mathcal{X}(\tau)} X)} (G \times_{\mathbb{F}} X)$ is an \'{e}tale cover.
\[
  \begin{tikzcd}
 W \arrow{r}{f} & T \times_{(X \times_{\mathcal{X}(\tau)} X)} (G \times_{\mathbb{F}} X) \arrow{r} \arrow[swap]{d} & T \arrow{d} \\
 & G \times_{\mathbb{F}} X \arrow{r}{F} & X \times_{\mathcal{X}(\tau)} X
  \end{tikzcd}
\]

The functor $F$ is unramified because it is locally of finite presentation with its diagonal an isomorphism. The only thing remaining to check then is that the map $W \xrightarrow{f} T \times_{(X \times_{\mathcal{X}(\tau)} X)} (G \times_{\mathbb{F}} X) \xrightarrow{pr_1} T$ is formally smooth, and since $f$ is already formally smooth, we reduce via \cite[\href{https://stacks.math.columbia.edu/tag/02HX}{Tag 02HX}]{stacks-project} to checking the lifting property for $pr_1$ along Artinian local rings. This is the content of the following lemma (since $pr_1$ is the scheme version of $F$):
\end{proof}
\begin{lemma} Suppose $R$ and $S$ are Artinian local $\mathbb{F}$-algebras with $j: Spec\; S \to Spec \; R$ a closed scheme and $j^{\#}: R \to S$ a surjection of local rings with the kernel squaring to zero. Then the dashed arrow exists in the following diagram
\[
  \begin{tikzcd}
  Spec \; S \arrow{r}{j} \arrow[swap]{d} & Spec \; R \arrow{d} \arrow[dashed]{ld}\\
  G \times_{\mathbb{F}} X \arrow{r}{F} & X \times_{\mathcal{X}(\tau)} X
  \end{tikzcd}.
\]
    
\end{lemma}

\begin{proof} The existence and uniqueness of the dashed arrow follows immediately from Lemma \ref{diagonal-is-surj-mon}, since $F$ induces a bijection for points valued in Artinian local rings.
\end{proof}

Let $l/\mathbb{F}$ be a field with $x$ an $l$-point of $X$, such that $\mathfrak{M} = \mathcal{T}(x)$ is a Breuil-Kisin module over $l$. Then there exists a map
$G \times_{\mathbb{F}}l 
\goto l \times_{\mathcal{X}(\tau)} X
\xrightarrow{\sim} \mathfrak{M} \times_{\mathcal{X}(\tau)} X$. By Lemma \ref{diagonal-is-surj-mon}, this map is surjective on field-valued points and the fiber of $G \times_{\mathbb{F}} l  
\goto l \times_{\mathcal{X}(\tau)} X$ over any field-valued point contains exactly one point, and is therefore of dimension $0$. By \cite[\href{https://stacks.math.columbia.edu/tag/0DS6}{Tag 0DS6}]{stacks-project}, the dimension of $l \times_{\mathcal{X}(\tau)} X
\xrightarrow{\sim} \mathfrak{M} \times_{\mathcal{X}(\tau)} X$ is $ \textrm{dim} (G \times_{\mathbb{F}} l) = \textrm{dim}\: G$. Since the fiber over $\M$ in $X$ is of the same dimension as $G$, the fiber over $\M$ in $[X/G]$ has dimension $0$.

Applying \cite[\href{https://stacks.math.columbia.edu/tag/0DS6}{Tag 0DS6}]{stacks-project} again to the map $\widetilde{\mathcal{T}}$ and using the above calculations of fiber dimension over $\mathfrak{M} \in \mathcal{X}(\tau)(l)$, we obtain that the dimension of the scheme-theoretic image of $\widetilde{\mathcal{T}}$ is the same as the dimension of $[X/G]$ which is $f$.

\begin{lemma}\label{candidate-presentation-is-surjective}
Suppose $\tau$ satisfies Assumption \ref{tau}.
\begin{enumerate}
   \item Let $R$ be an arbitrary $\mathbb{F}$-algebra and $\mathfrak{M} \in \mathcal{X}(\tau)(R)$. Fix an inertial basis for each $\M_i$ . Let $F_i$ denote the matrix for the Frobenius map $\varphi^{*} (\mathfrak{M}_{i-1}) \to \mathfrak{M}_{i}$ with respect to the chosen bases. Then, for each $i$, the top left entry of $F_i$ is $0$ mod $v$.
   \item
    The map $\mathcal{T}$ is a surjection onto $\mathcal{X}(\tau)$.
\end{enumerate}

\end{lemma}

\begin{proof}

Consider the substack $\mathcal{L}$ of $\mathcal{X}(\tau)$ defined in the following way: If $R$ is any $\mathbb{F}$-algebra, then $\mathcal{L}(R) \subset \mathcal{X}(\tau)(R)$ is the subgroupoid of those Breuil-Kisin modules for which the upper left entry of the Frobenius matrices is $0$ mod $v$ when the Frobenius matrices are written with respect to some inertial basis (hence, with respect to any inertial bases). A direct computation shows that this property is invariant under inertial base change. We claim, first of all, that $\mathcal{L}$ is a closed substack of $\mathcal{X}(\tau)$.

We can check it is representable by algebraic spaces and a closed immersion after pulling back to an affine scheme and working \textit{fpqc}-locally (by \cite[\href{https://stacks.math.columbia.edu/tag/0420}{Tag 0420}]{stacks-project}). Let $R$ be an $\mathbb{F}$-algebra and $\mathfrak{M}$ an $R$ point of $\mathcal{X}(\tau)$. For $i \in \mathbb{Z}/f\mathbb{Z}$, choose an inertial basis $\{e_i, f_i\}$ of $\mathfrak{M}_i$, and write Frobenius matrices $F_i$ of $\mathfrak{M}$ with respect to these bases. Suppose that for each $i$, the upper left entry of $F_i$ equals $a_i$ mod $v$, where $a_i \in R$. For every $R$-algebra $S$, the Frobenius matrices of $\mathfrak{M}_{S}$ with respect to these bases are given by $\{F_i \otimes S\}_i$. Then $\mathfrak{M}_{S}$ is a point of $\mathcal{L}$ if and only if $a_i = 0$ in $S$ for each $i$. Therefore the pullback of $\mathcal{L} \to \mathcal{X}(\tau)$ by the map $\mathfrak{M}$: $\text{Spec }R \to \mathcal{X}(\tau)$ is given by the closed immersion $V(a_0, ..., a_{f-1}) \hookrightarrow \text{Spec }R$.

Secondly, we note that $\mathcal{T}$ factors as $\mathcal{T}: X \twoheadrightarrow \mathcal{L} \hookrightarrow \mathcal{X}(\tau)$. The first map in this factorization is a surjection because for every field-valued point of $\mathcal{L}$, Proposition \ref{convergence-CDM-matrix} demonstrates the existence of an inertial basis with respect to which the Frobenius matrices are in CDM form, and thus the point is in the image of the functor $\mathcal{T}$. 
The dimension of the scheme-theoretic image of $\mathcal{T}$ is $f$ by the discussion before the statement of this Lemma, and the same is true for the dimension of $\mathcal{X}(\tau)$ by \cite[Prop.~5.2.20]{CEGS-local-geometry} (in particular, this relies on the fact that $K$ is an unramified extension of $\mathbb{Q}_p$, or else, the dimension of $\mathcal{X}(\tau)$ would be strictly greater than $f$). Since $\mathcal{X}(\tau)$ is reduced by construction in \cite[Cor.~5.3.1]{CEGS-components}, dimension considerations imply that it is the scheme-theoretic image of $\mathcal{T}$. However, the scheme-theoretic image of $\mathcal{T}$ must be contained in $\mathcal{L}$, the latter being a closed substack. Therefore, $\mathcal{L} = \mathcal{X}(\tau)$. 

Both assertions of the Lemma follow immediately.
\end{proof}


\begin{lemma}\label{candidate-presentation-is-etale-mono}
Suppose $\tau$ satisfies Assumption \ref{tau}. The map $\widetilde{\mathcal{T}}: \left[X/G\right] \to \mathcal{X}(\tau)$ is an \'etale monomorphism, representable by algebraic spaces.
\end{lemma}

\begin{proof}
To see that $\mathcal{\tilde{T}}$ is a monomorphism and representable by algebraic spaces, we show that the diagonal is an isomorphism. This is implied by the fact that the top arrow in the following cartesian diagram is an isomorphism (by Lemma \ref{diagonal-is-isom}) and \cite[\href{https://stacks.math.columbia.edu/tag/04XD}{Tag 04XD}]{stacks-project}.
 \[
 \begin{tikzcd}
    G \times_{\mathbb{F}} X \arrow{r}{\text{pr}_2, \text{action}} \arrow[swap]{d} & X \times_{\mathcal{X}(\tau)} X \arrow{d} \\
    \left[X/G\right] \arrow{r}{\Delta} & \left[X/G\right] \times_{\mathcal{X}(\tau)} \left[X/G\right]
 \end{tikzcd}
\]
Since the diagonal is an isomorphism, we also have that $\mathcal{\tilde{T}}$ is unramified. Therefore, to show \'{e}taleness, it suffices to show that $\mathcal{\tilde{T}}$ is formally smooth \cite[\href{https://stacks.math.columbia.edu/tag/0DP0}{Tag 0DP0}]{stacks-project}. As quotient map $X \to \left[X/G\right]$ is smooth, we reduce to showing formal smoothness of $\mathcal{T}$ by checking the lifting property along Artinian local rings as in Lemma \ref{diagonal-is-isom}. It suffices then to show the following:

Suppose $R$ and $S$ are Artinian local $\mathbb{F}$-algebras with $j: Spec\, S \to Spec \, R$ a closed scheme and $j^{\#}: R \to S$ a surjection of local rings with the kernel $I$ squaring to zero. Then the dashed arrow in the following diagram exists so that all triangles commute:

\[
  \begin{tikzcd}
  Spec \; S \arrow{r}{j} \arrow[swap]{d}{a} & Spec \; R \arrow{d}{b} \arrow[dashed]{ld}\\
   X \arrow{r}{\mathcal{T}} &\mathcal{X}(\tau)
  \end{tikzcd}.
\]
 
In order to construct such an arrow, we first claim that there exists some $c \in X(R)$ such that $\mathcal{T}(c) = b$. To see this, note that the determinant of each of the Frobenius matrices of $b$ is divisible by $v$ (by Lemma \ref{candidate-presentation-is-surjective}(1)). Further, modulo the maximal ideal of $R$, the $u$-adic valuation of the determinant of each Frobenius map is $e$ (by Lemma \ref{hodge-type-equiv-strong-det}). Therefore, the same holds true over $R$, and consequently $b$ is of Hodge type $\textbf{v}_0$ (see Definition \ref{defn-hodge-type}). Moreover, again by Lemma \ref{candidate-presentation-is-surjective}(1), each Frobenius matrix is in $\eta$-form (see Definition \ref{eta-eta'-form}). By Proposition \ref{convergence-CDM-matrix}, we can find a CDM form for $b$ giving us a suitable point $c \in X(R)$.

Since $\mathcal{T}(a) = b \circ j = \mathcal{T}(c \circ j)$, there exists some $g \in G(S)$, such that $g \boldsymbol{\cdot} (c \circ j) =  a$ (by Lemma \ref{diagonal-is-surj-mon}). Lift $g$ to any $\tilde{g} \in G(R)$. Then $\tilde{g} \boldsymbol{\cdot} c$ is the appropriate choice for the dashed arrow in the diagram above. 
\end{proof}

\begin{proof}[Proof of Proposition \ref{finite-presentation-of-irred-component}]
Follows from Lemmas \ref{candidate-presentation-is-surjective} and \ref{candidate-presentation-is-etale-mono}.
\end{proof}

\begin{prop}\label{global-functions-X}
Suppose $\tau$ satisfies Assumption \ref{tau}. The ring of global functions on $\mathcal{X}(\tau)$ is isomorphic to $\F[x, y][\dfrac{1}{y}]$.
\end{prop}

\begin{proof}
By Proposition \ref{finite-presentation-of-irred-component} the global functions of $\mathcal{X}(\tau)$ are the $G$-invariant global functions of $X$, where $G = (\mathbb{G}_{m})^{f+1}_{\mathbb{F}} \times_{\mathbb{F}} U_{\mathbb{F}}^{f}$ and $X = (\GL_{2})_{\mathbb{F}} \times_{\mathbb{F}} (\SL_{2})_{\mathbb{F}}^{f-1}$ and the $G$-action on $X$ is as in (\ref{gaction}). These functions are the same as the $(\mathbb{G}_{m})_{ \mathbb{F}}^{f+1}$-invariant global functions of $(U \backslash \GL_2)_{ \mathbb{F}} \times_{\mathbb{F}} (U \backslash \SL_2)_{\mathbb{F}}^{f-1}$. By the isomorphisms
\begin{align*}
    &U \backslash \GL_2 \xrightarrow{\sim} \mathbb{A}^2 \smallsetminus \{0\} \times \mathbb{G}_m,\\
    &\begin{pmatrix} a & b \\
c & d \end{pmatrix} \mapsto ((c, d), \: ad-bc) \nonumber
\end{align*} and
\begin{align*}
    &U \backslash \SL_2 \xrightarrow{\sim} \mathbb{A}^2 \smallsetminus \{0\}, \\
    &\begin{pmatrix} a & b \\
c & d \end{pmatrix} \mapsto (c, d), \nonumber
\end{align*}
the ring of global functions of $(U \backslash \GL_2)_{ \mathbb{F}} \times_{\mathbb{F}} (U \backslash \SL_2)_{\mathbb{F}}^{f-1}$ is isomorphic to $$\F[c_0, ..., c_{f-1}, d_0, ..., d_{f-1}, D][\dfrac{1}{D}]$$ where $\{c_i, d_i\}$ capture the lower two entries of the $i$-th matrix group while $D$ captures the determinant of the $\GL_2$ matrices. Under this identification, $(\lambda, \mu, r_1, r_2, ..., r_{f-1}) \in (\mathbb{G}_m)^{f+1}_{\mathbb{F}}$ acts on the global functions of $(U \backslash \GL_2)_{ \mathbb{F}} \times_{\mathbb{F}} (U \backslash \SL_2)_{\mathbb{F}}^{f-1}$ via:
\bal
&g \boldsymbol{\cdot} c_i =
\begin{cases}
\lambda \mu^{-1} r_{i-1}^{-1} c_i &\text{if } i = 0,\\
\lambda \mu^{-1} r_i^{-1}  c_i &\text{if } i = 1, \\
\lambda \mu^{-1} r_{i}^{-1} r_{i-1}^{-1} c_i &\text{if } i \in \{2,\dots, f-1\},\\
\end{cases}\\
&g \boldsymbol{\cdot} d_i =
\begin{cases}
r_{i-1} d_i &\text{if } i = 0,\\
r_i^{-1} d_i &\text{if } i = 1, \\
r_{i}^{-1} r_{i-1} d_i &\text{if } i \in \{2,\dots, f-1\},\\
\end{cases}\\
&g \boldsymbol{\cdot} D = D.\\
\nal
Therefore, the subring of $(\mathbb{G}_m)^{f-1}_{\mathbb{F}}$-invariant functions is $\F[d_0 \cdots d_{f-1}, D][\dfrac{1}{D}] \cong \F[x, y][\dfrac{1}{y}]$.
\end{proof}

\subsection{Identifying the component}\label{identify-component-subsec}
Our next order of business is to identify precisely which irreducible component of $\mathcal{C}^{\tau, \mathrm{BT}}$ can be written as the quotient stack $[X/G]$ using the strategy employed in Section \ref{smooth sec}. \cite[Cor.~5.3.1]{CEGS-components} shows that the irreducible components of $\mathcal{C}^{\tau, \mathrm{BT}}$ are in one-to-one correspondence with subsets of $\mathbb{Z}/f\mathbb{Z}$ called profiles. We now recall the definition of the profile of a Breuil-Kisin module and some of the specifics of the correspondence between irreducible components and profiles as it applies to our situation.

\begin{defn}\label{defn-profile}
    Let $\mathfrak{B} \in \mathcal{C}^{\tau, \mathrm{BT}}(\overline{\mathbb{F}})$ be an extension of $\M$ by $\mathfrak{N}$, where $\M$ and $\mathfrak{N}$ are two rank $1$ Breuil-Kisin modules.  For each $i$, let $m_i$ be a generator of $\M_i$ as an $\overline{\mathbb{F}}[\![u]\!]$ module.
    
    The profile of $\mathfrak{B}$ is the set $J := \{i \in \mathbb{Z}/f\mathbb{Z} \mid \forall g \in Gal(K'/K), g m_i \equiv \eta(g) m_i \mod u\}$. If we suppose the image of $m_{i-1}$ under Frobenius is $a_{i} u^{r_{i}} m_{i}$, where $a_{i} \in \overline{\mathbb{F}}[\![u]\!]^*$,
    then the refined profile of $\mathfrak{B}$ is the pair $(J, r)$ where $J$ is the profile of $\mathfrak{B}$ and $r= (r_i)_{i \in \mathbb{Z}/f\mathbb{Z}}$. 
\end{defn}


\begin{defn}
    Let $r = (r_i)_{i \in \mathbb{Z}/f\mathbb{Z}}$ be as follows:
\begin{align}\label{r-max-profile}
r_i = \begin{cases}
e &\text{ if } i-1, i \in J \text{ or if } i-1, i \not\in J,  \\
\gamma_i &\text{ if } i-1 \in J \text{ and } i \not\in J, \\
e-\gamma_i &\text{ if } i-1 \not\in J \text{ and } i \in J.
\end{cases}
\end{align}
Then the maximal refined profile associated to $J$ is $(J, r)$.
\end{defn}

By \cite[Lem.~4.2.14]{CEGS-components}, the irreducible component $\mathcal{C}^{\tau, \mathrm{BT}}(J)$ is the closure of a constructible set whose $\overline{\mathbb{F}}$ points are precisely the Breuil-Kisin modules of maximal refined profile associated to $J$.

\begin{lemma}\label{profile-of-I-eta-I-eta-genre}
Let $\tau = \eta \oplus \eta'$ be a tame principal series $\F$-type with $\eta \neq \eta'$. Let $J \subset \mathbb{Z}/f\mathbb{Z}$. Then $\mathcal{C}^{\tau, \mathrm{BT}}(J)(\overline{\mathbb{F}})$ contains a dense open subset of Breuil-Kisin modules $\mathfrak{B}$ satisfying $\G(\mathfrak{B}_i)=\Ieta$ for each $i$ if and only if $J = \mathbb{Z}/f\mathbb{Z}$.
\end{lemma}
\begin{proof}
Let $\mathfrak{B}$ be an $\overline{\mathbb{F}}$ point of maximal refined profile associated to $J$. Let $\mathfrak{B}$ be an extension of $\M$ by $\mathfrak{N}$ where $\M$ and $\mathfrak{N}$ are two rank $1$ Breuil-Kisin modules. For each $i$, choose a generator $m_i$ of $\M_i$ and $n_i$ of $\mathfrak{N}_i$ as $\overline{\mathbb{F}}[\![u]\!]$-modules. Let the image under Frobenius of $m_{i-1}$ be $a_{i} u^{r_{i}} m_{i}$ and that of $n_{i-1}$ be $a'_{i} u^{r'_{i}} n_{i}$ for some $a_{i}, a'_{i} \in \overline{\mathbb{F}}[\![u]\!]^{*}$. The strong determinant condition forces that $r_i + s'_i = e$ for each $i$ by Lemma \ref{hodge-type-equiv-strong-det}. By making careful choices of $m_i$ and $n_i$ (using either \cite[Lem.~4.1.1]{CEGS-components} or the proof of Lemma \ref{InertiaForm}), we can construct an inertial basis of $\mathfrak{B}_i$ made up of $m_i$ and a lift $\tilde{n}_i$ of $n_i$. Now we use the explicit description of $r_i$ in (\ref{r-max-profile}) to check the genre of the Frobenius maps for $\mathfrak{B}$ for different $J$'s. We have the following possibilities:
\begin{enumerate}
    \item If $J = \mathbb{Z}/f\mathbb{Z}$, then $\Gal(K'/K)$ acts on $m_i$ via $\eta$ (and therefore on $n_i$ via $\eta'$) for each $i \in \mathbb{Z}/f\mathbb{Z}$. Since $r_i = e$, the genre of $\mathfrak{B}_i$ is $\Ieta$ for each $i$.
    \item If $J$ is the empty set, $\Gal(K'/K)$ acts on $n_i$ via $\eta$ for each $i \in \mathbb{Z}/f\mathbb{Z}$. As $r'_i = 0$, the genre of $\mathfrak{B}_i$ is $\Ietaa$ for each $i$.
    \item If $J$ is neither $\mathbb{Z}/f\mathbb{Z}$ nor empty, then there exists an $i \in \mathbb{Z}/f\mathbb{Z}$ such that $i \in J$, but $i+1 \not \in J$. This implies that $r_{i+1} = \gamma_{i+1}$ and $r'_{i+1} = e - \gamma_{i+1}$. Consider the Frobenius matrix for $\varphi^{*}\mathfrak{B}_i \to \mathfrak{B}_{i+1}$ with respect to inertial bases $(m_i, \tilde{n}_i)$ of $\mathfrak{B}_i$ and $(\tilde{n}_{i+1}, m_{i+1})$ of $\mathfrak{B}_{i+1}$. The matrix has a zero in the lower right corner, and therefore is of genre $\Ietaa$ or of genre $\II$. Either way, $\mathcal{G}(\mathfrak{B}_{i+1}) \neq \Ieta$. 
\end{enumerate}

\end{proof}

\begin{cor}\label{image-of-T}
Recall $\widetilde{\mathcal{T}}$ from Definition \ref{T-tilde}. The scheme-theoretic image of $\widetilde{\mathcal{T}}$ is $\mathcal{C}^{\tau, \mathrm{BT}}(\mathbb{Z}/f\mathbb{Z})$.

\end{cor}
\begin{proof}
Since $\widetilde{\mathcal{T}}: [X/G] \to \mathcal{X}(\tau)$ is an isomorphism (by Proposition \ref{finite-presentation-of-irred-component}), there exists a dense open set of $\mathcal{X}(\tau)$ having the following property: If $\mathfrak{B}$ is an $\overline{\mathbb{F}}$ point of this dense open, then the lower right entry of each of its Frobenius matrices (with respect to inertial bases) is invertible. In other words, each Frobenius map has genre $\Ieta$. By Lemma \ref{profile-of-I-eta-I-eta-genre}, $\mathcal{X}(\tau)$ must be $\mathcal{C}^{\tau, \mathrm{BT}}(\mathbb{Z}/f\mathbb{Z})$.
\end{proof}

\begin{cor}\label{global-functions}
Let $\tau$ be a tame principal series $\F$-type satisfying Assumption \ref{tau}. Then the ring of global functions on $\mathcal{C}^{\tau, \mathrm{BT}}(\mathbb{Z}/f\mathbb{Z})$ is isomorphic to $\F[x, y][\dfrac{1}{y}]$.
\end{cor}
\begin{proof}
It follows from Proposition \ref{global-functions-X} and Corollary \ref{image-of-T}.
\end{proof}

\section{Passage to the Emerton-Gee stack}\label{passage sec}

\subsection{Image of irreducible components of \texorpdfstring{$\mathcal{C}^{\tau, \mathrm{BT}}$}{} in \texorpdfstring{$\mathcal{Z}$}{}}

Given a tame principal series $\F$-type $\tau$, $\mathcal{Z}^{\tau}$ is the scheme-theoretic image of $\mathcal{C}^{\tau, \mathrm{BT}}$ in $\mathcal{Z}$ (Definition \ref{defn-Z}). By \cite[Prop.~5.2.20]{CEGS-local-geometry}, $\mathcal{Z}^{\tau}$ is of pure dimension $[K:\mathbb{Q}_p]$. \cite[Cor.~5.3.1]{CEGS-components} tells us that the irreducible components of $\mathcal{Z}^{\tau}$ are indexed by profiles $J \in \mathcal{P}_{\tau}$, where $\mathcal{P}_{\tau}$ is defined in the following way.

\begin{defn}\label{defn-P-tau}    For a tame principal series $\F$-type $\tau=\eta\oplus\eta'$, let $\mathcal{P}_\tau$ be the collection of profiles $J\subset\mathbb{Z}/f\mathbb{Z}$ such that
\begin{itemize}
        \item if $i-1\in J$ and $i\not\in J$, then $z_{i}\neq p-1$;
        \item if $i-1\not\in J$ and $i\in J$, then $z_{i}\neq 0$.
    \end{itemize}
    
    (Recall $z_{i}$ from (\ref{defn-z_i})).
\end{defn}

We denote by $\mathcal{Z}^{\tau}(J)$ the irreducible component of $\mathcal{Z}^{\tau}$ indexed by $J$. \cite[Prop.~5.1.13]{CEGS-components} shows that $\mathcal{Z}^{\tau}(J)$ is the scheme-theoretic image of $\mathcal{C}^{\tau, \mathrm{BT}}(J)$. The irreducible components of $\mathcal{Z}$ are indexed by Serre weights, and for each $\sigma$ a Serre weight, $\mathcal{Z}(\sigma)$ can show up in $\mathcal{Z}^{\tau}$ for multiple choices of $\tau$. Thus we need to specify a dictionary to go from $J \in \mathcal{P}_{\tau}$ to a Serre weight $\sigma$.

For $J\in\mathcal{P}_\tau$, let $\delta_J$ denote the characteristic function of the set $J$. Define the integers $b_i$ and $a_i$ by
\begin{align}\label{z-to-serre-weight}
a_i=\begin{cases}
    z_{i}+\delta_{J^c}(i) & \text{if }i-1\in J \\
    0 & \text{if }i-1\not\in J
    \end{cases},\hspace{2cm}b_i=\begin{cases}
    p-1-z_{i}-\delta_{J^c}(i) & \text{if }i-1\in J \\
    z_{i}-\delta_{J}(i) & \text{if }i-1\not\in J
    \end{cases}.
\end{align}

Viewing $\eta'$ as a map $k^{\times} \to \F$ via Artin reciprocity, let $\sigma_{J}$ be the Serre weight $\sigma_{\vec{a}, \vec{b}} \otimes \eta' \circ \det$. Then by \cite[Thm.~5.1.17, Appendix~A]{CEGS-components}, $\mathcal{Z}^{\tau}(J)$ is the irreducible component indexed by the Serre weight $\sigma_{J}$. 

From now until the end of Section 5, our focus will be on the case where $J=\Z/f\Z$. However, in Appendix \ref{mixed-forms}, it will be necessary to refer back to the general definition of $J$ above.

\begin{prop}\label{finding-tau-for-serre-weight}
Set $J=\mathbb{Z}/f\mathbb{Z}$. Let $\sigma$ be a Serre weight that is not a twist of either the trivial or the Steinberg representation. That is, $\sigma = \sigma_{\vec{a}, \vec{b}}$ where $\vec{b} \not\in \{(0, ..., 0), (p-1, ..., p-1)\}$.

Then we can find a unique principal series $\F$-type $\tau = \eta \oplus \eta'$ such that $\eta \neq \eta'$ and $\sigma = \sigma_J$. 
\end{prop}
\begin{proof}
Let $z_i = p-1-b_i$. Define $\eta$ and $\eta'$ via
\bal
&\eta'(g) := \prod_{i=0}^{f-1} (\kappa_i \circ h(g)^{a_{i} - z_i}) \hspace{.5cm}
&\eta(g):= \eta'(g) \: \prod_{i=0}^{f-1}(\kappa_i \circ h(g)^{z_i}).
\nal
Let $\tau := \eta \oplus \eta'$. Clearly, $\sigma = \left(\otimes_{i=0}^{f-1} ({\det}^{z_{i}} \Sym^{b_i} k^2 ) \otimes_{k, \kappa_{i}} \F \right) \otimes \eta' \circ det  = \sigma_{J}$ for inertial $\F$-type $\tau$ as desired. Any $\tau$ so chosen is unique by (\ref{z-to-serre-weight}); $\vec{b}$ tells us exactly what the $\{z_i\}_i$ should be. Note that $\eta = \eta'$ if and only if all the $z_i$'s are $0$ or if all the $z_i$'s are $p-1$. Both of these situations are ruled out by the hypotheses in the statement of the Proposition.
\end{proof}

\begin{cor}\label{included-weights}
Let $S$ be the set of non-Steinberg Serre weights $\sigma$ such that $\mathcal{Z}(\sigma)$ is the image of $\mathcal{C}^{\tau, \mathrm{BT}}(\mathbb{Z}/f\mathbb{Z})$ for some $\tau =\eta \oplus \eta'$ satisfying Assumption \ref{tau}. Then $\sigma_{\vec{a}, \vec{b}} \in S$ if and only if each of the following conditions are satisfied:
\begin{enumerate}
    \item $\vec{b} \neq (0, 0, \dots, 0)$,
    \item $\vec{b} \neq (p-2, p-2, \dots, p-2)$, and
    \item Extend the indices of $b_i$'s to all of $\mathbb{Z}$ by setting $b_{i+f} = b_i$. Then $(b_i)_{i \in \mathbb{Z}}$ does not contain a contiguous subsequence of the form $(0, p-2, \dots, p-2, p-1)$, where the number of $p-2$'s in between $0$ and $p-1$ can be anything in $\mathbb{Z}_{\geq 0}$.
\end{enumerate}
\end{cor}
\begin{proof}
Proposition \ref{finding-tau-for-serre-weight} accounts for the first condition. By (\ref{z-to-serre-weight}), requiring $\tau$ to not face the first obstruction is equivalent to requiring $(b_i)_{i \in \mathbb{Z}}$ to not be made up entirely of concatenations of just two building blocks: $p-2$ and $(p-1, 0)$. Similarly, requiring $\tau$ to not face the second obstruction is equivalent to requiring $(b_i)_{i \in \mathbb{Z}}$ to not contain a contiguous subsequence of the form $(0, p-2, \dots, p-2, p-1)$ of length $\geq 2$. If $(b_i)_{i \in \mathbb{Z}}$ is entirely made up of and contains both $p-2$ and $(p-1, 0)$, then it automatically contains a contiguous subsequence of the form $(0, p-2, \dots, p-2, p-1)$. Therefore, removing the redundant condition, we get the list of the conditions in the statement of the Corollary.
\end{proof}





\subsection{Presentations of components of \texorpdfstring{$\mathcal{Z}$}{}}

We will now show that if $\mathcal{Z}(\sigma)$ is as in the statement of Corollary \ref{included-weights}, then it is isomorphic to $\mathcal{C}^{\tau, \mathrm{BT}}(\mathbb{Z}/f\mathbb{Z})$. 
A key ingredient in our proof will be the following proposition.

\begin{prop}\label{quasifiniteness-prop}
Let $\tau =\eta \oplus \eta'$ be a tame principal series type satisfying Assumption \ref{tau}. Let $\sigma = \sigma_{\mathbb{Z}/f\mathbb{Z}}$.
The map $q: [X/G] \cong \mathcal{X}(\tau) = \mathcal{C}^{\tau, \mathrm{BT}}(\mathbb{Z}/f\mathbb{Z}) \to \mathcal{Z}(\sigma)$ induced from Definition \ref{defn-Z} is a monomorphism.

\end{prop}

The proof of this Proposition depends on the following Lemma.

\begin{lemma}\label{det-B-i} Let $R$ be an arbitrary $\F$-algebra and let $\M, \mathfrak{N} \in \mathcal{C}^{\tau, \mathrm{BT}}(R)$ be such that with respect to some fixed inertial bases, the $i$-th Frobenius maps of $\mathfrak{M}$ and $\mathfrak{N}$ are in $\eta$-form. setting $v^{(\alpha, \beta)}=\begin{psmallmatrix}
v^{\alpha} & 0 \\
0 & v^\beta
\end{psmallmatrix}$, suppose that upon restriction to the $\eta'$-eigenspace, the Frobenius matrices for $\mathfrak{M}$ and $\mathfrak{N}$ are represented by
\begin{align}\label{Prop-F-form}
F_i = \begin{pmatrix} a_i & b_i\\
c_i & d_i
\end{pmatrix} v^{(1, 0)} \quad \text{ and } \quad G_i = \begin{pmatrix} a'_i & b'_i\\
c'_i & d'_i
\end{pmatrix} v^{(1, 0)}, \quad \text{respectively, }
\end{align}
where $F'_i = \begin{psmallmatrix} a_i & b_i\\
c_i & d_i
\end{psmallmatrix}$ and $G'_i = \begin{psmallmatrix} a'_i & b'_i\\
c'_i & d'_i
\end{psmallmatrix}$ are matrices in $\mathrm{GL}_2(R[\![v]\!])$.\\
If $\mathfrak{M}$ and $\mathfrak{N}$ are isomorphic as \'etale $\varphi$-modules, so that by (\ref{eta'-base-change}), there exist $B_0, \dots, B_{f-1} \in \GL_2(R(\!(v)\!) )$ such that 
\begin{align}\label{base-change-Bi}
    G_i = B_{i}^{-1} F_i \left( \text{Ad} \;
    v^{(p-1-z_{i}, 0)} (\varphi(B_{i-1})) \right),
\end{align} 
then $\text{det}(B_i) \in R[\![v]\!]^*$ and $v B_i \in M_2(R[\![v]\!])$.
\end{lemma}
We observe that in the statement of Lemma \ref{B-i}, $B_i \in \GL_2(R(\!(v)\!))$ by Lemma \ref{constraints-base-change-matrix} and Definition \ref{eigenspace-frobenii}.

\begin{proof}
    From (\ref{base-change-Bi}), we see that $$\det(B_{i})\det(G_i) = \det(F_i) \varphi( \text{det}(B_{i-1})).$$
    Since $\text{val}_v(\text{det}(F_i)) = \text{val}_v(\text{det}(G_i)) = 1$, we have 
\bal 
\text{val}_v(\text{det}(B_{i})) = \text{val}_v( \varphi(\text{det}(B_{i-1})) = p \text{val}_v( \text{det}(B_{i-1})).
\nal 
Iterating this equation gives us
\bal 
\text{val}_v(\text{det}(B_i)) = p \text{val}_v(\text{det}(B_{i-1})) = p^2 \text{val}_v(\text{det}(B_{i-2})) = \cdots = p^f \text{val}_v(\text{det}(B_{i})),
\nal 
which shows $\text{val}_v(\text{det}(B_{i})) =0$. We now choose $k_i \in \mathbb{Z}_{\geq 0}$ minimal such that 
\begin{align*}
    B_{i} &= v^{-k_i} B_i^+ = v^{-k_i}\begin{pmatrix}
    s_1^{(i)} & s_2^{(i)} \\
    s_3^{(i)} & s_4^{(i)}
    \end{pmatrix}, & \text{ where } s_1^{(i)}, s_2^{(i)}, s_3^{(i)}, s_4^{(i)} \in R[\![v]\!].
\end{align*}

Then from (\ref{base-change-Bi}), we have 
\bal
G'_i = B_{i}^{-1} F'_i \left( \text{Ad} \; v^{(p-z_i, 0)} (\varphi(B_{i-1})) \right).
\nal
Equivalently,
\bal 
    F_i^{\prime -1} B_{i}^+ G_i' &=  v^{k_{i} - p k_{i-1}} \begin{pmatrix}
    v^{p-z_{i}} & 0 \\
    0 & 1
    \end{pmatrix} \varphi(B_{i-1}^+) \begin{pmatrix}
    v^{-p+z_{i}} & 0 \\
    0 & 1
    \end{pmatrix}\\
    &= v^{k_{i} - p k_{i-1}} \begin{pmatrix} 
    \varphi(s_1^{(i-1)})&  v^{p-z_{i}}\varphi(s_2^{(i-1)})\\
    v^{-p + z_{i}} \varphi(s_3^{(i-1)})& \varphi(s_4^{(i-1)})
    \end{pmatrix}.
\nal 
Since $k_{i-1}$ is chosen to be minimal, we must have $\text{val}_v(\varphi(s_{m_i}^{(i-1)} ) ) = 0$ for some $m_i \in \{1,2,3,4\}$. Then, since $F'^{-1} B_i^+ G_i' \in \text{M}_2(R[\![v]\!])$, we have:
\bal 
&\hspace{1.2cm}m_i \in \{1,4\} &\implies &0 \leq k_{i} - p k_{i-1} &\implies &k_{i} \geq p k_{i-1},\\
&\hspace{1.2cm}m_i =2 &\implies &0\leq k_{i} - p k_{i-1} + p - z_{i} &\implies &k_{i} \geq p k_{i-1} - (p- z_{i}),\\
&\hspace{1.2cm}m_i =3 &\implies &0\leq k_{i} - p k_{i-1} -p +z_{i} &\implies &k_{i} \geq p k_{i-1} + (p- z_{i}). 
\nal 
In other words, $k_{i} \geq pk_{i-1} - \epsilon_i$ where 
\begin{align*}
    \epsilon_i = \begin{cases}0 &\text{ if } m_i \in \{1, 4\} \\
    p-z_i &\text{ if } m_i = 2 \\
    -(p-z_i) &\text{ if } m_i = 3
\end{cases}
\end{align*}

Iterating, we get 
\begin{align*}
    &\hspace{0.4cm}k_i \geq p^f k_i - (p^{f-1}\epsilon_{i+1} + p^{f-2} \epsilon_{i+2} + ... + \epsilon_{i+f})\\ &\iff
   (p^f - 1)k_i \leq (p^{f-1}\epsilon_{i+1} + p^{f-2} \epsilon_{i+2} + ... + \epsilon_{i+f}) \leq p(\sum_{j=0}^{f-1} p^j).
\end{align*}

Since $p^f - 1 = (p-1)(\sum_{j=0}^{f-1} p^j)$, we must have $k_i \in \{0, 1\}$, showing that $B_i \in vM_2(R[\![v]\!])$.
\end{proof}

\begin{lemma}\label{s_3}
Assume the setup and notation in the statement and proof of Lemma \ref{det-B-i}. If $k_{i-1} = 1$, then $v^2$ divides $s_3^{(i-1)}$. Consequently, there exists $B'_i \in \GL_2(R[\![v]\!])$ with the top right entry not divisible by $v$, and such that $B_i = \text{Ad}(v^{(0,1)}) (B_i')$.
\end{lemma}

\begin{proof}
Evidently, $k_{i-1} = 1$ implies that $\epsilon_i \geq p-1$, or equivalently, $z_{i} \in \{0, 1\}$ and $v$ divides $s_{1}^{(i-1)}$, $s_{3}^{(i-1)}$, and $s_{4}^{(i-1)}$ but not $s_2^{(i-1)}$. Let $x$ be the constant part of $\frac{s_{3}^{(i-1)}}{v}$. We wish to show that $x=0$. Note that since $\text{det}(B_{i-1})$ is a unit in $R[\![v]\!]$, $x s_2^{(i-1)} \equiv 0 \mod v$.
 
 From the top left and bottom left entries of the matrices in (\ref{base-change-Bi}), we get the following equalities:
    \begin{align}
        \label{a}v^{k_i + pk_{i-1}} a'_i \text{det}(B_i) = \begin{split}
            a_i \varphi(s_1^{(i-1)}) s_4^{(i)} + v^{-p+z_i} b_i \varphi(s_3^{(i-1)})s_4^{(i)} \\
             - c_i \varphi(s_1^{(i-1)})s_2^{(i)} - v^{-p+z_i}d_i \varphi(s_3^{(i-1)})s_2^{(i)},
        \end{split} \\[12pt]
        \label{c}v^{k_i + pk_{i-1}} c'_i \text{det}(B_i) = \begin{split}
            -a_i \varphi(s_1^{(i-1)}) s_3^{(i)} - v^{-p+z_i} b_i \varphi(s_3^{(i-1)})s_3^{(i)} \\
            + c_i \varphi(s_1^{(i-1)})s_1^{(i)} + v^{-p+z_i}d_i \varphi(s_3^{(i-1)})s_1^{(i)} .
        \end{split}
    \end{align}
    
    Consider the equations $s_1^{(i)}$(\ref{a}) $+ \; s_2^{(i)}$(\ref{c}) and $s_3^{(i)}$(\ref{a}) $+ \; s_4^{(i)}$(\ref{c}). Dividing both equations by $\text{det}(B_i^{+}) = v^{2k_i} \text{det}(B_i) $, we obtain:
    \begin{align}\label{e}
        v^{pk_{i-1}-k_i}(s_1^{(i)}a'_i + s_2^{(i)}c'_i)\hspace{0.1cm}  = &\hspace{0.1cm}a_i   
        \varphi(s_1^{(i-1)}) + v^{-p + z_i}b_i \varphi(s_3^{(i-1)}), \\
       \label{f} v^{pk_{i-1}-k_i}(s_3^{(i)}a'_i + s_4^{(i)}c'_i)\hspace{0.1cm}  = &\hspace{0.1cm}c_i   
        \varphi(s_1^{(i-1)}) + v^{-p + z_i}d_i \varphi(s_3^{(i-1)}).
    \end{align}

Using that $v^p$ divides $\varphi(s_1^{(i-1)})$ and $\varphi(s_3^{(i-1)})$, we make the following observations:
\begin{enumerate}
    \item In (\ref{e}), the LHS is divisible by $v^{p - k_{i}}$, while $a_i   
        \varphi(s_1^{(i-1)})$ is divisible by $v^p$. Therefore, $v^{p-k_i} \mid v^{-p + z_i}b_i \varphi(s_3^{(i-1)})$, implying $v^{p-k_i - z_i} \mid b_i x$.
\item Similarly, in (\ref{f}), $v^{p-k_i - z_i} \mid d_i x$.
\end{enumerate}


Since $z_i, k_i \in \{0, 1\}$ and $p>2$, $b_i x \equiv d_i x \equiv 0 \mod v$. 
Multiplying $\text{det} (F'_i) \in R[\![v]\!]^{*}$ by $x$, we get:
$$ a_i d_i x - c_i b_i x \equiv 0 \mod v.
$$

This shows that $x \equiv 0$ mod $v$. Since $x \in R$, $x = 0$ and we are done.
\end{proof}

\begin{lemma}\label{B-i}
 Assume the setup and notation in the statement and proof of Lemma \ref{det-B-i}. If $\tau$ does not face the first or second obstructions, then $B_i\in\GL_2(R[\![v]\!])$. Furthermore, each $B_i$ is upper triangular mod $v$.
\end{lemma}

\begin{proof}
To prove the first statement, it suffices to show that:
\begin{enumerate}
    \item $(k_{i-1},k_i)=(0,1) \implies z_i = p-1$,
    \item $(k_{i-1},k_i)=(1,1) \implies z_i = 1$,
    \item $(k_{i-1},k_i)=(1,0) \implies z_i = 0$.
\end{enumerate}

This is because if the above three results hold, then $k_i = 1$ for some $i$ implies that either each $z_j = 1$ (first obstruction), or there exists a contiguous subsequence $(p-1,1,\dots,1,0)$ in $\vec{z}$ (second obstruction). Thus, the hypothesis on $\tau$ forces each $k_i = 0$, and therefore, each $B_i$ is a matrix in $\mathrm{M}_2(R[\![v]\!])$. As $\text{det}(B_i)$ is a unit, we obtain that $B_i \in \mathrm{GL}_2(R[\![v]\!])$. 

We now prove the three claims.

\begin{enumerate}
    \item Suppose $k_{i-1}=0$ and $k_i=1$. Lemma \ref{s_3} allows us to write $B_i$ as $\text{Ad}(v^{(0,1)}) (B_i')$, where $B_i'= \begin{pmatrix} t_1& t_2\\ t_3& t_4\end{pmatrix} \in \mathrm{GL}_2(R[\![v]\!])$ and $t_2 \not\equiv 0$ mod $v$.

\bal 
&\text{By (\ref{base-change-Bi}), } \text{ Ad}(v^{(0,1)}) (B_i') \hspace{.1cm} G'_i \\
&\hspace{0.5cm}= \begin{pmatrix}
a'_i t_1 + v^{-1} c'_i t_2& b'_i t_1 + v^{-1} d'_i t_2\\
c'_i t_4 + v a'_i t_3& d'_i t_4 + v b'_i t_3
\end{pmatrix}\\
&\hspace{0.5cm}= \begin{pmatrix}
a_i \vp(s_1^{(i-1)}) + b_i \vp(s_3^{(i-1)}) v^{-(p-z_i)}& b_i \vp(s_4^{(i-1)}) + a_i \vp(s_2^{(i-1)}) v^{p-z_i}\\
c_i \vp(s_1^{(i-1)}) + d_i \vp(s_3^{(i-1)}) v^{-(p-z_i)}& d_i \vp(s_4^{(i-1)}) + c_i \vp(s_2^{(i-1)}) v^{p-z_i}
\end{pmatrix} \\
&\hspace{0.5cm}= F'_i  \hspace{.1cm}\text{Ad}(v^{(p-z_i,0)})\big(\vp(B_{i-1})\big).
\nal 
An examination of the bottom left entry shows $v^{p-z_i} \mid d_i \vp(s_3^{(i-1)})$. If $v\mid \vp(s_3^{(i-1)})$, then $v \mid s_3^{(i-1)}$ and $v^p \mid \vp(s_3^{(i-1)})$. But then both the top entries of $F'_i  \hspace{.1cm}\text{Ad}(v^{(p-z_i,0)})\big(\vp(B_{i-1})\big)$ are in $R[\![v]\!]$, which implies $c'_i t_2 \equiv d'_i t_2 \equiv 0 \mod v$. Multiplying $\text{det} (G'_i) \in R[\![v]\!]^{*}$ by $t_2$, we see that $t_2 \equiv 0$ mod $v$. This is a contradiction. Therefore, the constant part of $s_3^{(i-1)}$ is non-zero and we denote it by $x$.

Now, assume that $z_i < p-1$. Comparing the top and bottom left entries of $\text{Ad}(v^{(0,1)}) (B_i') \hspace{.1cm} G'_i$ and $F'_i  \hspace{.1cm}\text{Ad}(v^{(p-z_i,0)})\big(\vp(B_{i-1})\big)$, we obtain $b_i x \equiv d_i x \equiv 0$ mod $v$. Therefore, $v$ divides $x \; \text{det}(F'_i)$. Since $\text{det}(F'_i) \in R[\![v]\!]^{*}$, this means that $x=0$, a contradiction. Hence, $z_i = p-1$.

\item Suppose $k_{i-1} = k_i =1$. As before, let $B'_{i-1}, B'_{i} \in \GL_2(R[\![v]\!])$ be such that $B_{i-1} = \text{Ad}(v^{(0,1)}) (B_{i-1}')$ and $B_i = \text{Ad}(v^{(0,1)}) (B_i')$. Since $k_{i-1} =1$ implies $z_i \in \{0,1\}$, it suffices to show that $z_i=0$ leads to a contradiction. If $z_i=0$, then
\bal 
\text{Ad}(v^{(0,1)}) (B_i') \hspace{.1cm} G'_i &= F'_i  \hspace{.1cm}\text{Ad}(v^{(p,0)})\big(\varphi(\text{Ad}(v^{(0,1)}) (B_{i-1}')\big) = F'_i \vp(B_{i-1}')\\
\implies \text{Ad}(v^{(0,1)}) (B_i') &= F'_i \varphi(B'_{i-1}) G_i^{\prime -1} \in \GL_2(R[\![v]\!]).
\nal 
This forces $v \mid (B_i')_{1,2}$, a contradiction.

\item Suppose $k_{i-1} = 1$ and $k_i =0$, so that $z_i \in \{0,1\}$. Suppose $z_i = 1.$ Then
\bal 
 F_i^{\prime -1} B_i G'_i &= \text{Ad}(v^{(p-1,0)})\big(\varphi(\text{Ad}(v^{(0,1)}) (B_{i-1}')\big) \\
 &= \text{Ad}(v^{(0,1)})\big(\varphi(B'_{i-1})\big)\in \GL_2(R[\![v]\!]).
\nal 
Once again, this is a contradiction because it implies $v\mid (B'_{i-1})_{1,2}$.
\end{enumerate}

To justify the second statement , we continue using the notation of the proof of Lemma \ref{det-B-i}. Since each $k_i$ is $0$, (\ref{e}) gives us:
\begin{align}
        s_1^{(i)}a'_i + s_2^{(i)}c'_i \hspace{0.1cm}  = &\hspace{0.1cm}a_i \varphi(s_1^{(i-1)}) + v^{-p + z_i}b_i \varphi(s_3^{(i-1)})
\end{align}

Since the LHS is integral, the same must be true for RHS. Therefore $b_i s_3^{(i-1)} \equiv 0$ mod $v$. Considering (\ref{a}) and (\ref{c}), we have $d_i s_3^{(i-1)} s_1^{(i)} \equiv d_i s_3^{(i-1)} s_2^{(i)} \equiv 0$ mod $v$. Multiplying $\text{det} (B_i)$ by $d_i s_3^{(i-1)}$, we see that $d_i s_3^{(i-1)} \equiv 0$ mod $v$. Finally, multiplying $\text{det} (F'_i) = a_i d_i - b_i c_i$ by $s_3^{(i-1)}$, we see that $s_3^{(i-1)} \equiv 0$ mod $v$. 

\end{proof}

\begin{proof}[Proof of Proposition \ref{quasifiniteness-prop}] 
Let $R$ be an arbitrary $\mathbb{F}$-algebra, and let $\M, \mathfrak{N} \in \mathcal{X}(\tau)(R) \cong [X/G](R)$ be two Breuil-Kisin modules equipped with an isomorphism after inverting $u$. By Proposition \ref{finite-presentation-of-irred-component}, the Frobenius matrices of $\M$ and $\mathfrak{N}$ can be written in the form described in the statement of Lemma \ref{B-i} (after passing to an affine cover of $\text{Spec }R$ if necessary). Denoting the Frobenius matrices of $\M$ by $\{F_i\}_i$ and those of $\mathfrak{N}$ by $\{G_i\}_i$, the isomorphism between $q(\M)$ and $q(\mathfrak{N})$ is described by invertible matrices $\{B_i\}_i$ satisfying (\ref{base-change-Bi}). Lemma \ref{B-i} shows each $B_{i} \in \text{GL}_2(R[\![v]\!])$ and all are upper triangular mod $v$. Hence, by comparison with the form of inertial base change matrices for $\eta'$-eigenspace, the set $\{B_i\}_i$ gives an isomorphism between $\M$ and $\mathfrak{N}$. Therefore, $X \times_{\mathcal{X}(\tau)} X \to X \times_{\mathcal{Z}} X$ is an isomorphism. Using \ref{diagonal-is-isom}, the diagonal of $q$ is an isomorphism (the argument for this is the same as in the first paragraph of the proof of Lemma \ref{candidate-presentation-is-etale-mono}). Therefore, $q$ is a monomorphism.

\end{proof}

\begin{cor}\label{main-isom}
Fix $\tau = \eta \oplus \eta'$ a tame principal series $\F$-type with $\eta \neq \eta'$ such that $\tau$ does not face either the first or the second obstruction. Let $\sigma = \sigma_{\mathbb{Z}/f\mathbb{Z}}$.
The map $q: [X/G] \cong \mathcal{C}^{\tau, \mathrm{BT}}(\mathbb{Z}/f\mathbb{Z}) \to \mathcal{Z}(\sigma)$ is an isomorphism.
\end{cor}
\begin{proof}
The map $q$ is proper and scheme-theoretically dominant by \cite[Thm.~5.1.2]{CEGS-local-geometry} (using terminology from \cite{EG-Scheme-Image}, 1.1.1), and it is also a monomorphism by Proposition \ref{quasifiniteness-prop}. Since proper monomorphisms are the same as closed immersions (by \cite[\href{https://stacks.math.columbia.edu/tag/0418}{Tag 0418}, \href{https://stacks.math.columbia.edu/tag/04XV}{Tag 04XV}]{stacks-project})
, $q$ must be an isomorphism. 
\end{proof}

\section{Conclusion}

\begin{thm}\label{main-thm}
Let $p>2$. Let $K$ be an unramified extension of $\Q_{p}$ of degree $f$ with residue field $k$. Let $\mathcal{Z}(\sigma)$ be the irreducible component of $\mathcal{Z}$ indexed by a non-Steinberg Serre weight $\sigma = \sigma_{\vec{a}, \vec{b}} = \bigotimes_{i=0}^{f-1} ({\det}^{a_{i}} \Sym^{b_i} k^2 ) \otimes_{k, \kappa_{i}} \F$ satisfying the following properties:
 \begin{enumerate}
    \item $\vec{b} \neq (0, 0, \dots, 0)$,
    \item $\vec{b} \neq (p-2, p-2, \dots, p-2)$,
    \item Extend the indices of $b_i$'s to all of $\mathbb{Z}$ by setting $b_{i+f} = b_i$. Then $(b_i)_{i \in \mathbb{Z}}$ does not contain a contiguous subsequence of the form $(0, p-2, \dots, p-2, p-1)$ of length $\geq 2$.
\end{enumerate}
Then $\mathcal{Z}(\sigma)$ is smooth and isomorphic to a quotient of $\GL_2 \times \SL_2^{f-1}$ by $\mathbb{G}_m^{f+1} \times \mathbb{G}_a^{f}$. The ring of global functions of $\mathcal{Z}(\sigma)$ is isomorphic to $\F[x, y][\frac{1}{y}]$.
\end{thm}

\begin{proof}
Follows directly from Corollaries \ref{global-functions}, \ref{included-weights} and \ref{main-isom}.
\end{proof}

\begin{remark}
Assume the conditions in the statement of Theorem \ref{main-thm}. Let $\tau = \eta \oplus \eta'$ be an inertial type with $\eta(g) = \prod_{i=0}^{f-1}(\kappa_i \circ h(g)^{a_i})$ and $\eta'(g) = \prod_{i=0}^{f-1} (\kappa_i \circ h(g)^{a_{i} + b_i - (p-1)})$. 
Since $\vec{b} \neq (0, 0, \dots, 0)$, $\mathcal{Z}(\sigma_{\vec{a}, \vec{b}}) = \mathcal{Z}^{\tau}(\mathbb{Z}/f\mathbb{Z})$ by Proposition \ref{finding-tau-for-serre-weight}. The proof of Theorem \ref{main-thm} tells us that $$q: [X/G] \xrightarrow{\widetilde{\mathcal{T}}} \mathcal{C}^{\tau, \mathrm{BT}}(\mathbb{Z}/f\mathbb{Z}) \to \mathcal{Z}(\sigma_{\vec{a}, \vec{b}})$$ is an isomorphism, providing a concrete description of the points of $\mathcal{Z}(\sigma_{\vec{a}, \vec{b}})$.
\end{remark}

\begin{remark}
In fact, when $f \geq 2$ and $p>3$, we can allow $\vec{b} = (0, 0, \dots, 0)$ in the statement of Thoerem \ref{main-thm} (see discussion in Section \ref{sec-trivial-serre}).
\end{remark}

\appendix
\section{Allowing \texorpdfstring{$\eta'$-forms}{}}\label{mixed-forms}
The objective of this Appendix is to show that allowing some of the Frobenius matrices to be in $\eta'$-form does not allow us to obtain information on more irreducible components, with the exception of the component indexed by the trivial Serre weight. Before we embark on a proof, we first survey the overall strategy employed in the main body of the paper, and analyze how it might be affected by allowing some Frobenius matrices to be in $\eta'$-form.

A key ingredient in the proof of our main theorem is constructing the functor $\widetilde{\mathcal{T}}: [X/G] \to \mathcal{X}(\tau)$ (see Definition \ref{T-tilde}), where $X = \GL_2 \times \SL_2^{f-1}$ and $G = \mathbb{G}_m^{f+1} \times U^{f}$, and then showing that it is an isomorphism (see Proposition \ref{finite-presentation-of-irred-component}). The proof of the isomorphism relies, among other things, on the following:
\begin{enumerate}
    \item Let $\M$ be a regular Breuil-Kisin module and let $\{F_i\}_i$ be the set of its Frobenius matrices with respect to some choice of inertial bases. Suppose that each $F_i$ is in $\eta$-form. Then, upon imposing some conditions on $(z_i)_i$, we can guarantee that $\M$ is not of bad genre and therefore the algorithm in Proposition \ref{convergence-CDM-matrix} converges to give Frobenius matrices in CDM form. The minimal set of values of $(z_i)_i$ we need to exclude constitutes the definition of the first obstruction.
    \item For $\M$ as above, we also need to obtain the CDM form of Frobenius matrices through an action of $G$. The conditions on $(z_i)_i$ that prohibit this constitute the definition of the second obstruction.
\end{enumerate}

After showing that $\widetilde{\mathcal{T}}: [X/G] \to \mathcal{X}(\tau)$ is an isomorphism, our next step is to identify the irreducible component $\mathcal{X}(\tau) \subset \mathcal{C}^{\tau, \mathrm{BT}}$ by its profile index. We identify this profile index to be $\mathbb{Z}/f\mathbb{Z}$ by observing that $\mathcal{C}^{\tau, \mathrm{BT}}(\mathbb{Z}/f\mathbb{Z})$ is the only irreducible component containing a dense set of points with each Frobenius map of genre $\Ieta$ (see Lemma \ref{profile-of-I-eta-I-eta-genre}). Using (\ref{z-to-serre-weight}), we finally compute the Serre weight index of $\mathcal{Z}^{\tau}(\mathbb{Z}/f\mathbb{Z})$ which is the image of $\mathcal{C}^{\tau, \mathrm{BT}}(\mathbb{Z}/f\mathbb{Z})$ in $\mathcal{Z}$. 

If we allow $\eta'$-forms, we will need to change the definitions of first and second obstructions since they are presently tailored to work in the situation where each Frobenius matrix is in $\eta$-form. Furthermore, the definition of $\mathcal{T}$ (and therefore of $\tilde{\mathcal{T}}$) will have to be modified to allow for the image to have some Frobenius maps in $\eta'$-form. The image of $\widetilde{\mathcal{T}}$ will no longer be $\mathcal{C}^{\tau, \mathrm{BT}}(\mathbb{Z}/f\mathbb{Z})$. We will need to compute the correct profile index as a function of the indices $i \in \mathbb{Z}/f\mathbb{Z}$ for which we are allowing $\eta'$-form Frobenius matrices, and then compute the Serre weight index using the correct profile index.

Instead of directly replicating the structure of our proofs in the main body of the text, we will now evaluate the effect of allowing $\eta'$-form Frobenius matrices in a slightly non-linear fashion. We will first compute the profile $J$ needed such that $\mathcal{C}^{\tau, \mathrm{BT}}(J)$ contains a dense set of points with some Frobenius maps of genre $\Ieta$ and others of genre $\Ietaa$ as well as investigate the relationship of $J$ to Serre weights. Next, we will compute the altered conditions for first and second obstructions. Finally, we will show that although we could not include twists of trivial Serre weight in our main analysis, we can include them if we allow $\eta'$-form Frobenius matrices, and that this is the only extra advantage to be gained by allowing $\eta'$-form matrices.


To start, we introduce some notation: \\

We let $T \subset \mathbb{Z}/f\mathbb{Z}$ be the fixed set of indices $i$ such that the $i$-th Frobenius map is in $\eta$-form, while $T^c$ is the set of indices $i$ such that the $i$-th Frobenius map is in $\eta'$-form.

\begin{defn}
Let $i \in \mathbb{Z}/f\mathbb{Z}$. We say that $(i-1, i)$ is a transition if one of $\{i-1, i\}$ is in $T$ and the other in $T^c$.
\end{defn}

Given $\tau = \eta \oplus \eta'$ with $\eta \neq \eta'$, define $(\tilde{z}_i)_i$ via:
\begin{align}\label{define-tilde-z}
\tilde{z}_{i} = \begin{cases}
z_{i} &\text{ if } i-1 \in T, i \in T, \\
z_{i} + 1 &\text{ if } i-1 \in T, i \not\in T, \\
p - z_{i} &\text{ if } i-1 \not\in T, i \in T, \\
p-1-z_{i} &\text{ if } i-1 \not\in T, i \not\in T, \\
\end{cases}
\end{align}
where $z_i$ is defined in (\ref{defn-z_i}). As with $z_i$, we will take the indexing set of $\tilde{z}_i$ to be either $\mathbb{Z}/f\mathbb{Z}$ or $\mathbb{Z}$ depending on the situation.

\begin{remark}
By (\ref{define-tilde-z}), $\tilde{z}_{i} \neq 0$ whenever $(i-1, i)$ is a transition.
\end{remark}

\subsection{Profiles}
\begin{lemma}\label{relate-J-and-T}
Let $\tau$ be a tame principal series $\F$-type. Suppose $\mathcal{C}^{\tau, \mathrm{BT}}(J)$ is an irreducible component of $\mathcal{C}^{\tau, \mathrm{BT}}$ comprising a dense set of $\overline{\F}_p$-points corresponding to Breuil-Kisin modules that satisfy the following:
\begin{itemize}
\item The genre of the $i$-th Frobenius map is $\Ieta$ for $i \in T$.
\item The genre of the $i$-th Frobenius map is $\Ietaa$ for $i \not\in T$.
\end{itemize}
Then $J = T$.
\end{lemma}
\begin{proof}
 By the argument in the proof of Lemma \ref{profile-of-I-eta-I-eta-genre}, $\mathcal{C}^{\tau, \mathrm{BT}}(J)$ contains a dense constructible set of points such that if $i \in J$, then the upper left entry of $i$-th Frobenius is $0$ or $v$-divisible, making it necessarily of genre $\Ieta$ or $\II$. On the other hand, if $i \not \in J$, then the lower right entry of $i$-th Frobenius is either $0$ or $v$-divisible, making it necessarily of genre $\Ietaa$ or $\II$.
\end{proof}

\begin{lemma}
Let $\mathcal{C}^{\tau, \mathrm{BT}}(J)$ be as in Lemma \ref{relate-J-and-T}. Then $\mathcal{C}^{\tau, \mathrm{BT}}(J)$ is a cover of an irreducible component of $\mathcal{Z}$ if and only if $J \in \mathcal{P}_{\tau}$ if and only if 
for each $i$, $\tilde{z}_{i} \neq p$.
\end{lemma}
\begin{proof}
$\mathcal{C}^{\tau, \mathrm{BT}}(J)$ is a cover of an irreducible component of $\mathcal{Z}$ if and only if $J \in \mathcal{P}_{\tau}$ by \cite[Thm.~5.1.12]{CEGS-components}. By the definition of $\mathcal{P}_{\tau}$ (Definition \ref{defn-P-tau}) and the fact that $J=T$, the condition on 
$(\tilde{z}_{i})_i$ is immediate.
\end{proof}

Since the strategy of this paper rests on covering a suitable irreducible component of $\mathcal{Z}$ by the irreducible component of $\mathcal{C}^{\tau, \mathrm{BT}}$ in the image of $\mathcal{T}$, it is reasonable to impose the condition that for each $i$, $\tilde{z}_i \neq p$.

\begin{remark}\label{serre-weight-mixed-form}
Suppose $J=T$ as in Lemma \ref{relate-J-and-T} and $\tilde{z}_i \neq p$. Since $J \in \mathcal{P}_{\tau}$, we may compute the Serre weight corresponding to $J$. By (\ref{z-to-serre-weight}), the symmetric powers of the Serre weight are given by $b_i = p-1-\tilde{z}_i$.
\end{remark}


\subsection{First obstruction}
As in the greater part of Section \ref{classification subsec}, we will assume that all Breuil-Kisin modules in this section are regular (see Definition \ref{regular}). We will also assume that $\tilde{z}_i \neq p$.
\begin{defn}
Let $\M$ be a Breuil-Kisin module over an $\F$-algebra $R$ with Frobenius matrices $\{F_i\}_i$ written with respect to some inertial bases. We say that $\mathcal{G}(\M_i) = \mathcal{G}(F_i) = \text{\normalfont I}$ if $\mathcal{G}(\M_i) = \mathcal{G}(F_i) \in \{\Ieta, \Ietaa\}$. 
\end{defn}

\begin{lemma}\label{suped-bad-genre}
Let $R$ be an Artinian local ring over $\F$ with maximal ideal $\m$. A regular Breuil-Kisin module $\M$ defined over $R$ is of bad genre if and only if the following conditions are satisfied (assuming $\tilde{z}_i \neq p$ for all $i$):

\begin{enumerate}
    \item 
    If $(i-1,i)$ is not a transition, then $(\mathcal{G}(F_i), \tilde{z}_{i}) \in \{(\II, 0), (\II, p-1), (\text{\normalfont I}, 1), (\text{\normalfont I}, p-1)\}$.\\
    If $(i-1, i)$ is a transition, then $(\mathcal{G}(F_i), \tilde{z}_{i}) \in \{(\II, 1), (\text{\normalfont I}, 1), (\text{\normalfont I}, p-1)\}$.
    \item 
    If $(i-1, i)$ is not a transition and $(\mathcal{G}(F_i), \tilde{z}_{i})  = (\II, 0)$, then \\
    $(\mathcal{G}(F_{i+1}), \tilde{z}_{i+1}) = (\text{\normalfont I}, p-1)$, or $(\mathcal{G}(F_{i+1}), \tilde{z}_{i+1}) = (\II, p-1)$ with $(i, i+1)$ not a transition.
    \item 
    If $(i-1, i)$ is not a transition and $(\mathcal{G}(F_i), \tilde{z}_{i}) \in \{(\II, p-1), (\text{\normalfont I}, 1), (\text{\normalfont I}, p-1)\}$, then $(\mathcal{G}(F_{i+1}), \tilde{z}_{i+1}) = (\II, 0)$ with $(i, i+1)$ not a transition, or $(\mathcal{G}(F_{i+1}), \tilde{z}_{i+1}) = (\II, 1)$ with $(i, i+1)$ a transition, or $(\mathcal{G}(F_{i+1}), \tilde{z}_{i+1}) = (\text{\normalfont I}, 1)$.
    \item
    If $(i-1, i)$ is a transition and $(\mathcal{G}(F_i), \tilde{z}_{i}) \in \{(\II, 1), (\text{\normalfont I}, 1), (\text{\normalfont I}, p-1)\}$, then $(\mathcal{G}(F_{i+1}), \tilde{z}_{i+1}) = (\II, 0)$ with $(i, i+1)$ not a transition, or $(\mathcal{G}(F_{i+1}), \tilde{z}_{i+1}) = (\II, 1)$ with $(i, i+1)$ a transition, or $(\mathcal{G}(F_{i+1}), \tilde{z}_{i+1}) = (\text{\normalfont I}, 1)$.
\end{enumerate}
\end{lemma}

\begin{proof}
Suppose $i \in T$. We restate the conditions for bad genre by expressing the conditions from Definition \ref{define-modified-bad-genre} in terms of $\tilde{z}_i$:

\begin{enumerate}
    \item 
    If $i-1 \in T$, then $(\mathcal{G}(F_i), \tilde{z}_{i}) \in \{(\II, 0), (\II, p-1), (\text{\normalfont I}, 1), (\text{\normalfont I}, p-1)\}$.\\
    If $i-1 \not\in T$, then $(\mathcal{G}(F_i), \tilde{z}_{i}) \in \{(\II, 1), (\text{\normalfont I}, 1), (\text{\normalfont I}, p-1)\}$.
    \item 
    If $i-1 \in T$ and $(\mathcal{G}(F_i), \tilde{z}_{i})  = (\II, 0)$, then \\
    $(\mathcal{G}(F_{i+1}), \tilde{z}_{i+1}) = (\II, p-1)$ with $i+1 \in T$, or $(\mathcal{G}(F_{i+1}), \tilde{z}_{i+1}) = (\text{\normalfont I}, p-1)$.
    \item 
    If $i-1 \in T$ and $(\mathcal{G}(F_i), \tilde{z}_{i}) \in \{(\II, p-1), (\text{\normalfont I}, 1), (\text{\normalfont I}, p-1)\}$, then $(\mathcal{G}(F_{i+1}), \tilde{z}_{i+1}) = (\II, 0)$ with $i+1 \in T$ or $(\mathcal{G}(F_{i+1}), \tilde{z}_{i+1}) = (\II, 1)$ with $i+1 \not\in T$ or $(\mathcal{G}(F_{i+1}), \tilde{z}_{i+1}) = (\text{\normalfont I}, 1)$.
    \item
    If $i-1 \not\in T$ and $(\mathcal{G}(F_i), \tilde{z}_{i}) \in \{(\II, 1), (\text{\normalfont I}, 1), (\text{\normalfont I}, p-1)\}$, then $(\mathcal{G}(F_{i+1}), \tilde{z}_{i+1}) = (\II, 0)$ with $i+1 \in T$ or $(\mathcal{G}(F_{i+1}), \tilde{z}_{i+1}) = (\II, 1)$ with $i+1 \not\in T$ or $(\mathcal{G}(F_{i+1}), \tilde{z}_{i+1}) = (\text{\normalfont I}, 1)$.
\end{enumerate}

By symmetry, for $i \not\in T$, the conditions for bad genre are:
\begin{enumerate}
    \item 
    If $i-1 \not\in T$, then $(\mathcal{G}(F_i), \tilde{z}_{i}) \in \{(\II, 0), (\II, p-1), (\text{\normalfont I}, 1), (\text{\normalfont I}, p-1)\}$.\\
    If $i-1 \in T$, then $(\mathcal{G}(F_i), \tilde{z}_{i}) \in \{(\II, 1), (\text{\normalfont I}, 1), (\text{\normalfont I}, p-1)\}$.
    \item 
    If $i-1 \not\in T$ and $(\mathcal{G}(F_i), \tilde{z}_{i})  = (\II, 0)$, then \\
    $(\mathcal{G}(F_{i+1}), \tilde{z}_{i+1}) = (\II, p-1)$ with $i+1 \not\in T$, or $(\mathcal{G}(F_{i+1}), \tilde{z}_{i+1}) = (\text{\normalfont I}, p-1)$.
    \item 
    If $i-1 \not\in T$ and $(\mathcal{G}(F_i), \tilde{z}_{i}) \in \{(\II, p-1), (\text{\normalfont I}, 1), (\text{\normalfont I}, p-1)\}$, then $(\mathcal{G}(F_{i+1}), \tilde{z}_{i+1}) = (\II, 0)$ with $i+1 \not\in T$ or $(\mathcal{G}(F_{i+1}), \tilde{z}_{i+1}) = (\II, 1)$ with $i+1 \in T$ or $(\mathcal{G}(F_{i+1}), \tilde{z}_{i+1}) = (\text{\normalfont I}, 1)$.
    \item
    If $i-1 \in T$ and $(\mathcal{G}(F_i), \tilde{z}_{i}) \in \{(\II, 1), (\text{\normalfont I}, 1), (\text{\normalfont I}, p-1)\}$, then $(\mathcal{G}(F_{i+1}), \tilde{z}_{i+1}) = (\II, 0)$ with $i+1 \not\in T$ or $(\mathcal{G}(F_{i+1}), \tilde{z}_{i+1}) = (\II, 1)$ with $i+1 \in T$ or $(\mathcal{G}(F_{i+1}), \tilde{z}_{i+1}) = (\text{\normalfont I}, 1)$.
\end{enumerate}
Bringing the two sets of conditions together, the conditions for bad genre are as in the statement of the lemma.
\end{proof}
From Lemma \ref{suped-bad-genre}, it is immediate that the following is the appropriate generalization of the definition of first obstruction.
\begin{defn}\label{suped-first-obstruction}
We say that a tame prinicipal series $\F$-type $\tau$ faces the first obstruction if $(\tilde{z}_i)_{i\in \mathbb{Z}}$ is made up entirely of the building blocks $1$ and $(0, p-1)$.
\end{defn}
\subsection{Second obstruction}
To compute the right form of second obstruction conditions, we first state a version of Lemma \ref{Explicit-CDM-form-I-eta} for Frobenius matrices in $\eta'$-form.
\begin{lemma}\label{Explicit-CDM-form-I-eta'}
Let $R$ be an Artinian local ring over $\F$ with maximal ideal $\m$. Let $\M$ be a regular Breuil-Kisin module, not of bad genre. Suppose with respect to an inertial basis, $F_i$ has the form $$\begin{pmatrix}
a_i& u^{e - \gamma_{i}}b_i \\
u^{\gamma_{i}}c_i & v d_i
\end{pmatrix}$$ with $a_i, b_i, c_i, d_i \in R$. Let $$P^{(j)} = \lim\limits_{n \to \infty} P_{j+nf} = \begin{pmatrix}
q_j & u^{e - \gamma_{j}} r_j\\
u^{\gamma_{j}}s_j& t_j
\end{pmatrix}$$ denote the base change matrices described in the proof of Proposition \ref{convergence-CDM-matrix}. 
Let $F'_i = (P^{(i+1)})^{-1} F_i \varphi(P_i)$ be the matrix in \ref{intermediate-CDM-form}, and explicitly, let $$F'_i = \begin{pmatrix}
a'_i & b'_i u^{e - \gamma_{i}}\\
c'_i u^{\gamma_{i}}& v d'_i
\end{pmatrix}.$$ For any $\sigma \in R[\![v]\!]$, denote by $\overline{\sigma}$ the constant part of $\sigma$.

Then 
\begin{align*}
F'_i &= \begin{cases}
Ad \begin{pmatrix}
\frac{b_i + a_i \overline{r_{i-1}}}{b_i} & 0\\
0 & 1
\end{pmatrix} \left(\begin{pmatrix}
a_i & u^{e-\gamma_{i}} b_i\\
0 & (d_i - \frac{b_i}{a_i}c_i)v
\end{pmatrix}\right) &\text{ if } \mathcal{G}(F_i) = I_{\eta'}, z_{i} = p-1 ,\\
\hspace{1cm}\\
    \begin{pmatrix}
a_i & u^{e-\gamma_{i}} b_i \\
0 & (d_i - \frac{b_i}{a_i}c_i) v
\end{pmatrix} &\text{ if } \mathcal{G}(F_i) = I_{\eta'}, z_{i} \neq p-1 ,\\
\hspace{1cm}\\
 Ad \begin{pmatrix}
\frac{b_i + a_i \overline{r_{i-1}}}{b_i} & 0\\
0 & 1
\end{pmatrix} \left(\begin{pmatrix}
a_i & u^{e-\gamma_{i}}b_i\\
u^{\gamma_{i}} (c_i - \frac{a_i}{b_i} d_i)& 0
\end{pmatrix}\right) &\text{ if } \mathcal{G}(F_i) = \II, z_{i} = p-1 ,\\
\hspace{1cm}\\
 \begin{pmatrix}
a_i & u^{e-\gamma_{i}}b_i\\
u^{\gamma_{i}}(c_i - \frac{a_i}{b_i} d_i)& 0
\end{pmatrix} &\text{ if } \mathcal{G}(F_i) = \II, z_{i} \neq p-1,
\end{cases}
\end{align*}
where $Ad \: M \:(N)$ denotes the matrix $MNM^{-1}$.
\end{lemma}
\begin{proof}
By Lemma \ref{Explicit-CDM-form-I-eta} using symmetry.
\end{proof}

Analogous to Proposition \ref{unipotent-action-gives-CDM-form-I-eta}, we define a left action of lower unipotent matrices on $\eta'$-form via:
\begin{align}\label{lower-unipotent-action}
\begin{pmatrix}
1 & 0 \\
y & 1
\end{pmatrix} \star \begin{pmatrix}
 a_i& u^{e - \gamma_{i}}b_i \\
u^{\gamma_{i}}c_i & v d_i
\end{pmatrix} = \begin{pmatrix}
a_i & u^{e - \gamma_{i}}b_i \\
u^{\gamma_{i}}(c_i + ya_i) & v(d_i + yb_i)
\end{pmatrix}.
\end{align}

We will assume now that $\M$ is a regular Breuil-Kisin module with Frobenius matrices $\{F_i\}_i$ such that for $i \in T$, $F_i = \begin{psmallmatrix}
v a_i& u^{e - \gamma_{i}}b_i \\
u^{\gamma_{i}}c_i & d_i
\end{psmallmatrix}$ and for $i \not\in T$, $F_i = \begin{psmallmatrix}
a_i& u^{e - \gamma_{i}}b_i \\
u^{\gamma_{i}}c_i & v d_i
\end{psmallmatrix}$ with $a_i, b_i, c_i, d_i \in R$. Our objective is to find the minimal set of conditions on $z_{i}$ that prohibit unipotent action (upper or lower, depending on the form of $F_i$) from giving $F'_i$ ($F'_i$ are as defined in Lemmas \ref{Explicit-CDM-form-I-eta} and \ref{Explicit-CDM-form-I-eta'}). Evidently, left unipotent action on $F_i$ fails to give $F'_i$ if and only if one of the following is true:

\begin{itemize}
    \item $i \in T$, $z_{i} = 0$ and $s_{i-1} \not\equiv 0 \mod v$, or
    \item $i \not\in T$, $z_{i} = p-1$ and $r_{i-1} \not\equiv 0 \mod v$.
\end{itemize}

Recall that $P^{(i)} = \mathcal{B}(F_{i} \varphi(P^{(i-1)})) (\Delta^{i})^{-1} = \mathcal{B}(F_{i}) \mathcal{B}(M_{i} \varphi(P^{(i-1)})) (\Delta^{i})^{-1}$. Also by the explicit calculations in Lemma \ref{Explicit-CDM-form-I-eta}, $\mathcal{B}(F_i)$ is upper triangular if $i \in T$ and correspondingly, $\mathcal{B}(F_i)$ is lower triangular if $i \not\in T$.

We want to now ascertain criteria for when $s_{i-1} \not\equiv 0$ mod $v$. We have the following possibilities:
\begin{enumerate}
    \item If $i-1 \in T$, $\mathcal{B}(F_{i-1})$ is upper triangular. Therefore, $s_{i-1} \not\equiv 0$ if and only if $\mathcal{B}(M_{i-1} \varphi(P^{(i-2)}))$ is not upper triangular mod $u^eR[\![u]\!]$. By the calculations in Lemma 
    \ref{inductive-step}, this can happen only if one of the following statements holds:

    \begin{enumerate}
    \item $z_{i-1} = 1$ and $s_{i-2} \not\equiv 0$ mod $v$. In this situation, $s_{i-1}$ is a multiple of $s_{i-2}$ mod $v$.
    \item $z_{i-1} = p-1$ and $r_{i-2} \not\equiv 0$ mod $v$. In this situation, $s_{i-1}$ is a multiple of $r_{i-2}$ mod $v$. 
    \end{enumerate}

\item If $i-1 \not\in T$, $\mathcal{B}(F_{i-1})$ is lower triangular. In this case, if $\mathcal{G}(F_{i-1}) = \Ietaa$, $s_{i-1} \equiv C c_{i-1} \mod v$ where $C \in R^{*}$ . If  $\mathcal{G}(F_{i-1}) = \II$, $s_{i-1}$ is an $R$-linear combination of $r_{i-2}$ and $d_{i-1}$ mod $v$.

\end{enumerate}

Similarly, for the situation where $r_{i-1} \not\equiv 0$ mod $v$, we have the following possibilities:

\begin{enumerate}
    \item If $i-1 \not\in T$, $\mathcal{B}(F_{i-1})$ is lower triangular. Therefore, $r_{i-1} \not\equiv 0$ if and only if $\mathcal{B}(M_{i-1} \varphi(P^{(i-2)}))$ is not lower triangular mod $u^eR[\![u]\!]$. By the calculations in Lemma 
    \ref{inductive-step}, this can happen only if one of the following statements holds:

\begin{enumerate}
    \item $z_{i-1} = p-2$ and $r_{i-2} \not\equiv 0$ mod $v$. In this situation, $r_{i-1}$ is a multiple of $r_{i-2}$ mod $v$.
    \item $z_{i-1} = 0$ and $s_{i-2} \not\equiv 0$ mod $v$. In this situation, $r_{i-1}$ is a multiple of $s_{i-2}$ mod $v$.
\end{enumerate}

\item If $i-1 \in T$, $\mathcal{B}(F_{i-1})$ is upper triangular. In this case, if $\mathcal{G}(F_{i-1}) = \Ieta$, $r_{i-1} \equiv C b_{i} \mod v$ where $C \in R^{*}$ . If  $\mathcal{G}(F_{i-1}) = \II$, $r_{i-1}$ is an $R$-linear combination of $s_{i-2}$ and $a_{i-1}$ mod $v$.

\end{enumerate}

Suppose $s_{i-1} \not\equiv 0$ mod $v$. Then $z_{i}$ is preceded by some sequence $(z_{i-k-1}, ..., z_{i-1}) = (1, ..., 1)$ with $k\geq -1$ and such that $[i-k-2, i-2] \subset T$.  When $k=-1$, we mean that the sequence is empty. This sequence of $1$'s must be preceded by either of the following:
\begin{itemize}
    \item $z_{i-k-2} = p-1$ with $i-k-2, i-k-3 \in T$. This situation is enough to construct an example with $s_{i-1} \not\equiv 0$ as we saw while proving the minimality of the second obstruction conditions in the proof of Proposition \ref{unipotent-action-gives-CDM-form-I-eta}. In this case, $(\tilde{z}_{i-k-2}, \tilde{z}_{i-k-1}, ..., \tilde{z}_{i-1}) = (p-1, 1, ..., 1)$ and none of the pairs in $\{(i-k-3, i-k-2), (i-k-2, i-k-1), ..., (i-2, i-1)\}$ are transitions.
    \item $z_{i-k-2} = p-1$ with $i-k-2 \in T$, $i-k-3 \not\in T$ and $r_{i-k-3} \not\equiv 0$. This implies that $(\tilde{z}_{i-k-2}, \tilde{z}_{i-k-1}, ..., \tilde{z}_{i-1}) = (1, 1, ..., 1)$ and the sequence is preceded by another sequence that allows $r_{i-k-3} \not\equiv 0$. Moreover the pair $(i-k-3, i-k-2)$ is a transition but none of the pairs in $\{(i-k-2, i-k-1), ..., (i-2, i-1)\}$ are transitions.
\end{itemize}

By symmetry, similar conditions on $\tilde{z}_j$ exist when $r_{i-1} \not\equiv 0$ mod $v$. 

Combining the analyses for $s_{i-1}$ and $r_{i-1}$ together, we find that whenever there exists an $i$ such that $F'_i \neq U \star F_i$ for all possible choices of $U$ (where $U$ is upper unipotent if $F_i$ is in $\eta$-form and lower unipotent if $F_i$ is an $\eta'$-form), then $(\tilde{z}_j)_j$ must contain a contiguous subsequence of the form $(p-1, 1, ..., 1, 0)$ of length $\geq 2$. On the other hand, if such a contiguous subsequence exists, we can construct an example so that $F'_i \neq U \star F_i$ for some $i$, for any choice of $U$ (upper or lower unipotent depending on the form of $F_i$).

Thus, we generalize the definition of second obstruction as follows:

\begin{defn}\label{suped-second-obstruction}
We say that a tame principal series $\F$-type $\tau$ faces the second obstruction if $(\tilde{z}_i)_{i\in \mathbb{Z}}$ contains a contiguous subsequence $(p-1, 1, ..., 1, 0)$ of length $\geq 2$. 
\end{defn}


\subsection{Trivial Serre weight}\label{sec-trivial-serre}
The generalizations of the definitions of first and second obstructions (see Definitions \ref{suped-first-obstruction} and \ref{suped-second-obstruction}) are very similar to the original definitions of first and second obstructions (see Definitions \ref{defn-first-obstruction} and \ref{defn-second-obstruction}). Note that in the case where each Frobenius matrix is in $\eta$-form, $\tilde{z}_{i} = z_i$. By Remark \ref{serre-weight-mixed-form}, upon requiring $\tau$ to not face the first and second obstructions, we exclude no fewer irreducible components of $\mathcal{Z}$ than we had done earlier.

However, notice that the components of $\mathcal{Z}$ indexed by twists of the trivial Serre weight were also not covered under our strategy when we allowed only $\eta$-form Frobenius matrices, even though their exclusion did not arise from the first and second obstruction conditions. If $\mathcal{Z}(\sigma)$ is such a component, then by Proposition \ref{finding-tau-for-serre-weight}, the only possible tame principal series $\F$-type $\tau = \eta \oplus \eta'$ such that $\mathcal{C}^{\tau, \mathrm{BT}}(\mathbb{Z}/f\mathbb{Z})$ covers $\mathcal{Z}(\sigma)$ does not satisfy $\eta \neq \eta'$. This situation \textit{can} be rectified by allowing some Frobenius matrices to be in $\eta'$-form when $f\geq 2$. By the calculations in Remark \ref{serre-weight-mixed-form}, all we need is that each $\tilde{z}_i = p-1$, while not all $z_i$ equal $0$ (so that $\eta \neq \eta')$. 
For instance, we can choose $T = \mathbb{Z}/f\mathbb{Z} \smallsetminus \{0\}$, and choose $\tau$ so that $z_0 = p-2$, $z_1 = 1$ and all other $z_j$'s equal to $p-1$. A version of Proposition \ref{straightening} can be shown to hold for this situation when $p>3$ and we can find a similar result as in Theorem \ref{main-thm} when $p>3$, the Serre weight is trivial and $f\geq 2$. We omit the technical calculations from this paper because the trivial weight is in the Fontaine-Lafaille range and amenable to other methods.

\newpage
\bibliographystyle{amsalpha}
\bibliography{components}

\end{document}